\documentclass[11pt,letterpaper]{article}
\usepackage{amsfonts, amsmath, amssymb, amscd, amsthm, color, graphicx, mathabx, mathrsfs, wasysym, setspace, mdwlist, calc, float, mathtools, tikz-cd}
\usepackage{setspace}
\usepackage{soul}

\usepackage{enumitem}

\usepackage{tocloft}

\usepackage{xcolor}

\usepackage{hyperref}
\hypersetup{linktocpage}

\hypersetup{colorlinks,
    linkcolor={red!50!black},
    citecolor={blue!80!black},
    urlcolor={blue!80!black}}
\usepackage{float}

\numberwithin{equation}{section}

\renewcommand{\phi}{\varphi}
\newcommand{\WR}{\mathcal{WR}}

\renewcommand{\d }{{\rm d} }
\renewcommand{\dh}{\widehat \d}
\newcommand{\dl }{\dh_{H_i} }
\newcommand{\h}{\hookrightarrow_h}
\newcommand{\G }{\Gamma (G, \mathcal A)}
\newcommand{\C }{\mathcal C}

\newcommand{\Hl }{\{ H_i \} _{i\in I } }

\newcommand{\e }{\varepsilon }

\renewcommand{\P }{\mathcal P}
\renewcommand{\kappa }{\varkappa}

\newcommand{\lab}{{\bf Lab}}
\renewcommand{\ll }{\langle\hspace{-.7mm}\langle }
\newcommand{\rr }{\rangle\hspace{-.7mm}\rangle }
\renewcommand{\wr}{{\,\rm wr\,}}

\newcommand{\NN}{\mathbb N}
\newcommand{\Ast}{\mathop{\scalebox{1.5}{\raisebox{-0.2ex}{$\ast$}}}}

\DeclareMathOperator{\Ker}{Ker}

\newcommand{\M}{\mathcal M}
\newcommand{\Nn}{\mathcal N}
\newcommand{\ra}{\rightarrow}
\newcommand{\ca}{\curvearrowright}
\newcommand{\A}{\mathcal A}
\newcommand{\B}{\mathcal B}
\newcommand{\Q}{\mathcal Q}
\newcommand{\R}{\mathcal R}

\newcommand{\W}{\mathcal W}
\newcommand{\Ss}{\mathcal S}

\newcommand{\sU}{\mathscr U}
\newcommand{\sN}{\mathscr N}

\newcommand{\sR}{\mathscr R}
\newcommand{\sF}{\mathscr F}

\newcommand{\ZZ}{\mathbb Z}

\newtheorem{thm}{Theorem}[section]
\newtheorem*{thm*}{Theorem}
\newtheorem{cor}[thm]{Corollary}
\newtheorem{lem}[thm]{Lemma}
\newtheorem{claim}[thm]{Claim}
\newtheorem{prop}[thm]{Proposition}
\newtheorem{prob}[thm]{Problem}
\newtheorem{conj}[thm]{Conjecture}
\theoremstyle{definition}
\newtheorem{defn}[thm]{Definition}
\newtheorem{conv}[thm]{Convention}
\newtheorem{fact}[thm]{Fact}
\theoremstyle{remark}
\newtheorem{rem}[thm]{Remark}

\newtheorem{ex}[thm]{Example}

\let\OLDthebibliography\thebibliography
\renewcommand\thebibliography[1]{
  \OLDthebibliography{#1}
  \setlength{\parskip}{1.5pt}
  \setlength{\itemsep}{1.5pt plus 0.3ex}
}

\begin{document}

\title{\vspace*{-8mm}Wreath-like products of groups and their von Neumann algebras II: Outer automorphisms}

\author{Ionu\c t Chifan, Adrian Ioana, Denis Osin and Bin Sun}

\date{}

\maketitle

\vspace*{-7mm}

\begin{abstract}
Given a countable group $G$, let ${\rm L}(G)$ denote its von Neumann algebra. For a wide class of ICC groups with Kazhdan's property (T), we confirm  V.F.R. Jones' conjecture asserting that $Out(\text{L}(G))\cong Char (G)\rtimes Out(G)$. As an application, we show that, for every countable group $Q$, there exists an ICC group $G$ with property (T) such that $Out(\text{L}(G))\cong Q$. 
\end{abstract}

\tableofcontents

%%%%%%%%%%%%%%%%%%%%%%%%%%%%%%%%%%%%%%%%%%%%%%%%%%%%%%%%%%%%%%%%%%%%%%%%%%%%%%%%%%%%%%%%%%%%%%%%%%%%%%%%%%%%%%%
%%%%%%%%%%%%%%%%%%%%%%%%%%%%%%%%%%%%%%%%%%%%%%%%%%%%%%%%%%%%%%%%%%%%%%%%%%%%%%%%%%%%%%%%%%%%%%%%%%%%%%%%%%%%%%%

\section{Introduction}

%%%%%%%%%%%%%%%%%%%%%%%%%%%%%%%%%%%%%%%%%%%%%%%%%%%%%%%%%%%%%%%%%%%%%%%%%%%%%%%%%%%%%%%%%%%%%%%%%%%%%%%%%%%%%%%
%%%%%%%%%%%%%%%%%%%%%%%%%%%%%%%%%%%%%%%%%%%%%%%%%%%%%%%%%%%%%%%%%%%%%%%%%%%%%%%%%%%%%%%%%%%%%%%%%%%%%%%%%%%%%%%

A \emph{ von Neumann algebra} on a Hilbert space $\mathcal H$ is a $\ast$-subalgebra of the algebra of bounded linear operators $B(\mathcal H)$ that contains the identity operator and is closed in the weak (equivalently, strong) operator topology. Every countable discrete group $G$ gives rise to a von Neumann algebra ${\rm L}(G)$, which is defined as the weak operator closure of the span of the left regular representation of $G$ in  $B(\ell^2 (G))$. This construction was introduced by Murray and von Neumann in \cite{MvN36} and has served as an important source of examples ever since. 

General von Neumann algebras can be canonically decomposed into building blocks called factors \cite{vN49}. Recall that a von Neumann algebra $\M$ is a \textit{factor} if its \textit{center} 
$$
Z(\M)=\{ a\in M\mid ax=xa \; \text{ for all } x\in \M\}
$$
is trivial, i.e., $Z(\M)=\mathbb C \cdot 1$. The von Neumann algebra of a countable group $G$ is a factor if and only if all nontrivial conjugacy classes of $G$ are infinite (abbreviated \emph{ICC})  \cite{MvN43}; moreover, in this case L$(G)$ admits a finite positive trace, i.e., is a factor of type II$_1$. 

Understanding rigidity properties of group II$_1$ factors is one of the key challenges in the field of operator algebras. In this paper, we focus on computing the symmetries of von Neumann algebras of groups with Kazhdan's property (T). We begin by reviewing the basic definitions and some previously known results.

Let $\M$ be a $\ast$-algebra. An automorphism $\alpha$ of the underlying associative algebra over $\mathbb C$ is said to be an \emph{automorphism of $\M$} if it is $\ast$\textit{-preserving}, i.e., satisfies $\alpha (x^\ast)=\alpha(x)^\ast$ for all  $x\in\M$. If $\M$ is unital, automorphisms of the form $x\mapsto uxu^{-1}$, where $u$ is an invertible element of $\M$, are called \textit{inner}. Note that such a map is $\ast$-preserving if and only if $u^*u$ belongs to the center of $\M$.
Furthermore, if $\M$ is a von Neumann algebra (or, more generally, a C$^*$-algebra), the polar decomposition theorem and the existence of square roots of positive elements imply that every inner automorphism is induced by conjugation by a unitary element. 

As usual, the group of all (respectively, inner) automorphisms of $\M$ is denoted by $Aut(\M)$ (respectively, $Inn(\M)$), and the \emph{outer automorphism group} of $\M$ is defined to be the quotient $$Out(\M)=Aut(\M)/Inn(\M).$$

\smallskip

\begin{ex}\label{ex0}
The group ring $\mathbb CG$ of a group $G$ is a $\ast$-algebra,
%can be thought of as a $\ast$-algebra, 
where the operation $\ast$ is uniquely determined by the rule ${ g}^\ast=g^{-1}$, for all
$g\in G$. Let 
$$
\mathbb T=\{z\in \mathbb C\mid |z|=1\}\;\;\;\;\; {\rm and}\;\;\;\;\; Char(G)=Hom(G, \mathbb T).
$$
For any $\delta \in Aut(G)$ and $\rho\in Char(G)$, the map $g\mapsto \rho(g) \delta(g)$ uniquely extends to an automorphism of the $\ast$-algebra $\mathbb CG$. This yields a homomorphism $Char(G)\rtimes Out(G)\rightarrow Out(\mathbb CG)$. In case $G$ is torsion-free and satisfies the Kaplansky unit conjecture (i.e., the only units of $\mathbb CG$ are scalar multiples of elements of $G$), one can show this homomorphism is actually an isomorphism. The crucial point in proving surjectivity is the observation that every automorphism of $\mathbb C G$ preserves its set of units.
\end{ex}

\smallskip

Replacing $\mathbb CG$ with the von Neumann algebra $\text{L}(G)$ of $G$, we still get a well-defined homomorphism 
\begin{equation}\label{Jmap}
Char(G)\rtimes Out(G)\rightarrow Out(\text{L}(G)),
\end{equation}
which is neither injective, nor surjective in general.  If $G$ satisfies the ICC condition, the injectivity of this map is not difficult to prove (see Proposition \ref{inj} and Remark \ref{RemFin}). However, proving surjectivity would require additional assumptions, since the rigidity of $\mathbb CG$ caused by the Kaplansky unit conjecture has no analog in  ${\rm L}(G)$. Indeed, being a finite von Neumann algebra, 
${\rm L}(G)$ has a norm dense  set of invertible elements.
%Indeed, being a Banach algebra, ${\rm L}(G)$ has continuously many invertible elements.  

In fact, ${\rm L}(G)$ often possesses a huge group of outer automorphisms.  For instance, Connes' work \cite{Co76} implies that the von Neumann algebra of any non-trivial, ICC, amenable group $G$ is isomorphic to the Murray and von Neumann's hyperfinite II$_1$ factor $\mathcal R$, % \cite{Co76}, 
whose outer automorphism group contains every second countable locally compact group as a subgroup. In particular, $Out({\rm L}(G))$ has the cardinality of the continuum. 

In contrast, Connes showed that von Neumann algebras of groups with property (T) are much more rigid.

\begin{thm}[Connes, {\cite{Co80}}]\label{rigidOut}
For any ICC group $G$ with Kazhdan's property (T), $Out({\rm L}(G))$ is countable.
\end{thm}

The main goal of our paper is to address the following natural question prompted by Connes' theorem.

\begin{prob}\label{MainProb}
Which countable groups can be realized as outer automorphism groups of von Neumann algebras of ICC groups with property (T)?
\end{prob}

Theorem \ref{rigidOut} led Connes \cite{Co82} to formulate  his rigidity conjecture: if ${\rm L}(G)\cong{\rm L}(H)$, for ICC property (T) groups $G$ and $H$, then $G\cong H$. On a related note, Jones \cite[Problem 8]{Jo00} conjectured that property (T) endows ${\rm L}(G)$ with enough rigidity for the map (\ref{Jmap}) to be surjective. More precisely, he suggested the following as one of his millenium problems (see also \cite[Section 3]{Po06a}). 

\begin{conj}[Jones]\label{Jconj}
If $G$ is an ICC group with property (T), then  
\begin{equation}\label{Eq:JC}
Out(\emph{L}(G))\cong Char(G)\rtimes Out(G).
\end{equation}
\end{conj}

In the past two decades, remarkable progress in understanding 
automorphism
%symmetry 
groups of II$_1$ factors has been achieved via Popa's deformation/rigidity theory \cite{Po06a, Va10, Io18}. %For example, 
Ioana, Peterson, and Popa first showed that every compact abelian group arises as the outer automorphism group of a II$_1$ factor \cite{IPP05}. This result was generalized to all finitely presented groups by Popa and Vaes \cite{PV}, all countable groups by Vaes \cite{Va08} and all compact second countable groups by Falgui\`eres and Vaes \cite{FV} (see also \cite{Dep10}). These results were recovered and extended recently by Popa and Vaes, see \cite[Corollary 8.2]{PV22}.

Despite all of these advances, both Problem \ref{MainProb} and Jones' conjecture remained wide open. 
The II$_1$ factors with prescribed finitely presented outer automorphism groups found in \cite{PV} are actually von Neumann algebras of certain ICC groups such that the isomorphism~(\ref{Eq:JC}) holds. However, the algebraic structure of these groups is incompatible with property (T).  Moreover, \emph{no computation of $Out({\rm L}(G))$ for a non-trivial, ICC group $G$ with property (T) has been performed until now.} Even the existence of a single non-trivial ICC, property (T) group $G$ such that $Out({\rm L}(G))=\{ 1\}$ remained an open question.

In the present paper, we first show that a wide class of ICC groups with property (T) satisfy Jones' conjecture and then construct groups from this class that have prescribed outer automorphisms and no characters. This allows us to give a complete solution to Problem~\ref{MainProb} by proving the following.

\begin{thm}\label{MainThm}
For any countable group $Q$, there exists a non-trivial, ICC group $G$ with property (T) such that $Out({\rm L}(G))\cong Q$.
\end{thm}

Our approach utilizes the concept of a wreath-like product of groups introduced in our earlier paper \cite{CIOS1}. In addition to ideas elaborated in \cite{CIOS1}, the proof of Theorem~\ref{MainThm} requires further study of wreath-like products arising from the group-theoretic analog of Dehn filling in $3$-manifolds (see \cite{Osi07,DGO}). Our results in this direction seem to be of independent interest; they will be used in the forthcoming papers \cite{CIOS3, CIOS4} to study automorphisms of reduced group $C^\ast$-algebras and embeddings of group II$_1$ factors. For a more detailed discussion of the ideas and concepts involved in the proof of Theorem \ref{MainThm}, we refer to the next section.

To the best of our knowledge, Theorem \ref{MainThm} is new even without the property (T) assumption on the group $G$. While arbitrary countable groups $Q$ were shown to arise as outer automorphism groups of II$_1$ factors in \cite{Va08,PV22}, the II$_1$ factors considered therein do not arise from groups.
 
\paragraph{Acknowledgments.} I. Chifan was supported by the NSF grants DMS-1854194 and DMS-2154637. A. Ioana was supported by the NSF grants DMS-1854074 and DMS-2153805, and a Simons Fellowship. D. Osin was supported by the NSF grants DMS-1853989 and DMS-2405032. B. Sun received funding from the European Research Council (ERC) under the European Union’s Horizon 2020 research and innovation programme (Grant agreement No. 850930), and the Simons Foundation (Grant agreement No. IP00672308). We thank the anonymous referees for several comments which helped improve the exposition.

%%%%%%%%%%%%%%%%%%%%%%%%%%%%%%%%%%%%%%%%%%%%%%%%%%%%%%%%%%%%%%%%%%%%%%%%%%%%%%%%%%%%%%%%%%%%%%%%%%%%%%%%%%%%%%%
%%%%%%%%%%%%%%%%%%%%%%%%%%%%%%%%%%%%%%%%%%%%%%%%%%%%%%%%%%%%%%%%%%%%%%%%%%%%%%%%%%%%%%%%%%%%%%%%%%%%%%%%%%%%%%%

\section{Outline of the paper}

%%%%%%%%%%%%%%%%%%%%%%%%%%%%%%%%%%%%%%%%%%%%%%%%%%%%%%%%%%%%%%%%%%%%%%%%%%%%%%%%%%%%%%%%%%%%%%%%%%%%%%%%%%%%%%%
%%%%%%%%%%%%%%%%%%%%%%%%%%%%%%%%%%%%%%%%%%%%%%%%%%%%%%%%%%%%%%%%%%%%%%%%%%%%%%%%%%%%%%%%%%%%%%%%%%%%%%%%%%%%%%%

\subsection{Wreath-like products of groups and property (T)} 

We begin by recalling the following definition proposed in \cite{CIOS1}.

\begin{defn}\label{wlp}
Let $A$, $B$ be arbitrary groups, $I$ an abstract set, $B\curvearrowright I$ a (left) action of $B$ on $I$. A group $W$ is a \emph{wreath-like product} of groups $A$ and $B$ corresponding to the action $B\curvearrowright I$ if $W$
is an extension of the form
\begin{equation}\label{ext}
1\longrightarrow \bigoplus_{i\in I}A_i \longrightarrow  W \stackrel{\e}\longrightarrow B\longrightarrow 1,
\end{equation}
where $A_i\cong A$ and the action of $W$ on $A^{(I)}=\bigoplus_{i\in I}A_i$ by conjugation satisfies the rule
$$wA_iw^{-1} = A_{\e(w)i}\;\;\; \forall\, i\in I.$$  
We call $A$ the \textit{base} of a wreath-like product $W\in\W\R(A,B\curvearrowright I)$.

If the action $B\curvearrowright I$ is regular (i.e., free and transitive), we say that $W$ is a \emph{regular wreath-like product} of $A$ and $B$. The set of all wreath-like  products of groups $A$ and $B$ corresponding to an action $B\curvearrowright I$ (respectively, all regular wreath-like products) is denoted by $\WR(A, B\curvearrowright I)$ (respectively, $\WR(A,B)$). 
\end{defn}

\begin{ex}
Clearly, $A {\, \rm wr\,} B\in \WR(A,B)$ for any groups $A$ and $B$.
\end{ex} 

It is not difficult to show that $W\cong A{\, \rm wr\,} B$ in the setting of Definition \ref{wlp} whenever the extension (\ref{ext}) splits. One of the main findings of the paper \cite{CIOS1} is that a group theoretic version of Thurston's theory of hyperbolic Dehn filling in $3$-manifolds can be used to produce non-splitting wreath-like products of a completely different nature. The theorem below is a particular case of our construction. For a subset $S$ of a group $G$, we denote by $\ll S\rr$ the {\it normal closure} of $S$ in $G$, i.e., the smallest normal subgroup of $G$ containing $S$.

\begin{thm}[Chifan--Ioana--Osin--Sun, {\cite[Theorem 1.4]{CIOS1}}]\label{Thm:HypWR}
Let $H$ be a torsion-free hyperbolic group and let $h\in H$ be a non-trivial element. For any sufficiently large $k\in\mathbb N$, the group $H/\ll h^k\rr $ is  hyperbolic, ICC, and we have $$H/[\ll h^k\rr, \ll h^k\rr]\in \WR (\ZZ, H/\ll h^k\rr\curvearrowright I),$$ where the action $H/\ll h^k\rr\curvearrowright I$ is transitive with finite cyclic stabilizers. 
\end{thm}

Unfortunately, the construction based on Theorem \ref{Thm:HypWR} is not flexible enough to prove our main result, which requires constructing wreath-like products with prescribed outer automorphisms and certain additional properties (see Theorem \ref{Thm:Out} below). In fact, the approach suggested in \cite{CIOS1} is deficient in a fundamental way: although it allowed us to construct continuously many $W^\ast$-superrigid wreath-like products, their outer automorphism groups fall into countably many isomorphism classes (for the ultimate manifestation of this phenomenon, see Corollary \ref{Out} below).

To address this problem, we first suggest a more general approach to constructing wreath-like products based on the notions of a {\it Cohen-Lyndon triple} and {\it Cohen-Lyndon subgroup} introduced in Section~\ref{Sec:WRDF} and Section~\ref{Sec:CL}, respectively. 
Unlike examples provided by Theorem \ref{Thm:HypWR}, wreath-like products obtained via our new approach are regular and can have non-abelian bases. Furthermore, regularity allows us to consider additional factorizations, thus resulting in a very flexible construction leading to the proof of Theorem \ref{Thm:Out} (and eventually Theorem \ref{MainThm}). 

Our new approach to constructing wreath-like products appears to be of independent interest, and we discuss it in more detail here. We first consider a simple example illustrating the main idea.

\begin{ex}\label{ExCL}
Let $A$ and $B$ be arbitrary groups. It is well-known and easy to prove that the normal subgroup $\ll A\rr $ of $A\ast B$ is isomorphic to the free product of conjugates of $A$. Let $C$ denote the kernel of the natural homomorphism from $\ll A\rr$ to the direct sum of the same conjugates of $A$. It is not difficult to show that $C$ 
is normal in $A\ast B$ and $$(A\ast B)/C\cong A{\,\rm wr\,} B \in \WR(A,B).$$
\end{ex}

To generalize this example, we introduce the following.

\begin{defn}\label{Defn:CL}
Let $G$ be a group, $H$ a subgroup of $G$. Recall that a \emph{left transversal} of $H$ in $G$ is a subset of $G$ that contains exactly one representative from each left coset of $H$. We say that $H$ is a \emph{Cohen-Lyndon subgroup} of $G$ if there exists a left transversal $T$ of $H$ in $G$ such that $\ll H \rr=\Ast_{t\in T}tHt^{-1}$.
\end{defn}

Cohen and Lyndon \cite{CL} showed that every maximal cyclic subgroup of a free group satisfies this property, hence the name. Further, it is easy to see that every free factor is a Cohen-Lyndon subgroup. 

The relevance of Definition \ref{Defn:CL} in the context of wreath-like products is elucidated by the following result generalizing Example \ref{ExCL} (see Corollary \ref{Cor:CLWR}).

\begin{prop}
For any group $G$ and any Cohen-Lyndon subgroup $H\le G$, the subgroup generated by the set
$$
S=\{ [h_1^{g_1}, h_2^{g_2}] \mid g_1, g_2\in G,\; g_1\ll H\rr \ne g_2\ll H\rr\}.
$$ 
is normal in $G$ and we have $G/\langle S\rangle  \in \WR (H, G/\ll H\rr)$. 
\end{prop}

In Section \ref{Sec:CL}, we show how to construct Cohen-Lyndon subgroups of acylindrically hyperbolic groups using a generalization of relative hyperbolicity and group theoretic Dehn filling technique developed in \cite{DGO}. As we already mentioned, this construction provides us with enough flexibility to prove Theorem \ref{Thm:Out} (see below), which is one of the two major ingredients in the proof of our main result. 

As a byproduct, we also obtain some results of independent interest; for example, we prove the following.

\begin{thm}\label{Thm:WLP_hyp}
Let $G$ be a non-elementary hyperbolic group. For any finitely generated group $A$, there exists a quotient group $W$ of $G$ such that $W\in \WR(A,B)$, where $B$ is non-elementary hyperbolic.
\end{thm}

It is worth noting that the ordinary wreath product $A\, {\rm wr}\, B$ of a non-trivial group $A$ and an infinite group $B$ never has property (T) (see, for example, \cite[Theorem 2.8.2]{HV}). In contrast, Theorem \ref{Thm:WLP_hyp} applied to a uniform lattice in $Sp(n,1)$ provides us with a rich source of regular wreath-like products with property (T) (and even a stronger fixed-point property, see Corollary \ref{Cor:FLp}). % Theorem  \ref{Thm:AHQ}). 
The mere existence of such examples is a remarkable phenomenon with numerous non-trivial implications. Here we mention just three of them.

In \cite[Section 6.6]{Po05}, Popa asked whether $\text{H}^2(\alpha)=\text{H}^2(B)$, for Bernoulli actions $\alpha$ of property (T) groups $B$. Here, $\text{H}^2(B)$ and $\text{H}^2(\alpha)$ denote the  second cohomology groups of $B$ and $\alpha$ with values in $\mathbb T$ (see \cite{FM77a,Ji15}). This question was answered negatively in  \cite{Ji15} using Popa's cocycle superrigidity theorem \cite{Po05}.  The existence of property (T) groups $G\in \WR(A,B)$, where $A$ is nontrivial abelian,  allows us to give a simpler solution (see Proposition \ref{H^2} for details).

Further, Theorem \ref{Thm:WLP_hyp} was recently used in \cite{CDI22} to prove that {\it any separable II$_1$ factor embeds into a II$_1$ factor with property (T)}. This result can be thought as a von Neumann algebraic counterpart of the well-known fact that every countable group embeds into a countable group with property (T) (see \cite{De96,Ols95}). The proof relies on the existence of property (T) groups $G\in\W\R(A,B)$, where $A$ is the free group on two generators. 

Finally, the existence of wreath-like products with property (T) and non-commutative base plays a crucial role in our study of $C^\ast$-superrigidity and outer automorphism groups of reduced group $C^\ast$-algebras in \cite{CIOS4}. 

\subsection{Rigidity and outer automorphisms of von Neumann algebras} 

Connes' rigidity conjecture is equivalent to asking that any ICC group $G$ with property (T) is W$^*$-superrigid: for any group $H$, the isomorphism ${\rm L}(G)\cong {\rm L}(H)$ implies that $G\cong H$ \cite{Co82}. The first examples of groups satisfying Connes' conjecture were constructed in our earlier paper \cite{CIOS1}. In \cite{Po06a}, Popa proposed a generalization of Connes' conjecture that also incorporates Jones' formula (\ref{Eq:JC}). Informally, it can be stated as follows. 
\begin{conj}[Popa]
For any ICC group $G$ with property (T) and any group $H$, every isomorphism ${\rm L}(G)\cong {\rm L}(H)$ is induced by a character $G\to \mathbb T$ and an a group isomorphism $G\ra H$, up to conjugation by a unitary element.
\end{conj}
Our proof of Theorem \ref{MainThm} relies on a verification of Popa's conjecture within a wide class of wreath-like products. To state this result, we need to recall some definitions. Let $\mathcal M$ be a II$_1$ factor. The \textit{amplification} of $\mathcal M$ by some $t>0$ is the isomorphism class of the II$_1$ factor $p\mathbb M_n(\mathcal M)p$, where $n\geq t$ and $p\in\mathbb M_n(\mathcal M)$ is a projection of normalized trace $t/n$.

\begin{thm}\label{symmetries'}
Let $A$, $C$ be non-trivial groups that are either abelian or ICC and  $B$, $D$ non-parabolic ICC subgroups of groups that are hyperbolic relative to a finite family of finitely generated, residually finite groups. Let also $B \ca I$ and $D\ca J$ be faithful actions with infinite orbits. 

Suppose that $G\in\mathcal W\mathcal R(A,B \ca I)$ and $H\in\mathcal W\mathcal R(C, D\ca J)$ are groups with property (T) and $\theta:{\rm L}(G)^t\rightarrow {\rm L}(H)$, $t>0$, is a ($*$-preserving) isomorphism.  Then $t=1$ and there exist a group isomorphism $\delta:G\rightarrow H$,  a character $\rho:G\rightarrow\mathbb T$, and a unitary element $u\in {\rm L}(H)$ such that $\theta(u_g)=\rho(g)uv_{\delta(g)}u^*$ for every $g\in G$.

Here $(u_g)_{g\in G}$ and $(v_h)_{h\in H}$ denote  the canonical generating unitaries of ${\rm L}(G)$ and ${\rm L}(H)$, respectively. 
\end{thm}

In particular, Theorem \ref{symmetries'} applies to wreath-like products from Theorem \ref{Thm:WLP_hyp} whenever $A$ is abelian or ICC. 

The proof of Theorem \ref{symmetries'} is presented in Section \ref{OUT}. It relies on deformation/rigidity methods (see \cite{Po01b,Po04,Po05,IPP05,IPV10,PV12}) and uses the interaction between property (T) and certain properties that wreath-like products share with ordinary wreath products. Note, however, that ordinary wreath products do not satisfy the conclusion of Theorem \ref{symmetries'} despite having remarkably rigid von Neumann algebras (see \cite{Po03,Po04,Po05,Po06b,CI08,Io10,IPV10,IM19}). Indeed, we have $\text{L}(A\wr B)\cong \text{L}(A'\wr B)$, for any abelian groups $A,A'$ with $|A|=|A'|$ and any group $B$.
Theorem \ref{symmetries'} immediately implies the following. 

\begin{cor}\label{Cor:JC}
Under the assumptions of Theorem \ref{symmetries'}, we have $Out(\emph{L}(G))\cong Char(G)\rtimes Out(G)$  and $\mathcal F({\rm L}(G))=\{ 1\}$.
\end{cor}

Here, 
$\mathcal F(\mathcal M)=\{t>0\mid \mathcal M^t\cong \mathcal M\}$
denotes the {\it fundamental group} of a II$_1$ factor $\mathcal M$.

Thus,  Conjecture \ref{Jconj} holds for any property (T) group $G\in\W\R(A,B\ca I)$  as in the statement of Theorem \ref{symmetries'}. Corollary \ref{Cor:JC} also makes additional progress on the problem of calculating the fundamental group of $\text{L}(G)$ for ICC, property (T) groups $G$ (see \cite[Problem 2, page 551]{Co94}) by providing new evidence supporting a conjecture of Popa \cite{Po06a} that this group should be trivial. For other results in this direction, see \cite{CDHK20,CIOS1}.

A result similar to Theorem \ref{symmetries'} was obtained in \cite[Theorem 1.3]{CIOS1}. Note, however, that Theorem \ref{symmetries'} applies to a significantly wider class of wreath-like products than the one considered in \cite{CIOS1}.  This generalization provides us with enough flexibility to construct wreath-like products with property (T) that satisfy the assumptions of Theorem \ref{symmetries'}, have prescribed outer automorphism groups, and admit no non-trivial characters. Namely, in Section \ref{Sec:App}, we prove the following.

\begin{thm}\label{Thm:Out}
For any countable group $Q$, there exist a countable group $B$, a countable set $I$,  and $2^{\aleph_0}$ pairwise non-isomorphic finitely generated groups $\{U_j\}_{j\in J}$ such that the following conditions hold.
\begin{enumerate}
\item[(a)] For any $j\in J$,  $U_j\in \WR(A_j,B\curvearrowright I)$, where $A_j$ is non-trivial abelian and $B\curvearrowright I$ is a faithful action with infinite orbits.

\item[(b)] $B$ is a  non-parabolic ICC subgroup of a finitely generated, relatively hyperbolic group with residually finite peripheral subgroups.

\item[(c)] For any $j\in J$, $U_j$ has property (T), $[U_j,U_j]=U_j$, and $Out(U_j)\cong Q$.
\end{enumerate}
\end{thm}

The proof of Theorem \ref{Thm:Out} relies on the results about wreath-like products obtained in Section \ref{Sec:WR} and on a number of other non-trivial facts and group theoretic techniques, including small cancellation theory in relatively hyperbolic groups \cite{Osi10} and the work Agol, Haglund, and Wise on groups acting on $CAT(0)$ cube complexes \cite{A,HW,Wis}.

Recall that two II$_1$ factors $\mathcal M$ and $\mathcal N$ are said to be {\it stably isomorphic} if $\mathcal M^t\cong \mathcal N$, for some $t>0$. Combining Theorems \ref{symmetries'} and \ref{Thm:Out}, we obtain the following strong version of Theorem~\ref{MainThm}. 

\begin{cor}\label{Out}
Let $Q$ be any countable group. Then there is a continuum of ICC property (T) groups $(G_i)_{i\in I}$ such that the \emph{II}$_1$ factors $(\emph{L}(G_i))_{i\in I}$ are pairwise not stably isomorphic and satisfy $Out(\emph{L}(G_i))\cong Q$ and $\mathcal F(\emph{L}(G_i))=\{1\}$, for every $i\in I$.
\end{cor}

\begin{rem}
The fact that every countable group realizes as $Out(G)$ for some property (T) group was known before, see \cite{Min} and \cite{CIOS2}. However, groups $G$ constructed in these papers do not have have the required wreath-like structure and, therefore, Theorem~\ref{symmetries'} does not apply to them.    
\end{rem}

%%%%%%%%%%%%%%%%%%%%%%%%%%%%%%%%%%%%%%%%%%%%%%%%%%%%%%%%%%%%%%%%%%%%%%%%%%%%%%%%%%%%%%%%%%%%%%%%%%%%%%%%%%%%%%%
%%%%%%%%%%%%%%%%%%%%%%%%%%%%%%%%%%%%%%%%%%%%%%%%%%%%%%%%%%%%%%%%%%%%%%%%%%%%%%%%%%%%%%%%%%%%%%%%%%%%%%%%%%%%%%%

\section{Group theoretic preliminaries}\label{Sec:GTPrelim}

%%%%%%%%%%%%%%%%%%%%%%%%%%%%%%%%%%%%%%%%%%%%%%%%%%%%%%%%%%%%%%%%%%%%%%%%%%%%%%%%%%%%%%%%%%%%%%%%%%%%%%%%%%%%%%%
%%%%%%%%%%%%%%%%%%%%%%%%%%%%%%%%%%%%%%%%%%%%%%%%%%%%%%%%%%%%%%%%%%%%%%%%%%%%%%%%%%%%%%%%%%%%%%%%%%%%%%%%%%%%%%%

%%%%%%%%%%%%%%%%%%%%%%%%%%%%%%%%%%%%%%%%%%%%%%%%%%%%%%%%%%%%%%%%%%%%%%%%%%%%%%%%%%%%%%%%%%%%%%%%%%%%%%%%%%%%%%%

\subsection{Hyperbolic groups and their generalizations} \label{Sec:Hyp}

%%%%%%%%%%%%%%%%%%%%%%%%%%%%%%%%%%%%%%%%%%%%%%%%%%%%%%%%%%%%%%%%%%%%%%%%%%%%%%%%%%%%%%%%%%%%%%%%%%%%%%%%%%%%%%%

The proofs of the main results of our paper make use of relatively hyperbolic groups and their further generalization proposed in \cite{DGO}. In this context, it is convenient to work with generating alphabets instead of generating sets of groups. We begin by reviewing the necessary definitions. 

By a \emph{generating alphabet} $\mathcal A$ of a group $G$ we mean an abstract set given together with a (not necessarily injective) map $\mathcal A\to G$ whose image generates $G$. Note that every generating set $X$ of $G$ can be thought of as a generating alphabet with the obvious inclusion map $X\to G$. By $\mathcal A^\ast$ we denote the free monoid on $\mathcal A$. For a word $w\in \mathcal A^*$, let $\|w\|$ denote its length. Given a word $a_1\ldots a_k\in \mathcal A^*$, we say that it \emph{represents} an element $g\in G$ if $g=a_1\cdots a_k$ in $G$. If no confusion is possible, we identify elements of $\mathcal A^*$ with elements of $G$ represented by them.

By the \emph{Cayley graph} of $G$ with respect to a generating alphabet $\mathcal A$, denoted $\Gamma (G, \mathcal A)$, we mean a graph with the vertex set $G$ and the set of edges defined as follows. For every $a\in \mathcal A$ and every $g\in G$, there is an oriented edge $e$ going from $g$ to $ga$ in $\Gamma (G, \mathcal A)$ and labelled by $a$. We denote by $e^{-1}$ the opposite edge going from $ga$ to $g$ and labelled by the letter $a^{-1}$ (where $a^{-1}$ is the letter of $\mathcal A$ chosen according to the convention discussed above).

Given a combinatorial path $p$ in $\Gamma (G, \mathcal A)$, we denote by $\ell (p)$ its length, by $\lab (p)$ its label defined in the usual way, and by $p^{-1}$ the combinatorial inverse of $p$. We use the notation $\d_{\mathcal A}$ and $|\cdot|_{\mathcal A}$ to denote the standard metric on $\Gamma (G,\mathcal A)$ and the length function on $G$ with respect to (the image of) $\mathcal A$.

Recall that a metric space $S$ with a distance function $\d$ is said to be \emph{geodesic}, if every two points $a,b\in S$ can be connected by a path $p$ of length $\d(a,b)$. A geodesic metric space $S$ is said to be  \emph{$\delta$-hyperbolic} if there exists a constant $\delta\ge 0$ such that for any geodesic triangle $\Delta $ in $S$, every side of $\Delta $ is contained in the union of the closed
$\delta$-neighborhoods of the other two sides \cite{Gro}.

A group $G$ is {\it hyperbolic} if it is generated by a finite set $X$ and its Cayley graph $\Gamma (G,X)$ is a hyperbolic metric space. This definition is independent of the choice of a particular finite generating set $X$. Here and below, we think of graphs as metric spaces with respect to the natural length metric induced by identifying all open edges with the interval $(0,1)$. A hyperbolic group is called \emph{elementary} if it contains a cyclic subgroup of finite index. For examples and basic properties of hyperbolic groups, we refer the reader to \cite{Gro} and Chapters III.H,  III.$\Gamma$ of \cite{BH}.

Further, Let $G$ be a group, $\Hl$ a collection of subgroups of $G$, $X\subseteq G$. If $X$ and the union of all $H_i$ together generate $G$, we say that $X$ is a \emph{relative generating set} of $G$ with respect to $\Hl$. We think of $X$ and subgroups $H_i$ as abstract sets and consider the disjoint unions
\begin{equation}\label{calA}
\mathcal H = \bigsqcup\limits_{i\in I} H_i\;\;\;\;\; {\rm and}\;\;\;\;\; \mathcal A= X \sqcup \mathcal H.
\end{equation}

\begin{conv}
Henceforth, we assume that all generating sets and relative generating sets are symmetric, i.e., closed under inversion. To take care of the possible ambiguity arising from the fact that distinct letters of $\mathcal A$ may represent the same element of $G$, we agree to think of $a^{-1}$ as a letter from $X$ (respectively, from $H_i$) whenever $a\in X$ (respectively, $a\in H_i$).
\end{conv}

In particular, the alphabet $\mathcal A$ defined in (\ref{calA}) is also symmetric since so are $X$ and each $H_i$. Thus every element of $G$ can be represented by a word from $\mathcal A^*$.

In these settings, we can think of the Cayley graphs $\Gamma (H_i, H_i)$ (for each $i\in I$) as a subgraphs of $\G$. For every $i\in I$, we introduce a (generalized) metric $$\dl \colon H_i \times H_i \to [0, +\infty]$$ as follows.

\begin{defn}
Given $g,h\in H_i$,  $\dl (g,h)$ is defined to be the length of a shortest path in $\G $ that connects $g$ to $h$ and contains no edges of $\Gamma (H_i, H_i)$. If no such path exists, we set $\dl (g,h)=\infty $.
\end{defn}

Clearly $\dl $ satisfies the triangle inequality, where addition is extended to $[0, +\infty]$ in the natural way. We are now ready to define the notion of a hyperbolically embedded collection of subgroups introduced in \cite{DGO}. For more detail, we refer the reader to \cite{DGO,Osi16}.

\begin{defn}\label{hedefn}
A collection of subgroups $\Hl$ of $G$ is \emph{hyperbolically embedded  in $G$ with respect to a subset $X\subseteq G$}, denoted $\Hl \h (G,X)$, if the group $G$ is generated by the alphabet $\mathcal A$ defined by (\ref{calA}) and the following conditions hold.
\begin{enumerate}
\item[(a)] The Cayley graph $\G $ is hyperbolic.
\item[(b)] For every $n\in \mathbb N$ and $i\in I$, the set $\{ h\in H_i\mid \dl(1,h)\le n\}$ is finite.
\end{enumerate}
Further we say  that $\Hl$ is \emph{hyperbolically embedded} in $G$ and write $\Hl\h G$ if $\Hl\h (G,X)$ for some $X\subseteq G$. In case of a collection consisting of a single subgroup $H$, we simply write $H\h G$ and $H\h (G,X)$ to mean $\{H\}\h G$ and $\{H\}\h (G,X)$.
\end{defn}

\begin{rem}\label{Rem:he}
If $\Hl\h G$, then $H_i\h G$ for every $i\in I$, but the converse does not hold. For details, see \cite[Remark 4.26]{DGO}.
\end{rem}

\begin{prop}[{\cite[Proposition 4.28]{DGO}}]\label{rhhe}
A group $G$ is hyperbolic relative to a finite collection of subgroups $\Hl$ if and only if $\Hl\h(G, X)$ for some finite $X\subseteq G$.
\end{prop}

Readers unfamiliar with relative hyperbolicity can regard this proposition as the definition of a relatively hyperbolic group. For alternative definitions in the case of countable groups, see the survey \cite{Hru}. 

We illustrate the definitions discussed above by two examples. 

\begin{ex}\label{Ex:RH}
\begin{enumerate}
    \item[(a)] Let $G=H_1\ast H_2$, where $H_1$, $H_2$ are arbitrary groups. It is straightforward to verify that the corresponding functions $\d_i$ ($i=1,2$) satisfy $\d_i(a,b) =\infty$ for any distinct $a,b\in H_i$. Therefore, $\{H_1,H_2\} \h (G, \emptyset)$ and $G$ is hyperbolic relative to $H_1$ and $H_2$. The latter fact also follows immediately from the Bowditch definition of relative hyperbolicity as stated in \cite[Definition~3.4]{Hru} (applied to the action on the Bass-Serre tree associated with the free product decomposition) or the definition in terms of relative Dehn functions given in \cite[Definition 1.6]{Osi06}.
    
   \item[(b)] Let $G=H_1\times H_2$ and $X=\emptyset$. It is easy to see that the corresponding functions $\d_i$ ($i=1,2$) satisfy $\d_i(a,b) \le 3$ for all $a,b\in H_i$. In particular, $\{ H_1, H_2\}\not\h (G, \emptyset)$ whenever at least one of the subgroups is infinite. 
\end{enumerate}
\end{ex}

Subgroups from the collection $\Hl$ are called \emph{peripheral subgroups}. We say that $G$ is \emph{properly hyperbolic} relative to $\Hl$ if $H_i\ne G$ for all $i\in I$. A relatively hyperbolic group is said to be \emph{elementary} if it is virtually cyclic or one of the peripheral subgroups coincides with $G$. Thus, being non-elementary relatively hyperbolic means being properly relatively hyperbolic and not virtually cyclic.

The idea behind Example \ref{Ex:RH} (b) can be used to prove the following.

\begin{prop}[{\cite[Proposition 4.33]{DGO}}]\label{Prop:maln}
Let $G$ be a group, $\Hl $ a hyperbolically embedded collection of subgroups of $G$. Then for every $i\in I$ and $g\in G\setminus H_i$, we have $|H_i \cap g^{-1}H_i g|<\infty $. Also, if $i\ne j$, then $|H_i \cap g^{-1}H_jg|<\infty$ for all $g\in G$.
\end{prop}

The next result will be used to modify peripheral collections of relatively hyperbolic groups.

\begin{prop}[{\cite[Corollary 1.14]{DS}}]\label{Prop:rhrh}
Let $G$ be a group hyperbolic relative to a finite collection of subgroups $\Hl$. Suppose that each $H_i$ is hyperbolic relative to a finite collection of subgroups $\{ K_{ij}\}_{j\in J_i}$. Then $G$ is hyperbolic relative to $\bigcup_{i\in I} \{ K_{ij}\}_{j\in J_i}$.
\end{prop}

Every hyperbolic group is hyperbolic relative to the empty collection of subgroups and every group is hyperbolic relative to itself. Thus, we obtain the following.

\begin{cor}\label{Cor:rhh}
Let $G$ be a group hyperbolic relative to a collection of subgroups $\Hl \cup \{ K_j\}_{j\in J}$. If $I$ is finite and $H_i$ is hyperbolic for every $i\in I$, then $G$ is hyperbolic relative to $\{ K_j\}_{j\in J}$. In particular, a group hyperbolic relative to a finite collection of hyperbolic subgroups is itself hyperbolic.
\end{cor}

Finally, we discuss another generalization of hyperbolic groups -- the class of acylindrically hyperbolic groups -- that was introduced in \cite{Osi16} and received considerable attention in recent years (see \cite{Osi18} and references therein).

An isometric action of a group $G$ on a metric space $S$ is said to be {\it acylindrical} if, for every $\e>0$, there exist $R,N>0$ such that, for every two points $x,y\in S$ with $\d (x,y)\ge R$, there are at most $N$ elements $g\in G$ satisfying the inequalities
$$
\d(x,gx)\le \e \;\;\; {\rm and}\;\;\; \d(y,gy) \le \e.
$$
Recall that an action of a group $G$ on a hyperbolic space $S$ is \emph{non-elementary} if the limit set of $G$ on the Gromov boundary $\partial S$ has infinitely many points. An acylindrical action of a group $G$ on a hyperbolic space is non-elementary if and only if $G$ is not virtually cyclic and has unbounded orbits. Indeed, this follows from the classification of acylindrical actions obtained in \cite[Theorem 1.1]{Osi16}.

\begin{thm}\label{Thm:class}
Let $G$ be a group acting acylindrically on a hyperbolic space. Then $G$ satisfies exactly one of the following three conditions.
\begin{enumerate}
\item[(a)] $G$ has bounded orbits.
\item[(b)] $G$ has unbounded orbits and is virtually cyclic.
\item[(c)] The action of $G$ is non-elementary.
\end{enumerate}
\end{thm}

Every group has an acylindrical action on a hyperbolic space, namely the trivial action on the point. For this reason, we want to avoid elementary actions in the definition below.

\begin{defn}\label{ahdef}
A group $G$ is \emph{acylindrically hyperbolic} if admits a non-elementary acylindrical action on a hyperbolic space.
\end{defn}

We mention an equivalent characterization.

\begin{thm}[{\cite[Theorem 1.2]{Osi16}}]\label{Thm:ah}
A group $G$ is acylindrically hyperbolic if and only if $G$ contains a proper infinite hyperbolically embedded subgroup.
\end{thm}

Non-elementary hyperbolic and relatively hyperbolic groups are acylindrically hyperbolic. More precisely, we have the following (see \cite[Proposition 5.2 and Lemma 5.12]{Osi16} or \cite[Theorem 2.16]{CIOS2}).

\begin{thm}\label{Thm:RHAH}
Let $G$ be a group hyperbolic relative to a finite collection of subgroups $\Hl$, and $X$ a finite relative generating set of $G$ with respect to $\Hl$. Let also $\mathcal H$ and $\mathcal A$ be the alphabets defined by (\ref{calA}).
\begin{enumerate}
\item[(a)] The action of $G$ on $\G$ is acylindrical.
\item[(b)] If there is $i\in I$ such that $H_i$ is infinite and $H_i\ne G$, then the action of $G$ on $\G$ is non-elementary.
\end{enumerate}
\end{thm}

The class of acylindrically hyperbolic groups also includes mapping class groups of closed surfaces of non-zero genus \cite{Bow}, ${\rm Out}(F_n)$ for $n\ge 2$ \cite{BF}, groups of deficiency at least~$2$ \cite{Osi15} (see also the correction in \cite{MO19}), fundamental groups of most closed $3$-manifols \cite{MO15,MO19}, automorphism groups of non-elementary hyperbolic groups \cite{Gen, GH}, and many other examples. For more details, we refer to the survey \cite{Osi18}. 

By \cite[Theorem 2.24]{DGO}, every acylindrically hyperbolic group contains a unique maximal finite normal subgroup denoted by $K(G)$. We call $K(G)$  the \emph{ finite radical} of $G$. Let $F_n$ denote the free group of rank $n$. We will need the following result, which can be thought of as a more precise version of Theorem \ref{Thm:ah}.

\begin{thm}[{\cite[Theorem 2.24]{DGO}}]\label{Thm:Fhe}
Let $G$ be an acylindrically hyperbolic group. For every $n\in \NN$, there exists a hyperbolically embedded subgroup of $G$ isomorphic to $K(G)\times F_n$.
\end{thm}

\begin{rem}
For an acylindrically hyperbolic group $G$, $K(G)$ is contained in every infinite hyperbolically embedded subgroup of $G$ by Proposition \ref{Prop:maln}. Thus, $K(G)\times F_n$ cannot be replaced with $F_n$ in Theorem \ref{Thm:Fhe}. Similarly, the equality $E(g)=\langle g\rangle \times K(G)$ cannot be replaced with $E(g)=\langle g\rangle $ in Theorem \ref{Thm:Eg} below.
\end{rem}

We record a couple of useful results concerning the ICC condition.

\begin{thm}[{\cite[Theorem 2.35]{DGO}}]\label{Thm:HypICC}
An acylindrically hyperbolic group $G$ is ICC if and only if $K(G)=\{ 1\}$.
\end{thm}

The next lemma is a ``baby version" of \cite[Corollary 1.7]{Osi17}. It admits a simple proof, which we provide for the convenience of the reader.

\begin{lem}\label{Lem:FA}
Suppose that $G$ is an ICC acylindrically hyperbolic group acting on a set $I$. Then $Stab_G(i)$ is acylindrically hyperbolic for every $i\in I$ or the action is faithful. In particular, every action of an ICC acylindrically hyperbolic group with amenable stabilizers is faithful.
\end{lem}
\begin{proof}
Let $K=\bigcap_{i\in I}Stab_G(i)$ be the kernel of the action. Suppose that $K\ne \{ 1\}$. Since $G$ is ICC, $K$ must be infinite. Let $G\curvearrowright S$ be an acylindrical, non-elementary,  isometric action of the group $G$ on a hyperbolic space $S$. Every infinite normal subgroup of $G$ acts on $S$ non-elementarily by \cite[Lemma 7.1]{Osi16}; in particular, so does $K$. Since $K\le Stab_G(i)$ for all $i\in I$, every $Stab_G(i)$ also acts non-elementarily and, therefore, is acylindrically hyperbolic. This proves the first claim of the lemma. To prove the second claim, it suffices to note that acylindrically hyperbolic groups are non-amenable by Theorem \ref{Thm:Fhe}.
\end{proof}

We conclude this section with two useful examples of hyperbolically embedded subgroups in acylindrically hyperbolic groups.

\begin{defn}\label{Def:lox}
An element $g$ of a group $G$ acting on a hyperbolic space $S$ is called \emph{loxodromic} (or is said to \emph{act loxodromically} on $S$) if it acts as a translation along a bi-infinite quasi-geodesic in $S$. If the action of $G$ on $S$ is acylindrical, this is equivalent to the requirement that $\langle g\rangle$ has unbounded orbits (see \cite[Lemma 2.2]{Bow}).
\end{defn}

The following is proved in \cite[Lemma 6.5]{DGO}.

\begin{lem}\label{Lem:Eg}
Suppose that a group $G$ acts acylindrically on a hyperbolic space $S$. Then every loxodromic element $g\in G$ is contained in a unique maximal virtually cyclic subgroup of $G$.
\end{lem}

\begin{defn}
In the setting of Lemma \ref{Lem:Eg}, we denote the unique maximal virtually cyclic subgroup of $G$ containing $g$ by $E(g)$.
\end{defn}

The next result follows immediately from \cite[Corollary 3.12 and Lemma 5.15]{AMS}.

\begin{thm}\label{Thm:Eg}
Let $G$ be a group, $\Hl$ a finite collection of subgroups of $G$ such that $\Hl\h (G,X)$ for some $X\subseteq G$. Let also $\mathcal A$ be the alphabet defined by (\ref{calA}). Suppose that the action of $G$ on $\G$ is acylindrical. Then the following hold.
\begin{enumerate}
\item[(a)] For any element $g\in G$ acting loxodromically on $\Gamma (G,A)$, we have $\{ E(g)\}\cup \Hl \h (G, X)$.
\item[(b)] If the action of $G$ on $\G$ is non-elementary, then there exists an element $g\in G$ acting loxodromically on $\Gamma (G,A)$ such that $E(g)=\langle g\rangle \times K(G)$.
\end{enumerate}
\end{thm}

%%%%%%%%%%%%%%%%%%%%%%%%%%%%%%%%%%%%%%%%%%%%%%%%%%%%%%%%%%%%%%%%%%%%%%%%%%%%%%%%%%%%%%%%%%%%%%%%%%%%%%%%%%%%%%%

\subsection{Suitable subgroups and quotients of relatively hyperbolic groups} \label{Sec:Suit}

%%%%%%%%%%%%%%%%%%%%%%%%%%%%%%%%%%%%%%%%%%%%%%%%%%%%%%%%%%%%%%%%%%%%%%%%%%%%%%%%%%%%%%%%%%%%%%%%%%%%%%%%%%%%%%%

Let $G$ be a group hyperbolic relative to a collection of subgroups $\Hl$ and let $X$ be a finite relative generating set of $G$ with respect to $\Hl$. Recall that an element $g\in G$ is \emph{loxodromic with respect to the peripheral collection} $\Hl$ if it acts as a loxodromic isometry on the hyperbolic space $\G$, where $\mathcal A$ is defined by (\ref{calA}). By Theorems \ref{Thm:class} and \ref{Thm:RHAH}, this is equivalent to saying that $g$ acts on $\G$ with unbounded orbits. Yet another equivalent condition is that $g$ has infinite order and is not conjugate to an element of some $H_i$ (see \cite[Theorem 4.23]{DGO}). Given a subgroup $S\le G$, we denote by $\mathcal {\rm L}(S; \Hl)$ the set of all loxodromic elements of $S$ with respect to $\Hl$. In this notation, we have the following.

\begin{defn}\label{Def:suit}
A subgroup $S\le G$ is said to be {\it suitable with respect to the peripheral structure $\Hl$} (or simply \emph{suitable} if the peripheral structure can be understood from the context) if $\mathcal {\rm L}(S; \Hl)\ne \emptyset$, $S$ is not virtually cyclic, and $S$ does not normalize any non-trivial finite subgroup of $G$.
\end{defn}

The definition of a suitable subgroup was first formulated in \cite{Osi10} in a slightly different way; it was shown to be equivalent to Definition \ref{Def:suit} in \cite{AMO} (see Lemma 3.3 and Proposition 3.4 there). It is worth noting that the existence of a suitable subgroup in a group $G$ as above implies that $G$ is acylindrically hyperbolic and $K(G)=\{ 1\}$. Thus, all the results discussed below only apply to acylindrically hyperbolic groups with trivial finite radical.

\begin{lem}\label{Lem:suit}
Let $G$ be a non-elementary relatively hyperbolic group with peripheral subgroups $\Hl$. If $K(G)=\{ 1\}$, then every non-trivial normal subgroup of $G$ is suitable with respect to $\Hl$.
\end{lem}

\begin{proof}
If $K(G)=\{ 1\}$ and $S$ is a non-trivial normal subgroup of $G$, then $S$ is infinite. It follows from \cite[Proposition 5.2 and Lemma 7.1]{Osi16}) that $S$ is not virtually cyclic and contains loxodromic elements. By \cite[Lemma 3.3]{AMO}, for every such a subgroup, there exists a unique maximal finite subgroup $E\le G$ normalized by $S$. Since $S$ is normal in $G$, so is $E$. In particular, $E\le K(G)$, which implies $E=\{ 1\}$.
\end{proof}

The next theorem can be found in \cite[Theorem 2.4]{Osi10}. For a group $G$ and a subset $S\le G$, we denote by $\ll S\rr $ the normal closure of $S$ in $G$. I.e., $\ll S\rr$ is the smallest normal subgroup of $G$ containing $S$.

\begin{thm}\label{glue}
Let $G$ be a group hyperbolic relative to a collection of subgroups $\Hl$ and let $S\le G$ be a suitable subgroup of $G$. For every finite subset $\mathcal F =\{f_1, \ldots , f_k\}$ of $G$, there exist elements $s_1, \ldots , s_k\in S$ such that the quotient group
\begin{equation}\label{Eq:Gbar}
\overline{G}=G/\ll f_1s_1, \ldots , f_ks_k\rr
\end{equation}
satisfies the following conditions.
\begin{enumerate}
\item[(a)]  The restriction of the natural homomorphism $\gamma\colon G\to \overline{G}$ to the set $\bigcup \limits_{i\in I} H_i$ is injective and $\overline{G}$ is hyperbolic relative to $\{ \gamma(H_i)\}_{i\in I}$.
\item[(b)]  Every finite order element of $\overline{G}$ is the image of a finite order element of $G$.
\item[(c)] $\gamma(S)$ is a suitable subgroup of $\overline G$ with respect to $\{ \gamma(H_i)\}_{i\in I}$. In particular, $\overline G$ is non-elementary relatively hyperbolic and ICC.
\end{enumerate}
\end{thm}

Two remarks are in order. First, \cite[Theorem 2.4]{Osi10} was stated in a slightly different way: only the existence of a quotient group $\overline{G}$ of $G$ satisfying (a), (b), and the condition $\gamma(f)\in \gamma(S)$ for all $f\in \mathcal F$ is claimed in \cite{Osi10}. However, the fact that $\overline{G}$ is obtained from $G$ by passing to the quotient of the form (\ref{Eq:Gbar}) is obvious from the proof.

Second, only the first claim in part (c) is stated explicitly \cite[Theorem 2.4]{Osi10}. The second claim can be derived as follows.  By \cite[Definition 3.2 and Proposition 3.4]{AMO}, the definition of a suitable subgroup given in this paper is equivalent to the existence of two loxodromic elements $a,b \in \gamma(S)$ such that $E(a)\cap E(b) = \{1\}$, where $E(a)$ and $E(b)$ are the maximal virtually cyclic subgroups of $\overline{G}$ containing $a$ and $b$, respectively (such maximal subgroups always exist by \cite[Theorem 4.3]{Osi06b} or Lemma \ref{Lem:Eg}). Obviously, this implies that $\overline{G}$ is not virtually cyclic and properly relatively hyperbolic. Thus, $\overline{G}$ is non-elementary relatively hyperbolic and, in particular, acylindrically hyperbolic. Note that $K(\overline{G})$ is normalized by $\gamma(S)$. Therefore, we have $K(\overline{G})=\{ 1\}$ by Definition \ref{Def:suit}. The latter condition is equivalent to being ICC by Theorem \ref{Thm:HypICC}.

Recall that every hyperbolic group is hyperbolic relative to the empty set of subgroups. In this case, Definition \ref{Def:suit} reduces to the following: a subgroup $S$ of a hyperbolic group $G$ is \emph{suitable} if and only if $S$ is not virtually cyclic, contains an element of infinite order, and does not normalize any non-trivial finite subgroup of $G$. Applying Theorem \ref{glue} to a suitable subgroup $S$ of a hyperbolic group $G$ and the empty collection $\Hl$, we obtain a hyperbolic quotient group $\overline{G}$. In these settings, a similar theorem was formulated by Gromov \cite{Gro} and proved by Olshanskii in \cite{Ols93}.

We record a few simple applications of Theorem \ref{glue}. For the first application, we do not need relative hyperbolicity and results of \cite{Ols93} would suffice.

\begin{cor}\label{Cor:TrAb}
Every non-elementary hyperbolic group $H$ has a non-elementary hyperbolic quotient $\overline{H}$ with trivial abelianization. Moreover, if $H$ is torsion-free, then so is $\overline{H}$.
\end{cor}

\begin{proof}
By passing from $H$ to $H/K(H)$ if necessary, we can assume that $K(H)=\{ 1\}$. Since $H$ is non-elementary, $[H,H]\ne \{ 1\}$. By Lemma \ref{Lem:suit}, $[H,H]$ is a suitable subgroup of $H$. This allows us to apply Theorem \ref{glue} to a finite set $\mathcal F$ of generators of $H$ and $S=[H,H]$. The obtained quotient group has the required properties.
\end{proof}

The next corollary is a more precise version of Theorem \ref{Thm:Fhe} for relatively hyperbolic groups. It can be proved in a number of ways, utilizing either the geometric ideas from \cite{Osi06b} (further developed in \cite{DGO}) or the random walk technique used in \cite{AH}. It is somewhat surprising that this fact has not been recorded in the literature. We provide a brief proof based on the results discussed above.

\begin{cor}\label{Cor:Fhe}
Let $G$ be a group hyperbolic relative to a collection of subgroups $\Hl$ such that $K(G)=\{ 1\}$. For any $n\in \NN$ and any suitable subgroup $S\le G$, there is a free subgroup $F_n\le S$ of rank $n$ such that  $G$ is hyperbolic relative to $\Hl \cup \{ F_n\}$.
\end{cor}

\begin{proof}
We fix $n\in \NN$ and a suitable subgroup $S\le G$. The free product $P=G\ast F_n$ is hyperbolic relative to $\{ G, F_n\}$ (see Example \ref{Ex:RH}. By Proposition \ref{Prop:rhrh}, $P$ is also hyperbolic relative to $\Hl\cup \{ F_n\}$. Using the standard properties of free products and the assumption $K(G)=\{ 1\}$, it is easy to show that $S$ is suitable in $P$ with respect to the peripheral collection $\Hl\cup \{ F_n\}$.

Let $\mathcal F=\{ f_1, \ldots, f_n\}$ be a finite basis of $F_n$ and let $\{ s_1, \ldots, s_n\}$ be the elements provided by Theorem \ref{glue}. Let $\langle X\mid \mathcal R\rangle $ be a presentation of $P$, where $X=G\cup \mathcal F$ and $\mathcal R$ is a set of words in the alphabet $X$. Let also $$\overline P= \langle X\mid \mathcal R, \; f_1s_1, \ldots, f_n s_n\rangle.$$ By part (a) of Theorem \ref{glue}, $\overline P$ is hyperbolic relative to the $\gamma$-image of the collection $\Hl\cup \{ F_n\}$. Using Tietze transformations, we can remove the relators $f_is_i$ and generators $f_i$ from the presentation of $\overline P$. This yields an isomorphism $\overline P\cong G$ that sends each $\gamma(H_i)$ to $H_i$ and $\gamma(F_n)$ to $F_n$. Thus, $G$ is hyperbolic relative to $\Hl \cup \{ F_n\}$.
\end{proof}

The following variant of the famous Rips construction \cite{Rips} is similar to the one considered in \cite{BO}.

\begin{prop}\label{Rips}
Let $S$ be a finitely generated group, $M$ the normal closure of finitely many elements in $S$. For any non-elementary hyperbolic group $H$, there exist a finitely generated group $G$ containing $S$ and an infinite normal subgroup $N\lhd G$ such that the following conditions hold.
\begin{enumerate}
\item[(a)] $G$ is ICC and non-elementary hyperbolic relative to $S$.
\item[(b)] $N$ is a non-trivial quotient of $H$. In particular, $N$ is finitely generated.
\item[(c)] The restriction of the natural homomorphism $G\to G/N$ to $S$ is surjective and we have a natural isomorphism $G/N\cong S/M$ (i.e., $S\cap N=M$).
\item[(d)] If $H$ and $S$ are torsion-free, then so is $G$.
\end{enumerate}
\end{prop}

\begin{proof}
For convenience of the reader, we provide a commutative diagram illustrating conditions (b) and (c).
\begin{center}
\begin{tikzcd}
H \arrow[r, two heads] & N \arrow[r, hook]                 & G \arrow[r, two heads]  & G/N \\
                       & M \arrow[u, hook] \arrow[r, hook] & S \arrow[ru, two heads] \arrow[u, hook]&    
\end{tikzcd}
\end{center}

Let $X$ be a finite generating set of $S$ and let $M$ be the normal closure of a finite set $Y$ in $S$. Passing from $H$ to $H/K(H)$ if necessary, we can assume that $K(H)=\{ 1\}$. The free product $P=S\ast H$ is hyperbolic relative to $\{ S, H\}$ (see Example \ref{Ex:RH}). Since $H$ is hyperbolic, $P$ is also hyperbolic relative to $S$ by Corollary \ref{Cor:rhh}. Since $H$ is non-elementary and $K(H)=\{ 1\}$, $H$ is a suitable subgroup of $P$ with respect to $S$. Let $Z$ be a finite generating set of $H$ and let 
$$
\mathcal F=\{ x^{-1}zx, xzx^{-1}\mid x\in X, z\in Z\}\cup Y.
$$ 
By Theorem \ref{glue},  there exists a set of elements $\{ h_f\in H\mid f\in \mathcal F\}$ such that the quotient group
\begin{equation}\label{Eq:G0}
G=P/\ll fh_f \mid f\in \mathcal F\rr
\end{equation}
is hyperbolic relative to an isomorphic image of $S$; we will identify this image with $S$. Let $N$ denote the image of $H$ in $G$. 

Given an element $p\in P$, we denote by $\overline p$ its image in $G$. Since $x^{\pm 1}z\,x^{\mp 1}\in \mathcal F$ for all $x\in X$ and $z\in Z$,  formula (\ref{Eq:G0}) and the fact that $h_f\in H$ imply that $\overline x^{\pm 1}\,\overline z\, {\overline x}^{\mp 1}\in N$ for all $x\in X$ and $z\in Z$. Thus, $N\lhd G$. Furthermore, observe that $N$ is the image of the normal closure $\ll \mathcal F \cup H\rr$ in $P=S\ast H$ under the natural homomorphism $\gamma\colon P\to G$. Clearly, we have $\ll\mathcal F \cup H\rr= \ll Y \cup H\rr$ in $P$. Therefore, the restriction of the map $G\to G/N$ to $S\le G$ is surjective and the kernel of the corresponding homomorphism $S\to G/N$ coincides with $\gamma(\ll Y\rr)=M$. Thus, (c) holds. 

It remains to note that $N$ is suitable (in particular, $N\ne \{ 1\} $) and $G$ is ICC and non-elementary relatively hyperbolic with respect to $S$ by part (c) of Theorem~\ref{glue}. Part (d) also follows from Theorem \ref{glue}.
\end{proof}

%%%%%%%%%%%%%%%%%%%%%%%%%%%%%%%%%%%%%%%%%%%%%%%%%%%%%%%%%%%%%%%%%%%%%%%%%%%%%%%%%%%%%%%%%%%%%%%%%%%%%%%%%%%%%%%

\subsection{Group theoretic Dehn filling}\label{Sec:DF}

%%%%%%%%%%%%%%%%%%%%%%%%%%%%%%%%%%%%%%%%%%%%%%%%%%%%%%%%%%%%%%%%%%%%%%%%%%%%%%%%%%%%%%%%%%%%%%%%%%%%%%%%%%%%%%%

The classical Dehn surgery on a $3$-dimensional manifold consists of cutting off a solid torus, which may be thought of as ``drilling" along an embedded knot, and then gluing it back in a different way. The second part of the process, called {\it Dehn filling}, has a purely algebraic counterpart. Below we briefly review the necessary definitions and results. For details and connections to $3$-dimensional geometry, we refer to  \cite{DGO,GM,Osi07} and references therein.

\begin{defn}\label{Def:DF}
Let $G$ be a group, $\Hl$ a collection of subgroups of $G$. Given a collection $\mathcal N=\{N_i\}_{i\in I}$, where $N_i\lhd H_i$ for all $i\in I$, we call the quotient group
$$
G(\mathcal N)=G\left/\left\langle\hspace{-2mm}\left\langle  \bigcup\limits_{i\in I} N_i\right\rangle\hspace{-2mm}\right\rangle\right. .
$$
the \emph{Dehn filling of $(G, \Hl)$ corresponding to the collection of kernels $\mathcal N$}. Further, we say that a certain property $\mathcal P$ \emph{holds for any sufficiently deep subgroups $N_i\lhd H_i$} if there exist finite subsets $\mathcal F_i\subseteq H_i\setminus \{ 1\}$ such that $\mathcal P$ holds for $G(\mathcal N)$ whenever the kernels satisfy $N_i\cap \mathcal F_i=\emptyset$ for all $i\in I$.
\end{defn}

The following theorem summarizes results of \cite{DGO,Osi07} necessary for this paper.

\begin{thm}\label{Thm:DF}
Let $G$ be a group, $\Hl$ a collection of subgroups of $G$. Suppose that $\Hl\h (G,X)$ for some $X\subseteq G$. Then for any sufficiently subgroups $N_i\lhd H_i$, the following conditions hold.
\begin{enumerate}
\item[(a)] The natural maps $H_i/N_i\to G(\mathcal N)$ are injective. In what follows, we think of $H_i/N_i$ as subgroups of $G(\mathcal N)$.
\item[(b)] $\{ H_i/N_i\}_{i\in I}\h (G(\mathcal N), \e(X))$, where $\e\colon G\to G(\mathcal N)$ is the natural homomorphism. In particular, if $G$ is hyperbolic relative to a finite collection $\Hl$, then $G(\mathcal N)$ is hyperbolic relative to $\{H_i/N_i\}_{i\in I}$.
\item[(c)] There exists a subset $T_i\subset G$ for each $i$ such that
\[\ll \cup_{i\in I}N_i \rr=\ast_{i\in I}\ast_{t\in T_i}N^t_i.\]
\item[(d)] Suppose $G$ is hyperbolic relative to $\{H_i\}_{i\in I}$. If $G$ and $H_i/N_i$ are torsion-free for all $i\in I$ then so is $G(\mathcal N)$.
\end{enumerate}
\end{thm}

\begin{proof}
(a), (b) and (c) were proved in \cite[Theorem 7.19]{DGO}. In the particular case of relatively hyperbolic groups, (a) and (b) were proved in \cite{Osi07} (for torsion-free relatively hyperbolic groups, an alternative proof was given in \cite{GM}).

Now we prove  (d). As $G$ is hyperbolic relative to $\{H_i\}_{i\in I}$ we may assume $|X|<\infty$ by Proposition \ref{rhhe}. Let $\overline X$ be the image of $X$ under the quotient map $G\rightarrow G(\mathcal{N})$ and let $\bar g$ be a finite-order element of $G(\mathcal{N})$. By \cite[Theorem 7.19 (f)]{DGO}, $\bar g$ is the image of an element $g\in G$ that acts elliptically on the Cayley graph $\Gamma(G,\mathcal A)$, where $\mathcal A$ is defined by (\ref{calA}). By \cite[Theorem 4.23]{Osi06}, such an element $g$ either has finite order or is conjugate to an element of $H_i$ for some $i\in I$. As we assume $G$ is torsion-free, only the latter can happen, and thus $\bar g$ is conjugate to an element of $H_i/N_i$. Since $H_i/N_i$ is torsion-free, $\bar g=1$, as desired.
\end{proof}

We will also need the following.

\begin{prop}\label{Lem:DFICC}
Let $G$ be a group, $\Hl$ a collection of proper subgroups of $G$. Suppose that $\Hl\h G$ and $G$ is ICC.  Then, for any sufficiently deep subgroups $N_i\lhd H_i$,  the corresponding Dehn filling $G(\mathcal N)$ of $(G,\Hl)$ is acylindrically hyperbolic and ICC.
\end{prop}

\begin{proof}
If each $H_i$ is finite, the lemma is vacuously true. Indeed, we can take $\mathcal F_i=H_i\setminus\{ 1\}$ and then only the subgroups $N_i=\{ 1\}$ are sufficiently deep.

Thus, we can assume that at least one $H_i$ is infinite. By \cite[Theorem 5.4]{Osi16}, there exists $Y\subseteq G$ such that $\Hl\h (G,Y)$ and the action of $G$ on $\Gamma (G, \mathcal H \cup Y)$ is acylindrical (recall that $\mathcal H$ is defined by (\ref{calA})). Let $a\in G\setminus H_i$. By Proposition \ref{Prop:maln}, we have $|a^{-1}H_ia\cap H_i|<\infty$. The existence of such an element $a$ and the assumption $|H_i|=\infty $ allow us to apply \cite[Theorem 6.11]{DGO}, which implies that the group $G$ contains an element acting loxodromically on $\Gamma (G, \mathcal H \cup Y)$. In particular, the action $G\curvearrowright \Gamma (G, \mathcal H \cup Y)$ has unbounded orbits. Since $G$ is ICC, it is not virtually cyclic. Therefore, the action $G\curvearrowright\Gamma (G, \mathcal H \cup Y)$ is non-elementary by Theorem \ref{Thm:class}. Note that $K(G)=\{1\}$ by Theorem \ref{Thm:HypICC}. By Theorem \ref{Thm:Eg} (b), $G$ contains a loxodromic element $g\in G$ such that $\langle g\rangle \cup \Hl \h G$.

Every Dehn filling $G(\mathcal N)$ of $(G,\Hl)$ can also be thought of as the Dehn filling of $(G, \Hl\cup \langle g\rangle )$ corresponding to the trivial kernel $\{ 1\}\lhd \langle g\rangle$. Applying Theorem \ref{Thm:DF} to the hyperbolically embedded collection of subgroups $\Hl\cup \langle g\rangle$, we obtain that the Dehn filling $G(\mathcal N)$ of $(G,\Hl)$  with respect any collection of sufficiently deep subgroups $N_i\lhd H_i$ contains a hyperbolically embedded, proper, infinite cyclic subgroup (namely, the image of $\langle g\rangle$). In particular, $G(\mathcal N)$ is acylindrically hyperbolic by Theorem \ref{Thm:ah}. It remains to note that the existence of an infinite cyclic hyperbolically embedded subgroup in an acylindrically hyperbolic group implies triviality of the finite radical by Proposition \ref{Prop:maln}. Applying Theorem \ref{Thm:HypICC} again we conclude that $G(\mathcal N)$ is ICC.
\end{proof}

We end this section with a lemma which will be used in the proof of Theorem \ref{symmetries}. We say that a group $S$ is a {\it nontrivial free product}  if $S=S_1*S_2$ with $|S_1|\geq 2$ and $|S_2|\geq 3$. Recall that a subgroup of a hyperbolic group is {\it elementary} if it contains a cyclic subgroup of finite index. Recall also that a subgroup $H$ of a group $G$ is \emph{malnormal} (respectively, \emph{almost malnormal}) if $H\cap g^{-1}Hg=\{ 1\} $ (respectively, $|H\cap g^{-1}Hg|<\infty $) for all $g\in G\setminus H$.

\begin{lem}\label{exact}
Let  $G$ be a group that is hyperbolic relative to a finite collection of finitely generated, residually finite  subgroups $\{H_i\}_{i\in I}$. Let $D$ be a non-trivial, ICC, property (T) subgroup of $G$. Suppose also that, for every $i\in I$,  $D$ is not conjugate to a subgroup of $H_i$.
Then there is a short exact sequence $1\ra S \ra D \ra K\ra 1,$ where $ S$ is either trivial or a nontrivial free product group, and $K$ is a non-elementary subgroup of a hyperbolic group.
\end{lem}

\begin{proof}
Since the groups $\{H_i\}_{i\in I}$ are residually finite, Theorem \ref{Thm:DF} implies the existence of finite index normal subgroups $N_i\lhd H_i$  such that conditions (a)--(d) of the theorem hold. Since $H_i/N_i$ is finite for all $i\in I$, Corollary \ref{Cor:rhh} implies that the quotient group $G(\mathcal N)$ is hyperbolic.

Let $N$ denote the kernel of the natural homomorphism $G\to G(\mathcal N)$ and let $K$ denote the image of $D$ in $G(\mathcal N)$. Further, let $S=D\cap N$.  Obviously, we have a short exact sequence $1\ra  S \ra D \ra K\ra 1$.

First, we claim that $K$ is infinite. Indeed, otherwise, there is a finite index subgroup $D_0\le D$ such that $D_0\le  N$. Since $D_0$ has property (T) and $N$ decomposes as a free product of conjugates of subgroups $N_i$ (see part (c) of Theorem \ref{Thm:DF}), we can find $i\in I$ and $g\in G$ such that $g D_0g^{-1}\le  N_i$. Since $ H_i$ is almost malnormal in $ G$ by Proposition \ref{Prop:maln} and $D$ is infinite (being non-trivial and ICC), we get $g D g^{-1}\le H_i$, which contradicts the assumption of the lemma.

Second, we show that $S$ is either trivial or a nontrivial free product group. Indeed, suppose this is not true. Theorem \ref{Thm:DF} (c) and the Kurosh subgroup theorem applied to $S\leqslant N$ imply that either a)  $g S g^{-1}\le N_i$, for some $i\in I$ and $g\in G$, or b) $ S\cong\mathbb Z$ or $S\cong \mathbb Z/2\mathbb Z\ast \mathbb Z/2\mathbb Z$.
Note that $S$ is infinite since $S\lhd D$ and $D$ is ICC. If a) holds, Proposition \ref{Prop:maln} implies that $g D g^{-1}\le N_i$,  which is a contradiction. Further, case (b) contradicts the assumption that $D$ is ICC. \end{proof}

%%%%%%%%%%%%%%%%%%%%%%%%%%%%%%%%%%%%%%%%%%%%%%%%%%%%%%%%%%%%%%%%%%%%%%%%%%%%%%%%%%%%%%%%%%%%%%%%%%%%%%%%%%%%%%%
%%%%%%%%%%%%%%%%%%%%%%%%%%%%%%%%%%%%%%%%%%%%%%%%%%%%%%%%%%%%%%%%%%%%%%%%%%%%%%%%%%%%%%%%%%%%%%%%%%%%%%%%%%%%%%%

\section{Wreath-like products of groups}\label{Sec:WR}

%%%%%%%%%%%%%%%%%%%%%%%%%%%%%%%%%%%%%%%%%%%%%%%%%%%%%%%%%%%%%%%%%%%%%%%%%%%%%%%%%%%%%%%%%%%%%%%%%%%%%%%%%%%%%%%
%%%%%%%%%%%%%%%%%%%%%%%%%%%%%%%%%%%%%%%%%%%%%%%%%%%%%%%%%%%%%%%%%%%%%%%%%%%%%%%%%%%%%%%%%%%%%%%%%%%%%%%%%%%%%%%

%%%%%%%%%%%%%%%%%%%%%%%%%%%%%%%%%%%%%%%%%%%%%%%%%%%%%%%%%%%%%%%%%%%%%%%%%%%%%%%%%%%%%%%%%%%%%%%%%%%%%%%%%%%%%%%

\subsection{Some algebraic properties}\label{Sec:WRAlg}

%%%%%%%%%%%%%%%%%%%%%%%%%%%%%%%%%%%%%%%%%%%%%%%%%%%%%%%%%%%%%%%%%%%%%%%%%%%%%%%%%%%%%%%%%%%%%%%%%%%%%%%%%%%%%%%

In this section, we discuss the basic algebraic properties of wreath-like products necessary for what follows. We begin by describing the structure of subgroups of wreath-like products containing the base. Our first result is obvious.

\begin{lem}\label{Lem:WRsubgr0}
Let $A$, $B$ be any groups, $B\curvearrowright I$ any action of $B$ on a set $I$. Let $W\in \WR(A,B\curvearrowright I)$ and let $\e\colon W\to B$ denote the canonical homomorphism. For any subgroup $V\le W$ containing the base $A^{(I)}$, we have $V\in \WR (A,\e(V)\curvearrowright I)$.
\end{lem}
\begin{proof}
The lemma follows immediately from the definition of a wreath-like product.
\end{proof}

Under the assumptions of Lemma \ref{Lem:WRsubgr0}, the action of $\e(V)$ on $I$ is, in general, not transitive. However, if $W$ is a regular wreath-like product, $V$ has a natural regular wreath-like structure. The result below can be seen as a particular case of \cite[Lemma 2.8]{CIOS1}. We provide a simple direct proof for convenience of the reader.

\begin{lem}\label{WRsubgr}
Let $A$, $B$ be any groups, $W\in \WR(A,B)$, and let $\e\colon W\to B$ denote the canonical homomorphism. Suppose that $V$ is a subgroup of $W$ containing the base $A^{(I)}$. Then $V\in \WR (C,\e(V))$, where $C$ is the direct sum of $|B : \e(V)|$-many copies of $A$.
\end{lem}

\begin{proof}
Let $D=\e(V)$ and let $T$ be a right transversal of $D$ in $B$. That is,
\begin{equation}\label{B=DT}
B=\bigsqcup_{t\in T} Dt.
\end{equation}
For every $d\in D$, we let $C_d=\bigoplus_{t\in T} A_{dt}\le A^{(B)}$. Obviously, (\ref{B=DT}) implies that $A^{(B)}=\bigoplus_{d\in D}C_d=C^{(D)}$. For every $v\in V$, we have
$$
vC_dv^{-1}=\bigoplus_{t\in T}vA_{dt}v^{-1}=\bigoplus_{t\in T}A_{\e(v)dt}=C_{\e(v)d}
$$
and the result follows.
\end{proof}

We now turn to quotients of wreath-like products. The following result was proved by the authors in \cite{CIOS1}.

\begin{lem}[{\cite[Lemma 2.12]{CIOS1}}]\label{A/N}
Let $A$, $B$ be any groups, $W\in \WR(A,B)$. We identify $A$ with the subgroup $A_1$ of the base $\bigoplus_{b\in B}A_b\le W$. For any $N\lhd A$, we have $W/\ll N\rr \in \WR (A/N, B)$.
\end{lem}

Next, we discuss centralizers of elements. The results obtained below will be used in the proofs of Theorems \ref{Thm:MC} and \ref{symmetries}.
As usual, by $C_G(g)$ and $C_G(S)$ we denote the centralizer of an element $g$ and a subset $S$ of a group $G$, respectively. For a subgroup $H\le G$, the \textit{virtual centralizer} of $H$ in $G$, denoted by $vC_G(H)$, is the set of all elements $g\in G$ that commute with a finite index subgroup of $H$. Clearly, $vC_G(H)$ is a subgroup.

\begin{lem}\label{Lem:Centr}
Let $A$ be a non-trivial group, $B$ an ICC acylindrically hyperbolic group and $B \curvearrowright I$ a faithful action of $B$ on a set $I$. For any $W\in \WR(A,B\curvearrowright I)$ and any finite index subgroup $A_0\leq A^{(I)}$, we have $C_W(A_0)\leqslant A^{(I)}$. Moreover, the following hold.
\begin{enumerate}
\item[(a)] If $A$ is ICC, then $C_W(A_0)=1$; in particular, $vC_W(A^{(I)})=1$;
\item[(b)] If $A$ is abelian, then $C_W(A_0)=A^{(I)}$; in particular, $vC_W(A)=A^{(I)}$.
\end{enumerate}
\end{lem}

\begin{proof}
We first show that the set $$S(b)=\{i\in I\mid b\cdot i\not=i\}$$ is infinite for all $b\in B\setminus\{ 1\}$. Indeed, let $K\lhd B$ be the set of all $b\in B$ such that 
$|S(b)|<\infty$. Clearly, $K$ is a normal subgroup of $B$. Since the action of $B$ on $I$ is faithful, $B$ embeds in the permutation group $Sym(I)$. Under this embedding, $K$ is mapped to the subgroup of finitely supported permutations. The latter subgroup is locally finite and, therefore, so is $K$. In particular, $K$ is amenable. By \cite[Corollary 8.1 (a)]{Osi16}, the amenable radical of an acylindrically hyperbolic group is finite. Thus, $K$ is finite. Since $B$ is ICC, we deduce that $K=\{1\}$.

Let now $A_0\leqslant A^{(I)}$ be a finite index subgroup. Assume by contradiction that $C_W(A_0)$ is not contained in $A^{(I)}$. Then there is  $g\in C_W(A_0)$ such that  $\e(g)\ne 1$, where $\e\colon W\to B$ denotes the canonical homomorphism. As shown above, the set $S(\e(g))$ must be infinite. It follows that there exists an infinite subset  $J\subseteq S(\e(g))$ such that 
\begin{equation}\label{Eq:J}
\e(g) J\cap J=\emptyset. 
\end{equation}
Indeed, let $J_1=\{i_1\}$, where $i_1$ is an arbitrary element of $S(\e(g))$. By induction, for $n\ge 2$, we define a finite set $J_{n+1}=J_{n}\cup \{i_{n+1}\}$, where $i_{n+1}$ is an arbitrary element of the (infinite) set  $S(\e(g))\setminus \big(J_{n}\cup \e(g)J_{n}\cup \e(g^{-1})J_n\big)$. It is easy to see that $J=\{ i_n\}_{n\in \NN}$ satisfies (\ref{Eq:J}).

Since  $A_0$ has finite index in $A^{(I)}$, $A_0\cap A^{(J)}$ has finite index in the (infinite) group $A^{(J)}$. In particular, $A_0\cap A^{(J)}\ne \{ 1\}$. Let $a\in (A_0\cap A^{(J)})\setminus\{1\}$. By the definition of a wreath-like product, we obtain $a=gag^{-1}\in A^{(\e(g)J)}\setminus\{1\}$ which contradicts (\ref{Eq:J}). This proves that $C_W(A_0)\leqslant A^{(I)}$.

If $A$ is ICC, then so is $A^{(I)}$. Since $A_0\leqslant A^{(I)}$ has finite index,  $C_{A^{(I)}}(A_0)=\{1\}$ and thus $C_{W}(A_0)=\{1\}$, which proves  (a). If $A$ is abelian, then $A^{(I)}\leqslant C_W(A_0)$ and hence $C_W(A_0)=A^{(I)}$, which proves (b).
 \end{proof}

Given an element $a\in A^{(I)}$ of a wreath-like product $W\in \WR(A, B\curvearrowright I)$, we define its \emph{support} by
$$
supp (a)=\{ i\in I\mid a(i)\ne 1\}.
$$

\begin{lem}\label{Lem:centr}
Let $A$ be an arbitrary group, $B$ a torsion-free group, $B \curvearrowright I$ an action of $B$ on a set $I$ with almost malnormal stabilizers. Let $W\in \WR(A,B\curvearrowright I)$ and let $\e\colon W\to B$ denote the canonical homomorphism. For every $a\in A^{(I)}$, we have $\e(C_W(a))\le Stab_B(i)$ for all $i\in supp(a)$.
\end{lem}

\begin{proof}
The result is vacuously true if $a=1$ or $\e(C_W(a))=\{ 1\}$, so we assume that $a\ne 1$ and $\e(C_W(a))\ne\{ 1\}$. Since $B$ is torsion-free, $\e(C_W(a))$  must be infinite. The subgroup $\e(C_W(a))\le B$ stabilizes the finite non-empty subset $supp (a)\subseteq I$ setwise. Hence, the subgroup $I=\e(C_W(a))\cap Stab_B(i)$ has finite index in $\e(C_W(a))$ for all $i\in supp (a)$.  Since the intersection of finite index subgroups of an infinite group is infinite, we have $$|c^{-1} Stab_B(i) c \cap Stab_B(i)|\ge |c^{-1}Ic\cap I| =\infty $$ for any $c\in \e(C_W(a))$ and any $i\in supp (a)$. Since $Stab_B(i)$ is almost malnormal, we have $c\in Stab_B(i)$ and the claim of the lemma follows.
\end{proof}

In the rest of this section, we consider wreath-like products with abelian bases. We record a trivial yet useful observation.

\begin{lem}\label{Lem:WRAb}
Let $W\in \WR (A,B\curvearrowright I)$, where $B$ is an arbitrary group acting on a set $I$ and $A$ is abelian. Let $\e\colon W\to B$ denote the canonical homomorphism. If $a\in A^{(I)}$, $u,v\in W$, and $\e(u)=\e(v)$, then $u^{-1}au=v^{-1}av$.
\end{lem}

\begin{proof}
The assumption $\e(u)=\e(v)$ implies that $u$ and $v$ differ by an element of $A^{(I)}$.  Since $A^{(I)}$ is abelian, the result follows.
\end{proof}

Let $W\in \WR(A,B)$, where $A$ is abelian and $B$ is an arbitrary group. Recall that we identify $A$ with the subgroup $A_1\le A^{(B)}$.  Fix a section $\sigma $ of the canonical homomorphism $\e\colon W\to B$. That is, $\sigma$ is any map $B\to W$ such that $\e\circ \sigma \equiv id_B$. Below we think of elements $a\in A^{(B)}$ as functions $B\to A$. For every subset $X\subseteq B$ and every $a\in A^{(B)}$, we define
$$
\pi_X(a) =\prod_{x\in X} \big(\sigma(x)^{-1} a\sigma (x) (1)\big) \in A_1= A.
$$
Note that $\pi_X(a)$ is well-defined. Indeed, conjugation by $\sigma(x)$ defines an isomorphism between $A_x$ and $A_1=\sigma(x)^{-1} A_x\sigma(x)$. This isomorphism sends $a(x)$ to $\sigma(x)^{-1} a\sigma (x)(1)$. Thus, $\sigma(x)^{-1} a\sigma (x)(1)\ne 1$  if and only if $a(x)\ne 1$. This implies that the product $\prod_{x\in X} \big(\sigma(x)^{-1} a\sigma (x) (1)\big)$ has only finitely many non-trivial terms. Note also that $\pi_X(a)$ is independent of the choice of a particular section $\sigma$ by Lemma \ref{Lem:WRAb}.

To state the next result, we need the following.

\begin{defn}\label{Def:UWR}
Let $U\in \WR(A,B \curvearrowright I)$ for some groups $A$, $B$ and some action $B\curvearrowright I$. For any $i\in I$, let $P_i$ denote the preimage of $Stab_B(i)$ in $U$ under the canonical homomorphism $U\to B$. By the definition of a wreath-like product, $P_i$ normalizes the subgroup $A_i\le A^{(I)}$. In general, elements of $P_i$ can act as non-trivial automorphisms of $A_i$.  We say that the wreath-like product $U$ is \emph{untwisted} if $P_i\le C_U(A_i)$ for all $i$. The subset of untwisted wreath-like products in $\WR(A,B \curvearrowright I)$ will be denoted by $\WR_0(A,B \curvearrowright I)$.
\end{defn}

If the action of $B$ on $I$ is transitive, it suffices to check that $P_i\le C_U(A_i)$ for at least one $i\in I$ in order to show that $W$ is untwisted. Note also that $\WR_0(A,B \curvearrowright I)$ is non-empty only if $A$ is abelian.

\begin{lem}\label{Lem:N(R)}
Let $W\in \WR(A,B)$, where $A$ is non-trivial abelian and $B$ is an arbitrary group. Let $R$ be a subgroup of $B$ and let
\begin{equation}\label{Eq:NR}
N_{R} = \left \{ \left. a\in A^{(B)} \;\right|\; \pi_{bR}(a)=1\;\, \text{ for all }\, b\in B\right\} .
\end{equation}
\begin{enumerate}
\item[(a)] $N_R$ is a normal subgroup of $W$.
\item[(b)] The quotient group $W_R=W/N_R$ belongs to $\WR_0(A,B\curvearrowright I)$, where $I$ is the set of cosets $B/R$ and the action of $B$ on $I$ is by left multiplication.
\end{enumerate}
\end{lem}

\begin{proof}
Since $A^{(B)}$ is abelian, for any $b\in B$ and any $a_1, a_2\in A^{(B)}$, we have 
$$
\pi_{bR}(a_1a_2)= \prod_{x\in X} \big(\sigma(x)^{-1} a_1a_2\sigma (x) (1)\big) = \prod_{x\in X} \big(\sigma(x)^{-1} a_1\sigma (x) (1)\big)\cdot  \prod_{x\in X} \big(\sigma(x)^{-1} a_2\sigma (x) (1)\big)=\pi_{bR}(a_1)\pi_{bR}(a_2)
$$
Therefore, $N_R$ is a subgroup (even a normal subgroup) of $A^{(B)}$. To prove (a), it remains to verify that $N_R$ is normal in $W$. Let $\e\colon W\to B$ denote the canonical homomorphism. Using Lemma \ref{Lem:WRAb}, for any $a\in A^{(B)}$, any $b\in B$, and any $w\in W$, we obtain
\begin{equation}\label{piwaw}
\begin{array}{rcl}
\pi_{bR} (w^{-1}aw) & = & \prod\limits_{x\in bR} \big(\sigma(x)^{-1} w^{-1}aw\sigma (x) (1)\big)\\ &&\\ &=&\prod\limits_{x\in bR} \big(\sigma(\e(w)x)^{-1} a\sigma (\e(w)x) (1)\big) \\ &&\\ &=& \prod\limits_{y\in \e(w)bR} \big(\sigma(y)^{-1} a\sigma (y) (1)\big)=\pi_{\e(w)bR}(a).
\end{array}
\end{equation}
This implies that $N_R$ is normal in $W$.

Let $\gamma \colon W\to W/N_R$ be the natural homomorphism and let $I=B/R$. For every coset $i=bR\in I$, we denote by $A^{(i)}$ the subgroup $\bigoplus_{x\in bR} A_x\le A^{(B)}$. Note that $A^{(B)}=\bigoplus_{i\in I} A^{(i)}$ and $N_R=\bigoplus_{i\in I} N_{R,i}$, where
$$
N_{R,i}= N_R\cap A^{(i)}=\left \{ \left. a\in A^{(i)} \;\right|\; \pi_{i}(a)=1\right\}.
$$
Therefore, $\gamma (A^{(B)})=\bigoplus_{i\in I} \widehat A_i$, where $\widehat A_i =A^{(i)}/N_{R,i}$.  It is easy to see that the restriction of the map $\pi_i$ to $A^{(i)}$ is a homomorphism with the image $A$ and kernel $N_{R,i}$. Hence, $\widehat A_i \cong A$. Identifying each $\widehat A_i$ with $A$ via this homomorphism, we can think of the restriction of $\gamma $ to $A^{(B)}$ as a map $A^{(B)} \to A^{(I)}=\bigoplus_{i\in I} \widehat A_i$ given by
\begin{equation}\label{gai}
\gamma(a)(i)=\pi_i(a)
\end{equation}
for all $a\in A^{(B)}$ and all $i\in I$.  Throughout the rest of the proof, we keep the ``hat" in the notation $\widehat A_i$ to help the reader distinguish between the direct summands of $A^{(B)}$ and $A^{(I)}$.

Let $W_R=W/N_R$. Since $\Ker \gamma =N_R \le \Ker \e$, there is a homomorphism  $\e_R\colon W_R\to B$ such that $\e=\e_R\circ \gamma $.  For any $i\in I$, and any $\widehat w\in W_R$, and any preimage $w\in W$ of $\widehat w$ under $\gamma$, we have $\e(w)=\e_R(\gamma(w))=\e_R(\widehat w)$. Using the definition of $\gamma$, we obtain
$$
\widehat w \widehat A_i \widehat w^{-1}= \gamma\left(wA^{(i)} w^{-1}\right)= \gamma \left(A^{(\e(w)i)}\right)=\widehat A_{\e(w)i}= \widehat A_{\e_R(\widehat w)i}.
$$
Thus, $W_R\in \WR(A,B\curvearrowright I)$.

Finally, let $e=R\in I$. To show that $W_R$ is untwisted, it suffices to prove that, for any $\widehat a\in \widehat A_{e}\le A^{(I)}$ and any  $\widehat w\in W_R$ such that $\e_R(\widehat w)\in R$, we have $\widehat w ^{-1} \widehat a \widehat w =\widehat a$. In turn, the latter equality reduces to $\big(\widehat w ^{-1} {\widehat a} \widehat w\big) (e) =\widehat a (e)$ since the values of $\widehat a$ and $\widehat w ^{-1} \widehat a \widehat w$ at all arguments other than $e$ are equal to $1$. Let $w$ and $a$ be preimages of $\widehat w$ and $\widehat a$ in $W$, respectively. Using (\ref{gai}), (\ref{piwaw}), and taking into account that $\e(w)e=\e_R(\widehat w)e=e$, we obtain
$$
\big(\widehat w ^{-1} \widehat a \widehat w\big)(e) = \gamma(w^{-1}aw) (e) =\pi_{e}(w^{-1}aw) =\pi_{\e(w)e}(a) = \pi_{e}(a)= \gamma(a) (e)=\widehat a (e).
$$\end{proof}

%%%%%%%%%%%%%%%%%%%%%%%%%%%%%%%%%%%%%%%%%%%%%%%%%%%%%%

\subsection{Wreath-like products associated with Dehn fillings}\label{Sec:WRDF}

%%%%%%%%%%%%%%%%%%%%%%%%%%%%%%%%%%%%%%%%%%%%%%%%%%%%%%

This section aims to discuss examples of wreath-like products that naturally occur in the
context of group theoretic Dehn filling. The main result of this section can be thought of as a generalization of \cite[Theorem 2.6]{CIOS1}. We assume the reader to be familiar with the preliminary material discussed in Section \ref{Sec:DF}.

Recall that a \emph{left transversal} of a subgroup $K$ in a group $G$ is a set that contains exactly one element from each left coset of $K$ in $G$. We denote the set of all left transversals of $K\le G$ by $LT(G,K)$.

\begin{defn}\label{CLdef}
Let $G$ be a group, $H$ a subgroup of $G$, $N\lhd H$. We say that  $(G,H,N)$ is a \textit{Cohen-Lyndon triple} if there exists $T\in LT(G, H\ll N \rr)$ such that
\begin{equation} \label{CL}
\ll N \rr=\Ast_{t\in T}tNt^{-1}.
\end{equation}
\end{defn}

\begin{ex}
Let $F_2$ be the free group of rank $2$ with basis $\{x,y\}$ and let $H=\langle x\rangle$. It is well-known that $\ll x \rr^{F_2}$ is free with basis $\{y^nxy^{-n}\}_{n\in\mathbb{Z}}$ and thus $(F_2, H,H)$ is a Cohen-Lyndon triple. More generally, $(F,H,N)$ is a Cohen-Lyndon triple for any maximal cyclic subgroup $H$ of a free group $F$ and any $N\lhd H$. This was proved by Cohen and Lyndon in \cite{CL}, hence the name.
\end{ex}

The following result is proved in \cite[Theorem 2.5]{Sun} (the reader is encouraged to review the terminology introduced in Definition \ref{Def:DF}).

\begin{thm}[Sun]\label{Thm:Sun}
Let $G$ be a group, $H$ a subgroup of $G$. Suppose that $H\h G$. Then $(G,H,N)$ is a Cohen-Lyndon triple for all sufficiently deep $N\lhd G$
\end{thm}

We now explain the relevance of Cohen-Lyndon triples to wreath-like products. Throughout the rest of this section, we employ the notation $$x^y=yxy^{-1}$$ for group elements; in particular, we have $(x^y)^z=x^{zy}$. Note that this notation differs from the widely accepted $x^y=y^{-1}xy$, which satisfies the usual law of exponents. Our choice is dictated by the decision to consider left actions in the definition of wreath-like products.

\begin{defn}\label{Def:defWGHN}
To each chain of groups $N\lhd H\le G$, we associate a group $W(G,H,N)$ as follows. Let
$$
S=\big\{ [n_1^{g_1}, n_2^{g_2}] \mid n_1,n_2\in N, \; g_1,g_2 \in G,\; g_1H\ll N\rr \ne g_2H\ll N\rr\} .
$$
We define
\begin{equation}\label{Eq:defWGHN}
W(G,H,N)= G/\langle S\rangle.
\end{equation}
\end{defn}

Note that we do not need to take the normal closure of $S$ in (\ref{Eq:defWGHN}). Indeed, for any $n_1, n_2\in N$ and any $g_1,g_2, g\in G$ such that $g_1H\ll N\rr \ne g_2H\ll N\rr$, we have $$[n_1^{g_1}, n_2^{g_2}]^g=[n_1^{gg_1}, n_2^{gg_2}]\in S$$ since $gg_1H\ll N\rr \ne gg_2H\ll N\rr$. Thus, the set $S$ is closed under conjugation by elements of $G$ and the subgroup $\langle S\rangle $ is normal in $G$.

\begin{ex}\label{Ex:ABWR}
Let $G=H\ast K$ for some groups $H$ and $K$. It is not difficult to show that $W(G,H,H)$ is naturally isomorphic to the ordinary wreath product $H\,{\rm wr}\, K$.
\end{ex}

In general, it is not even clear how to show that $W(G,H,N)$ is nontrivial. For example, if $G$ is simple we have $W(G,H,N)=\{ 1\}$ for any $H$ and $N\ne \{ 1\}$. However, we will prove the following result, which can be thought of as a generalization of Example \ref{Ex:ABWR}.

\begin{prop}\label{WRCL}
Suppose that $(G,H,N)$ is a Cohen-Lyndon triple. Then $$W(G,H,N)\in \WR (N, G/\ll N\rr \curvearrowright I),$$ where $I=G/H\ll N\rr$, and the action $G/\ll N\rr \curvearrowright I$ is induced by left multiplication.
\end{prop}

\begin{proof}
By the definition of a Cohen-Lyndon triple, there is a transversal $T\in LT(G,H\ll N \rr)$ such that $\ll N\rr =\Ast_{t\in T}N^t$. Let $I=G/H\ll N\rr$. We define a homomorphism
$$
\delta\colon \Ast_{t\in T}N^t \longrightarrow \bigoplus_{i\in I}N_i,
$$
where $N_i\cong N$, by letting $\delta \vert_{N^t}$ be any isomorphism sending $N^t$ to $N_{i}$, where $i=tH\ll N\rr$. We first prove the following.

{\noindent \bf Claim.} \emph{For any $g\in G$ and any $t\in T$, we have $\delta (gNg^{-1})= N_{gH\ll N\rr}$.}

\begin{proof}[Proof of the claim]
Indeed, let $g\in th\ll N\rr$, where $t\in T$, $h\in H$. Since $\ll N\rr $ is normal in $G$, there exists $n\in \ll N\rr $ such that $g=nth$. Further, since $N\lhd H$, we have $N^g=(N^h)^{nt}=N^{nt}$. Note that $\delta (N^t)$ is a direct summand in $\bigoplus_{i\in I}N_i$; in particular, $\delta (N^t)$ is normal in $\bigoplus_{i\in I}N_i$. This implies
$$
\delta(N^g)=\delta (N^{nt})=(\delta (N^t))^{\delta(n)}=\delta (N^t)=N_i,
$$
where $i=tH\ll N\rr=gH\ll N\rr$.
\end{proof}

Continuing with the proof of Proposition \ref{WRCL}, we observe that $\Ker(\delta)$ is generated by elements of the form $[n_1^{t_1},n_2^{t_2}]$, where $n_1,n_2\in N$ and $t_1, t_2\in T$ are distinct. Since $T\in LT(G, H\ll N\rr)$, we have $t_1H\ll N\rr \ne t_2H\ll N\rr$ and $[n_1^{t_1},n_2^{t_2}]\in S$ for any distinct $t_1,t_2\in T$, where $S$ is given by Definition \ref{Def:defWGHN}. Thus, $\Ker(\delta) \le \langle S\rangle $.

Conversely, let $[n_1^{g_1}, n_2^{g_2}]\in S$, where $n_1,n_2\in N$, $g_1,g_2 \in G$, and $g_1H\ll N\rr \ne g_2H\ll N\rr$. Using the claim proved above, we obtain
$$
\delta ([n_1^{g_1}, n_2^{g_2}])\in \delta ([N_{g_1H\ll N\rr}, N_{g_2H\ll N\rr}])=\{1\}.
$$
Thus $S\subseteq \Ker (\delta)$. Combining this with the inclusion proved in the previous paragraph, we obtain $\langle S\rangle = \Ker(\delta)$. In particular, $\Ker(\delta)$ is normal in $G$.

Passing to quotients by $\Ker(\delta)$ converts the exact sequence $$1\to \ll N\rr \to G\to G/\ll N\rr\to 1$$ to the exact sequence
$$
1\to \bigoplus_{i\in I}N_i\to W(G,H,N)\stackrel{\e}\to G/\ll N \rr \to 1.
$$
The group $G/\ll N\rr$ acts on $I$ by left multiplication; that is, an element $g\ll N\rr \in G/\ll N\rr$ sends $fH\ll N\rr \in I$ to $gfH\ll N\rr$. Since $\ll N\rr$ is normal in $G$, this action is well-defined. It remains to note that for every $i=tH\ll N\rr \in I$, where $t\in T$, and every $w=g\Ker(\delta)\in W(G,H,N)$, we have
$$
wN_iw^{-1}=\delta (gN^tg^{-1})=\delta(N^{gt})=N_{gtH\ll N\rr} = N_{\e(w)i};
$$
here the third equality follows from the Claim.
\end{proof}

For group theoretic Dehn fillings, we obtain the following.

\begin{cor}\label{Cor:WRDF}
Let $G$ be a group, $H$ a hyperbolically embedded subgroup of $G$. For any sufficiently deep $N\lhd H$, we have $$W(G,H,N)\in \WR (N, G/\ll N\rr \curvearrowright I),$$ where the action of $G/\ll N\rr $ on $I$ is transitive with stabilizers isomorphic to $H/N$.
\end{cor}

\begin{proof}The fact that $W(G,H,N)\in \WR (N, G/\ll N\rr \curvearrowright I)$ is a straightforward combination of Proposition \ref{WRCL} and Theorem \ref{Thm:Sun}. Moreover, by Proposition \ref{WRCL} the action of $G/\ll N\rr$ on $I$ is transitive with stabilizers isomorphic to $H\ll N\rr/\ll N\rr$. By part (a) of Theorem \ref{Thm:DF}, we have $H\ll N\rr/\ll N\rr\cong H/(\ll N\rr \cap H)=H/N$ for all sufficiently deep $N\lhd H$.
\end{proof}

%%%%%%%%%%%%%%%%%%%%%%%%%%%%%%%%%%%%%%%%%%%%%%%%%%%%%%%%%%%%%%%%%%%%%%%%%%%%%%%%%%%%%%%%%%%%%%%%%%%%%%%%%%%%%%%

\subsection{Cohen-Lyndon subgroups and regular wreath-like products}\label{Sec:CL}

%%%%%%%%%%%%%%%%%%%%%%%%%%%%%%%%%%%%%%%%%%%%%%%%%%%%%%%%%%%%%%%%%%%%%%%%%%%%%%%%%%%%%%%%%%%%%%%%%%%%%%%%%%%%%%%

Wreath-like products $W(G,H,N)$ considered in Corollary \ref{Cor:WRDF} are not regular. Indeed, for sufficiently deep $N\lhd H$, we have  $\ll N\rr \cap H=N\ne H$ by part (a) of Theorem \ref{Thm:DF}. Our next goal is to show that regular wreath-like products can be constructed in a similar way by considering Cohen-Lyndon subgroups (see Definition \ref{Defn:CL}). In the terminology of Section \ref{Sec:WRDF}, the definition can be reformulated as follows. 

\begin{defn}
Let $G$ be a group. We say that $H\le G$ is a \emph{Cohen-Lyndon subgroup} of $G$ if $(G,H,H)$ is a Cohen-Lyndon triple. 
\end{defn}

\begin{ex}\label{Ex:FPCL}
Let $G=A\ast_C B$ and let $H$ be a subgroup of $A$ such that $H\cap C=\{ 1\}$. Using Bass-Serre theory, it is not difficult to show that $H$ is a Cohen-Lyndon subgroup of $G$ (see, for example \cite[Ch. I, Sec. 5.5, Theorem 14]{Ser}). In particular, if $G=A\ast B$, then $A$ is a Cohen-Lyndon subgroup of $G$.
\end{ex}

We will use the simplified notation $W(G,H)$ for the group $W(G,H,H)$ associated with the chain of groups $H\lhd H\le G$ (see Definition \ref{Def:defWGHN}). Thus,
\begin{equation}\label{Eq:WGHdef1}
W(G,H)=G/\langle S\rangle,
\end{equation}
where
\begin{equation}\label{Eq:WGHdef2}
S=\{ [h_1^{g_1}, h_2^{g_2}] \mid g_1, g_2\in G,\; g_1\ll H\rr \ne g_2\ll H\rr\}.
\end{equation}
Applying Proposition \ref{WRCL} to the triple $(G,H,H)$, we obtain the following.

\begin{cor}\label{Cor:CLWR}
For any Cohen-Lyndon subgroup $H$ of a group $G$, we have $$W(G,H) \in \WR (H, G/\ll H\rr).$$
\end{cor}

\begin{rem}\label{Rem:CLWR}
It is easy to see from the proof of Proposition \ref{WRCL} that the wreath-like structure of $W(G,H)$ is the natural one. That is, the natural homomorphism $\gamma\colon G\to W(G,H)$ is injective on $H$ and sends $H$ to the summand $H_{e}$ of the base $\gamma\left(\ll H\rr\right)=\bigoplus_{i\in G/\ll H\rr}H_i$ corresponding to the trivial coset $e=\ll H\rr$.
\end{rem}

Corollary \ref{Cor:CLWR} can be used to obtain results about Cohen-Lyndon subgroups unrelated to wreath-like products. The following lemma will be used in Section \ref{Sec:MC}. Here and below, we employ the more precise notation $\ll X\rr^G$ for the normal closure of a subset $X$ in a group $G$ whenever any ambiguity is possible.

\begin{lem}\label{Lem:CLmaln}
For any Cohen-Lyndon subgroup $H$ of a group $G$, the following hold.
\begin{enumerate}
\item[(a)] $H$ is malnormal in $G$.
\item[(b)] For any subgroup $K\le G$ such that $\ll H\rr^G\le K$, we have $N_G(\ll H\rr^K)=K$.
\end{enumerate}
\end{lem}
\begin{proof}
To prove (a), we first note that $H$ is malnormal in $\ll H\rr^G$ being a free factor. Thus, we only need to show that $H\cap gHg^{-1}= \{ 1\}$ for all $g\in G\setminus \ll H\rr^G$. To this end, we consider the group $W=W(G,H)$ and let $G\stackrel{\gamma}\to W\stackrel{\e}\to G/\ll H\rr^G$ denote the natural homomorphisms. By Corollary \ref{Cor:CLWR} (and Remark \ref{Rem:CLWR}), $W$ is a wreath-like product of $H$ and $G/\ll H\rr^G$ with the base $\gamma\left(\ll H\rr^G\right) = \bigoplus_{i\in G/\ll H\rr^G} H_i$ and we have $\gamma (H)=H_{e}$, where $e=\ll H\rr^G\in I$ is the identity element of $G/\ll H\rr^G$. By the definition of a wreath-like product, we have
\begin{equation}\label{Eq:HgH}
\gamma(gHg^{-1})=\gamma(g)H_{e}\gamma(g^{-1})=H_{\e\circ\gamma(g)}.
\end{equation}
Since $g\notin \ll H\rr^G$, we have $\e\circ\gamma(g)\ne e$. Therefore, $\gamma(H\cap gHg^{-1})=H_{e}\cap H_{\e\circ\gamma(g)} =\{ 1\}$. Since the restriction of $\gamma $ to $H$ is injective, we obtain $H\cap gHg^{-1}= \{ 1\}$.

Let us prove (b). Obviously, $K\le N_G(\ll H\rr^K)$. To prove the opposite inclusion, we again pass to $W$. Let $g\in N_G(\ll H\rr^K)$. Using (\ref{Eq:HgH}), we obtain
$$
H_{\e\circ\gamma(g)}=\gamma(gHg^{-1})\le \gamma\left(\ll H\rr ^K\right)=\gamma\left(\left\langle \bigcup_{k\in K} kHk^{-1}\right\rangle\right) =\bigoplus_{i\in K/\ll H\rr^G} H_i.
$$
This implies $\e\circ \gamma(g)\in K/\ll H\rr^G$. Since $\Ker(\e\circ \gamma)=\ll H\rr^G\le K$, we obtain $g\in K$.
\end{proof}

In general, constructing Cohen-Lyndon subgroups is a rather non-trivial task. The main goal of this section is to prove the following strengthening of Theorem \ref{Thm:Fhe}. Recall that $F_n$ denotes the free group of rank $n$. For the definition of a suitable subgroup, see Section \ref{Sec:Suit}.

\begin{prop}\label{freeCL}
Let $G$ be an acylindrically hyperbolic group with $K(G)=\{1\}$. %trivial finite radical.
For every $n\in\mathbb{N}$, there exists a Cohen-Lyndon subgroup $H\cong F_n$ of $G$ such that the following hold.
\begin{enumerate}
\item[(a)] $G/\ll H\rr ^G$ is ICC and acylindrically hyperbolic.
\item[(b)] If $G$ is non-elementary relatively hyperbolic with respect to a collection of peripheral subgroups $\Hl$, we can choose $H$ inside any suitable subgroup $S$ of $G$. In addition, we can ensure the following.
    \begin{enumerate}
    \item[(b$_1$)] The restriction of the natural homomorphism $\gamma\colon G\to G/\ll H\rr ^G$ to each $H_i$ is injective and $G/\ll H\rr ^G$ is non-elementary relatively hyperbolic with respect to $\{\gamma(H_i)\}_{i\in I}$.
    \item[(b$_2$)] $\gamma (S)\ne \{ 1\}$.
    \item[(b$_3$)] If $G$ is torsion free, then so is $G/\ll H\rr ^G$.
    \end{enumerate}
\end{enumerate}
\end{prop}

We begin by establishing a transitivity property of Cohen-Lyndon triples. Throughout the rest of this section, we often consider normal closures of the same set in different groups. %To eliminate any ambiguity, we sometimes employ the notation $\ll R\rr ^G$ for the minimal normal subgroup of a group $G$ containing a subset $R\subseteq G$.

\begin{lem}\label{lem. transitivity of CL}
Let $H\le K\le G$ be groups such that $(G,K,\ll H \rr^K)$ is a Cohen-Lyndon triple and $H$ is a Cohen-Lyndon subgroup of $K$. Then $H$ is a Cohen-Lyndon subgroup of $G$.
\end{lem}

\begin{proof}
By the definition of a Cohen-Lyndon triple, there exist $T\in LT(G,K\ll H \rr^G)$ and $S\in LT(K,\ll H \rr^K)$ such that
$$
\ll H \rr^G = \Ast_{t\in T}t\ll H \rr^K t^{-1},\;\;\; {\rm and}\;\;\; \ll H \rr^K=\Ast_{s\in S}sHs^{-1}.
$$
Thus, we have
$$
\ll H \rr^G = \Ast_{t\in T,s\in S}ts H (ts)^{-1}.
$$
It suffices to prove that
\begin{equation}\label{Eq:TS}
TS\in LT(G,\ll H \rr^G).
\end{equation}

For any $g\in G$, there exist $t\in T$ and $k\in K$ such that $g\in tk\ll H\rr^G$. Further, there is $s\in S$ such that $k\in s\ll H\rr ^K$. We obtain $g\in ts\ll H\rr^K\ll H\rr^G=ts\ll H\rr^G$. Thus, $G=TS \ll H\rr^G$. Further, suppose there are $t_1,t_2\in T$ and $s_1,s_2\in S$ such that $t_1s_1\in t_2s_2\ll H\rr^G$. Then
$$
t^{-1}_2t_1\in s_2\ll H \rr^G s^{-1}_1= s_2s^{-1}_1\ll H \rr^G\subseteq K\ll H \rr^G.
$$
Therefore, $t_1=t_2$. This implies $s_2 ^{-1}s_1\in \ll H\rr ^G$. Obviously, we also have $s_2^{-1}s_1\in K$. By \cite[Proposition 6.1 (a)]{Sun}, we have $K\cap \ll H\rr^G=\ll H\rr ^K$ as $(G,K,\ll H\rr^K)$ is a Cohen-Lyndon triple. Therefore, $s_2^{-1}s_1\in \ll H\rr ^K$, which implies $s_1=s_2$.  This completes the proof of \eqref{Eq:TS} and the lemma.
\end{proof}

Our strategy of constructing a Cohen--Lyndon subgroup is first look at a hyperbolically embedded free subgroup $K$ of an acylindrically hyperbolic group $G$ with $K(G)=1$, and then look at a free factor $H$ of $K$ and argue that $H$ is a Cohen--Lyndon subgroup of $G$ by using Lemma \ref{lem. transitivity of CL}. The lemma below is the second step of this strategy.

\begin{lem}\label{Lem:F7n}
Let $k\ge 7$ and $n$ be positive integers, $F_{kn}$ the free group with the basis $f_1, \ldots, f_{kn}$. For any finite subset $\mathcal{F}\subset F_{kn}\setminus\{1\}$, there is $M\in \NN$ such that for any $m>M$, the subgroup $H$ generated by elements
$$
r_i= f_{(i-1)k+1} f_{(i-1)k+2}^m\cdots f_{(i-1)k+k}^m, \;\;\; i=1, \ldots, n
$$
is a free factor of $F_{kn}$ of rank $n$, the quotient group $F_{kn}/\ll H\rr$ is free of rank $(k-1)n$, and we have $\ll H \rr \cap \mathcal{F}=\emptyset$. In particular, $H$ is a Cohen-Lyndon subgroup of $F_{kn}$.
\end{lem}

\begin{proof}
Since the set $\{r_i\}^n_{i=1}$ satisfies the classical $C^\prime (1/6)$ small cancellation condition, we have  $\ll H \rr\cap \mathcal{F}=\emptyset$ for all sufficiently large $m$ by \cite[Chapter V Theorem 4.4]{LS}. It remains to note that $F_{kn}$ decomposes as a free product $F_{kn}=H\ast L$, where $L$ is the subgroup of $F_{kn}$ generated by the set $\{ f_{(i-1)k+j}\mid i=1,\ldots,n,\; j=2,\ldots,k\}$. In particular, $H$ is a Cohen-Lyndon subgroup of $F_{kn}$ (see Example \ref{Ex:FPCL}).
\end{proof}

\begin{proof}[Proof of Proposition \ref{freeCL}]
By Theorem \ref{Thm:Fhe}, there is a subgroup $K\cong F_{7n}$ such that $K\hookrightarrow_h G$. By Theorems \ref{Thm:DF} and \ref{Thm:Sun}, there exists a finite set $\mathcal{F}\subset K\smallsetminus\{1\}$ such that if $N\lhd K$ and $N\cap \mathcal{F}=\emptyset$, then $(G,K,N)$ is a Cohen-Lyndon triple, $K\cap \ll N\rr^G =N$, and $K/N\h G/\ll N\rr^G$.

By Lemma \ref{Lem:F7n}, there exists a Cohen-Lyndon subgroup $H\cong F_n$ of $K$ such that $\ll H \rr^K\cap \mathcal{F}=\emptyset $ and $K/\ll H \rr^K$ is a non-cyclic free group. In particular, $(G,K,\ll H \rr^K)$ is a Cohen-Lyndon triple. By Lemma \ref{lem. transitivity of CL}, $H$ is a Cohen-Lyndon subgroup of $G$. Since $K/\ll H\rr^K$ is non-cyclic free and hyperbolically embedded in $G/\ll H\rr^G$, the latter quotient group is acylindrically hyperbolic by Theorem \ref{Thm:ah}. Note also that the existence of a non-trivial, torsion-free, hyperbolically embedded subgroup implies the ICC condition by Proposition \ref{Prop:maln} and Lemma \ref{Lem:DFICC}.

If $G$ is hyperbolic relative to a collection of subgroups $\Hl$ and $S$ is a suitable subgroup of $G$, we can choose $K$ inside $S$ and assume that $G$ is hyperbolic relative to $\Hl\cup\{K\}$ by Corollary \ref{Cor:Fhe}. In these settings, Theorem \ref{Thm:DF} yields relative hyperbolicity of $G/\ll H\rr^G$ with respect to the collection of isomorphic images of subgroups $H_i$ and the free group $K/\ll H\rr^K$. Applying Corollary \ref{Cor:rhh}, we conclude that $G/\ll H\rr^G$ is hyperbolic relative to $\{\gamma(H_i)\}_{i\in I}$. We have $S\not\subset \ll H\rr ^G $ since $K\cap \ll H\rr ^G= K\cap \ll H\rr^K \ne K$. Finally, if $G$ is torsion-free, then so is $G/\ll H\rr^G$  by Theorem \ref{Thm:DF} (d).
\end{proof}

We mention one simple application.

\begin{thm}\label{Thm:AHQ}
Let $G$ be a non-elementary hyperbolic (respectively, acylindrically hyperbolic) group. For every finitely generated group $A$, there exists a quotient $W$ of $G$ such that $W\in \WR (A,B)$ for some non-elementary hyperbolic (respectively, acylindrically hyperbolic) group $B$.
\end{thm}

\begin{proof}
The quotient group $G/K(G)$ is acylindrically hyperbolic whenever $G$ is and we have $K(G/K(G))=\{ 1\}$ (see \cite[Lemma 5.10]{Hull}). Further, if $G$ is non-elementary hyperbolic, then so is $G/K(G)$ as it is quasi-isometric to $G$. Thus, passing from $G$ to $G/K(G)$, we can assume $K(G)=\{1\}$ without loss of generality. 

Let $A=F_n/N$ for some $n\in \NN$ and $N\lhd F_n$. Combining Corollary \ref{Cor:CLWR} and Proposition \ref{freeCL}, we obtain a quotient group $U$ of $G$ such that $U\in \WR (F_n, B)$, where $B$ is hyperbolic or acylindrically hyperbolic whenever so is $G$. Applying now Lemma \ref{A/N} to the wreath product $U$ and $N\lhd F_n$, we obtain the required group $W$.
\end{proof}

In particular, we obtain examples of wreath-like products with strong fixed point properties.  Recall that a group $G$ is said to have {\it property $FL^p$} if every affine isometric action of $G$ on an $L^p$-space has a fixed point. Note that having property $FL^p$ for all $p\geq 1$ is a significant strengthening of property (T) since the latter is equivalent to property $FL^2$ by the Delorme-Guichardet theorem.

%(see the discussion after Theorem \ref{Thm:WLP_hyp} for the definition of property $FL_p$).

\begin{cor}\label{Cor:FLp}
For every finitely generated group $A$, there is an infinite group $B$ and a regular wreath-like product $W\in \WR(A,B)$ such that $W$ has property $FL^p$ for all $p\ge 1$. In particular, $W$ has property (T) of Kazhdan.
\end{cor}

\begin{proof}[Proof of Corollary \ref{Cor:FLp}]
It was proved in \cite{MO} that there exists an acylindrically hyperbolic group $G$ having property $FL^p$ for all $p\ge 1$. Applying Theorem \ref{Thm:AHQ} to this group $G$ yields the desired result.
\end{proof}

%%%%%%%%%%%%%%%%%%%%%%%%%%%%%%%%%%%%%%%%%%%%%%%%%%%%%%%%%%%%%%%%%%%%%%%%%%%%%%%%%%%%%%%%%%%%%%%%%%%%%%%%%%%%%%%

\subsection{Automorphisms of wreath-like products}\label{Sec:WRAut}

%%%%%%%%%%%%%%%%%%%%%%%%%%%%%%%%%%%%%%%%%%%%%%%%%%%%%%%%%%%%%%%%%%%%%%%%%%%%%%%%%%%%%%%%%%%%%%%%%%%%%%%%%%%%%%%

Let $W\in \WR (A,B\curvearrowright I)$, where $A$, $B$ are some groups and $B\curvearrowright I$ is an action of $B$ on some set $I$. Throughout this section, we keep the notation introduced in Definition \ref{wlp}. In particular, we denote by $A^{(I)}$ the base of $W$ and by $\e\colon W\to B$ the natural homomorphism with kernel $A^{(I)}$.

If  $A^{(I)}$ is characteristic in $W$, then every automorphism $\alpha \in Aut(W)$ induces an automorphism $\phi(\alpha) \in Aut(B)$ such that the following diagram is commutative:
$$
\begin{tikzcd}
W\ar{d}{\e}\ar{r}{\alpha} & W \ar{d}{\e} \\
B\ar{r}{\phi(\alpha)}& B
\end{tikzcd}
$$
Equivalently, we have
$$
\phi(\alpha) (wA^{(I)}) = \alpha (w) A^{(I)}
$$
for all $w\in W$. The rule $\alpha \mapsto \phi(\alpha)$ induces a homomorphism $Aut(W)\to Aut(B)$. The main result of this section -- Proposition \ref{Prop:WRAut} -- allows us to control the kernel of this map under certain additional assumptions. We begin with a lemma that provides a sufficient condition for $\phi$ to be well-defined.

\begin{lem}\label{Lem:char}
Suppose that $W\in \WR (A,B\curvearrowright I)$, where $A$ is amenable and $B$ is an ICC acylindrically hyperbolic  group acting on a set $I$. Then $A^{(I)}$ is a characteristic subgroup of $W$.
\end{lem}

\begin{proof}
Let $\alpha \in Aut(W)$, $Q=\e\circ \alpha (A^{(I)})$, where $\e\colon W\to B$ is the canonical homomorphism. It suffices to show that $Q=\{ 1\}$. Arguing by contradiction, assume that $Q\ne \{ 1\}$. Then $|Q|=\infty $ since $B$ is ICC. By \cite[Corollary 1.5]{Osi16}, the class of acylindrically hyperbolic groups is closed under passing to infinite normal subgroups.  Therefore, $Q$ is acylindrically hyperbolic. The group $Q$ is also amenable being a homomorphic image of the amenable group $A^{(I)}$. However, this contradicts Theorem \ref{Thm:Fhe} which gives non-abelian free subgroups of $Q$.
\end{proof}

The following result is a consequence of Corollary \ref{CS}, which will be proved later using an extension of Popa's cocycle superrigidity theorem recorded in \cite{CIOS1}.

\begin{prop}\label{Prop:WRAut}
Let $A$ and $B$ be any countable groups and let $W\in \WR(A,B\curvearrowright I)$, where the action $B\curvearrowright I$ has infinite orbits. Assume that $W$ has property (T) and $A^{(I)}$ is a characteristic subgroup of $W$. Then every automorphism $\alpha \in Aut(W)$ such that $\phi(\alpha)\in Inn(B)$ belongs to $Inn (W)$. In particular, $\Ker (\phi)\le Inn (W)$.
\end{prop}

\begin{proof}
Assume that $\phi(\alpha)\in Inn(B)$.
After composing $\alpha$ with an inner automorphism of $W$, we may assume that $\phi(\alpha)=id_B$.
In other words, $\alpha(w)w^{-1}\in A^{(I)}$ for every $w\in W$. The map $c\colon W\rightarrow A^{(I)}$ sending each $w\in W$ to $\alpha(w)w^{-1}$ is a $1$-cocycle, where $A^{(I)}$ is endowed with the $W$-module structure induced by the conjugation action of $W$. Equivalently, the map $ d\colon W\rightarrow\mathcal U(\text{L}(A^{(I)}))$ given by $d_w=u_{c(w)}$ is a $1$-cocycle for the action $W\curvearrowright^{\sigma}\text{L}(A^{(I)})$ defined by $\sigma_w=\text{Ad}(u_w)$.

By Corollary \ref{CS}, $w$ is cohomologous to a character $\eta$ of $W$. Thus, there is $u\in\mathcal U(\text{L}(A^{(I)}))$ such that $d_w=\eta_wu\sigma_w(u)^*$, for every $w\in W$. Since $\sigma_w(u)=u_wuu_w^*$ and $d_w=u_{\alpha(w)}u_w^*$, we get that $u_{\alpha(w)}=\eta_wuu_wu^*$ and so $u_{\alpha(w)}uu_w^*=\eta_wu$, for every $w\in W$. Let $u=\sum_{a\in A^{(I)}}\zeta_au_a$ be the Fourier decomposition of $u$ and let $a\in A^{(I)}$ with $\zeta_a\not=0$. Then the set $\{\alpha(w)a w^{-1}\mid w\in W\}$ is finite. Thus, there is a finite index subgroup $W_0\le W$, which can be taken normal, such that $\alpha(w)=awa^{-1}$, for every $w\in W_0$. Thus, after replacing  $\alpha$ by $\text{Ad}(a^{-1})\circ\alpha$,  we may assume that $\alpha(w)=w$, for every $w\in W_0$, while still having that $\alpha(w)w^{-1}\in A^{(I)}$ for every $w\in W$. Let $w\in W$. Since $W_0\lhd W$ is normal, $\alpha(w)^{-1}v\alpha(w)=\alpha(w^{-1}vw)=w^{-1}vw$ for every $v\in W_0$. Hence, $\alpha(w)w^{-1}\in A^{(I)}$ commutes with $W_0$. Since $W_0$ has finite index in $W$ and the action $B\curvearrowright I$ has infinite orbits, it follows that $\alpha(w)w^{-1}=1$ for all $w\in W$.
\end{proof}

%%%%%%%%%%%%%%%%%%%%%%%%%%%%%%%%%%%%%%%%%%%%%%%%%%%%%%%%%%%%%%%%%%%%%%%%%%%%%%%%%%%%%%%%%%%%%%%%%%%%%%%%%%%%%%%
%%%%%%%%%%%%%%%%%%%%%%%%%%%%%%%%%%%%%%%%%%%%%%%%%%%%%%%%%%%%%%%%%%%%%%%%%%%%%%%%%%%%%%%%%%%%%%%%%%%%%%%%%%%%%%%
\section{Constructing groups with prescribed outer automorphisms}\label{Sec:App}

%%%%%%%%%%%%%%%%%%%%%%%%%%%%%%%%%%%%%%%%%%%%%%%%%%%%%%%%%%%%%%%%%%%%%%%%%%%%%%%%%%%%%%%%%%%%%%%%%%%%%%%%%%%%%%%
%%%%%%%%%%%%%%%%%%%%%%%%%%%%%%%%%%%%%%%%%%%%%%%%%%%%%%%%%%%%%%%%%%%%%%%%%%%%%%%%%%%%%%%%%%%%%%%%%%%%%%%%%%%%%%%

\subsection{Embedding countable groups into quotients of residually finite groups}

The main goal of this section is to prove the following embedding result, which plays a crucial role in the proof of Corollary \ref{Out}.

\begin{prop}\label{Prop:Embed}
Every countable group can be embedded in a quotient group $S/M$, where both $S$ and $M$ are finitely generated and $S$ is residually finite and torsion-free.
\end{prop}

The proof of the proposition is based on small cancellation theory and deep results on groups acting on cubical complexes obtained by Wise, Haglund-Wise, and Agol. 

We begin by recalling the definition of the small cancellation condition used below. Let $W$ be a word in the alphabet $\{x_1, x_1^{-1}, x_2, x_2^{-1}, \ldots \} $. One says that $W$ is \emph{reduced} if it contains no subwords of the form $x_ix_i^{-1}$ and $x_i^{-1}x_i$ and \emph{cyclically reduced} if every cyclic shift of $W$ is reduced. When talking about group presentations, we always assume that relators are cyclically reduced. A group presentation
\begin{equation}\label{eq-pres}
G=\langle X \mid \mathcal R\rangle.
\end{equation}
is said to be \emph{symmetric} if for every word $R\in \mathcal R$, all cyclic shifts of $R^{\pm 1}$ belong to $\mathcal R$. If the presentation (\ref{eq-pres}) is not symmetric, by its \emph{symmetrization} we mean the presentation obtained by adding all cyclic shifts of $R^{\pm 1}$ for all $R\in \mathcal R$ to the set of relators. A symmetric presentation (\ref{eq-pres}) satisfies the $C^\prime(\lambda)$ \emph{small cancellation condition} for some $\lambda \in [0, 1]$ if any common initial subword $U$ of two distinct words $R,S\in \mathcal R$ satisfies
$$
\| U\| < \lambda \min\{ \| R\|, \, \| S\|\},
$$
where $\| \cdot \| $ denotes the number of letters in the corresponding word. A non-symmetric group presentation satisfies $C^\prime(\lambda)$ if so does its symmetrization.

In the paper \cite{Wis}, Wise proved that every finitely presented $C^\prime(1/6)$ group acts geometrically (i.e., properly cocompactly) on a CAT(0) cube complex. Agol \cite{A} showed that every hyperbolic group $G$ acting geometrically on a CAT(0) cubical complex satisfies certain additional conditions, which imply that $G$ is residually finite by the work of Haglund and Wise \cite{HW}. Since finitely presented $C^\prime(1/6)$ groups are hyperbolic, we obtain the following.

\begin{thm}[Wise, Haglund-Wise, Agol]\label{Thm:WHWA}
Every finitely presented $C^\prime(1/6)$ group is residually finite.
\end{thm}

It is worth noting that Theorem \ref{Thm:WHWA} does not extend to finitely generated (but not necessarily finitely presented) $C^\prime (1/6)$ groups.

\begin{proof}[Proof of Proposition \ref{Prop:Embed}]
Consider the group
$$
H=\left\langle a_1,a_2,a_3,a_4 \mid a_\ell^{-1}a_{\ell+1}a_\ell=a_{\ell+1}^2\;\; (\ell=1,\ldots, 4)\right\rangle.
$$
Here and below, the indices of elements $a_1, \ldots, a_4$ are always taken modulo $4$, i.e., $a_{\ell+1}=a_1$ for $\ell=4$. Higman \cite{Hig} proved that the group $H$ has no non-trivial finite quotients. On the other hand, $H$ is acylindrically hyperbolic by \cite[Corollary 4.26]{MO}. Every acylindrically hyperbolic group is $SQ$-universal by \cite[Theorem 2.33]{DGO}. This means that every countable group embeds in a quotient of $H$. Let $K$ denote a quotient of $H$ containing a given countable group $C$.

We adopt the following convention. Given a group presentation $O=\langle X\mid \mathcal P\rangle $ and a (finite or infinite) set of words $R_1, R_2, \ldots $ in the alphabet $X$, we write $\langle O\mid R_1, R_2, \ldots \rangle $ for the presentation obtained from $O$ by adding $R_1, R_2, \ldots $ to the set of relators. By abuse of notation, we do not distinguish between group presentations and groups represented by them.

By construction, we have
$$
K= \langle H \mid  R_1,\; R_2, \ldots \rangle,
$$
where $R_1, R_2, \ldots$ is a (possibly infinite) set of words in the alphabet $\{ a_1^{\pm 1}, a_2^{\pm 1}, a_3^{\pm 1}, a_4^{\pm 1}\}$.
Let $$X=\{ a_1,\ldots, a_4, x_1, \ldots, x_{20}, y_1, \ldots, y_{20}\}.$$ It is easy to find words $S_{ij}$, $T_{ij}$, $U_{k}$, $V_{\ell}$ in the alphabet $\{ x_1, \ldots, x_{20}\}$, where $i, \ell=1, \ldots, 4$ and $j, k=1, \ldots , 20$ such that the presentation
\begin{equation}\label{Eq:S0}
S_0= \left\langle \;X \;\;\left| \;
\begin{array}{l}
a_i^{-1}x_ja_iS_{ij},\;\; a_ix_ja_i^{-1}T_{ij} \;\;\; (i=1, \ldots, 4, \; j=1, \ldots , 20)
\\ y_kU_{k}\;\;\;(k=1, \ldots , 20)\\
a_\ell^{-1}a_{\ell+1}a_\ell a_{\ell+1}^{-2}V_\ell\;\;\; (\ell=1, \ldots, 4)
\end{array}
\right.\right\rangle
\end{equation}
satisfies $C^\prime(1/6)$. (E.g., one can take $S_{ij}$, $T_{ij}$, $U_{k}$, $V_{\ell}$ to be any distinct words from the set $\{ x_1^{n}x_2^{n} \ldots x_{20}^{n} \mid n\in \NN\}$; the verification of the small cancellation condition is straightforward.) Let $M_0$ denote the subgroup of $S_0$ generated by $\{x_1, \ldots, x_{20}, y_1, \ldots, y_{20}\}$. Relations in the first and the second rows of (\ref{Eq:S0}) ensure that $$M_0=\langle x_1, \ldots, x_{20}, y_1, \ldots, y_{20}\rangle =\langle x_1, \ldots, x_{20}\rangle\lhd S_0;$$ relations in the third row yield $S_0/M_0\cong H$. Since $H$ has no non-trivial finite quotients, we have the following:
\begin{enumerate}
\item[($\ast$)] $M_0$ maps surjectively onto every finite quotient of $S_0$.
\end{enumerate}

We construct the required group $S$ as a quotient of $S_0$ by induction. Let $K_0=S_0$ and let $n\in \NN$. Assume that we have already constructed a group presentation
\begin{equation}\label{Sn-1}
S_{n-1}=\langle S_0\mid P_1, \ldots , P_{n-1}\rangle
\end{equation}
(the set of relators is empty for $n=1$) and a finite index normal subgroup $K_{n-1}$ of $S_{n-1}$ such that the following conditions hold.

\begin{enumerate}
\item[(a)] The presentation (\ref{Sn-1}) satisfies $C^\prime (1/6)$.
\item[(b)] Let $R$ be one of the relators $P_1,\ldots,P_{n-1}$ or a relator of the presentation \eqref{Eq:S0}. Then every subword of a cyclic shift of $R^{\pm 1}$ of length at least $\| R\|/6$ contains either two consecutive letters from the alphabet $\{ a_1,\ldots, a_4, x_1, \ldots, x_{20}\}^{\pm 1}$ or a subword of the form $(y_1y_2^qy_3)^{\pm 1}$ for some $q\in \ZZ\setminus \{ 0\}$.
\item[(c)] Any non-trivial element of $K_{n-1}$ has length at least $n-1$ with respect to $X$.
\end{enumerate}

The group $S_{n}$ is obtained from $S_{n-1}$ as follows. By Theorem \ref{Thm:WHWA}, $S_{n-1}$ is residually finite. Therefore,  there exists a finite index normal subgroup $L_{n-1}$ of $S_{n-1}$ such that $L_{n-1}\le K_{n-1}$  and $L_{n-1}$ contains no non-trivial elements of length less than $n$. Suppose that $$R_n=a_{i_1}^{\alpha_1}\ldots a_{i_r}^{\alpha_r},$$ where $i_1, \ldots, i_r\in \{ 1,\ldots ,4\}$ and $\alpha_1, \ldots, \alpha_r=\pm 1$. By ($\ast$) there exists a word $x_{j_1}^{\xi_1} \ldots x^{\xi_s}_{j_s}$, where $j_1, \ldots, j_s\in \{ 1, \ldots, 20\}$ and $\xi_1, \ldots, \xi_s=\pm 1$, such that the word
$$
W_n=a_{i_1}^{\alpha_1}\ldots a_{i_r}^{\alpha_r}x_{j_1}^{\xi_1} \ldots x^{\xi_s}_{j_s}
$$
represents $1$ in the finite quotient group $S_{n-1}/L_{n-1}$. Further, we can choose a natural number $m$ satisfying the following two conditions:
\begin{enumerate}
\item[(+)] $m$ is divisible by $|S_{n-1}/L_{n-1}|$;
\item[(++)] for any $i\in \{ 1, \ldots, 20\}$, no cyclic shift of a relation in (\ref{Sn-1}) contains a subword of the form $y_i^{\pm m}$.
\end{enumerate}
For every $t\in \NN$, we let
$$
Z_t= y_1^t\ldots y_{20}^t.
$$
Let $S_n$ be the group presentation obtained from $S_{n-1}$ by adding the relation
\begin{equation}\label{Eq:Pn}
P_n=a_{i_1}^{\alpha_1}\, Z_{m}\,a_{i_2}^{\alpha_2}\,Z_{2m}\,\ldots \,a_{i_r}^{\alpha_r}\,Z_{rm}\,x_{j_1}^{\xi_1}\,Z_{(r+1)m}\, x_{j_2}^{\xi_2}\,Z_{(r+2)m}\,\ldots\, x^{\xi_s}_{j_s}\,Z_{(r+s)m}.
\end{equation}
Condition (++) and parts (a), (b) of the inductive assumption easily imply that $S_n$ satisfies $C^\prime(1/6)$.  Note also that condition (b) for the word $P_n$ holds by inspection.

Let $K_n$ denote the image of $L_{n-1}$ in $S_n$. By (+), $Z_t$ represents $1$ in $S_{n-1}/L_{n-1}$ whenever $t$ is a multiple of $m$. Since $W_n$ represents $1$ in $S_{n-1}/L_{n-1}$, the word $P_n$ has the same property; equivalently, $P_n$ represents an element of $L_{n-1}$ in $S_{n-1}$. Thus, the kernel of the natural homomorphism $\gamma_n\colon S_{n-1}\to S_n$ satisfies
\begin{equation}\label{Kern}
\Ker (\gamma_n)=\ll P_n\rr^{S_{n-1}} \le L_{n-1}\le K_{n-1}.
\end{equation}
By the choice of $L_{n-1}$, the shortest non-trivial element of $K_n$ has length at least $n$. This concludes the inductive step. We also note that
\begin{equation}\label{Kn}
K_n= \gamma_n(L_{n-1})\le \gamma_n(K_{n-1}).
\end{equation}

Let
$$
S=\langle S_0\mid P_1, P_2, \ldots \rangle
$$
be the inductive limit of the sequence $S_0\stackrel{\gamma_1}\to S_1\stackrel{\gamma_2}\to \ldots $. Let also $M$ denote the image of $M_0$ is $S$. We claim that $S$ is residually finite.

Indeed, let $W$ be a word in the alphabet $X\cup X^{-1}$ representing a non-trivial element $s\in S$ and let $n=\| W\|$. By part (c) of the inductive assumption, $W$ represents a non-trivial element in $S_{n+1}/K_{n+1}$. An easy induction using (\ref{Kern}) and (\ref{Kn}) shows that the kernel of the natural homomorphism $S_{n+1}\to S$ is contained in $K_{n+1}$. Therefore, there is a natural surjection $S\to S_{n+1}/K_{n+1}$ sending $s$ to the non-trivial element of $S_{n+1}/K_{n+1}$ represented by $W$. Thus, $S$ is residually finite.

The presentation of the group $S$ given above satisfies $C^\prime(1/6)$ and it is easy to see that no relation is a proper power. By the main result of \cite{Lip}, this implies that $S$ is torsion-free. Since $M_0$ is generated by $\{ x_i, y_i\mid i=1, \ldots, 20\}$, a presentation of the group $S/M$ can be obtained from the presentation (\ref{Eq:S0}) by adding all relations $P_n=1$ ($n\in \NN$) and relations $x_i=y_i=1$ ($i=1, \ldots, 20$). Using the obvious Tietze transformations to remove the generators  $x_i$ and $y_i$ from this presentation, yields the presentation $\langle H \mid  R_1,\; R_2, \ldots \rangle$ of the group $K$ (note that, after deleting all letters $x_i$ and $y_i$ in the word $P_n$, we obtain the word $R_n$ for all $n\in \NN$). Thus, $S/M \cong K$. Since $K$ contains $C$, the proposition is proved.
\end{proof}

%%%%%%%%%%%%%%%%%%%%%%%%%%%%%%%%%%%%%%%%%%%%%%%%%%%%%%%%%%%%%%%%%%%%%%%%

\subsection{The main construction}\label{Sec:MC}

%%%%%%%%%%%%%%%%%%%%%%%%%%%%%%%%%%%%%%%%%%%%%%%%%%%%%%%%%%%%%%%%%%%%%%%%

We begin with an auxiliary definition.

\begin{defn}\label{Def:IndOut}
With each short exact sequence of groups
\begin{equation}\label{Eq:SES}
1\to N\to G\stackrel{\e}\to Q\to 1,
\end{equation}
we associate a homomorphism
$$
\iota\colon G\to Aut(N)
$$
as follows. For every $g\in G$, $\iota(g)$ is the automorphism of $N$ given by $n\mapsto gng^{-1}$ for all $n\in N$.
\end{defn}

We summarize some necessary results about the map $\iota$.

\begin{lem}\label{Lem:OutAH}
Let $G$ be an ICC acylindrically hyperbolic group and let $N$ be a non-trivial normal subgroup. In the notation of Definition \ref{Def:IndOut}, we have the following.
\begin{enumerate}
\item[(a)] The map $\iota\colon G\to Aut(N)$ associated with the short exact sequence (\ref{Eq:SES}) is injective.
\item[(b)] If, in addition, $G$ is non-elementary relatively hyperbolic and $N$ has property (T), then $|Aut(N):\iota (G)|<\infty$.
\end{enumerate}
\end{lem}

\begin{proof}
Part (a) is proved in \cite{CIOS2}.

Part (b) easily follows from the results of \cite{BS}. Indeed, assume that $G$ is properly relatively hyperbolic, $N$ has property (T), and $|Aut(N):\iota (G)|=\infty$. Composing automorphisms of $N$ with the embedding $N\le G$, we obtain infinitely many pairwise $G$-non-conjugate homomorphisms $\gamma_i\colon N\to G$. Note that $N$ must be infinite as $G$ is ICC. An infinite normal subgroup of a properly relatively hyperbolic group cannot belong to a conjugate of a parabolic subgroup by Proposition \ref{Prop:maln}. In the language of \cite{BS}, this means that our homomorphisms $\gamma_i$ have non-parabolic images. This allows us to apply \cite[Theorem 1.2]{BS}, which yields an action of $N$ on an $\mathbb R$-tree without a global fixed point. However, every action of a property $(T)$ group on an $\mathbb R$-tree has a global fixed point \cite{HV}. This contradiction proves the claim.
\end{proof}

We are now ready to prove the main result of this section.

\begin{thm}\label{Thm:MC}
For any a countable group $C$, there exist a countable group $B$, a countable set $I$,  and $2^{\aleph_0}$ pairwise non-isomorphic finitely generated groups $\{U_j\}_{j\in J}$ such that the following conditions hold.
\begin{enumerate}
\item[(a)] For any $j\in J$,  $U_j\in \WR(A_j,B\curvearrowright I)$, where $A_j$ is abelian and $B\curvearrowright I$ is a faithful action with infinite orbits.

\item[(b)] $B$ is an ICC subgroup of a finitely generated, relatively hyperbolic group with residually finite peripheral subgroups.

\item[(c)] For any $j\in J$, $U_j$ has property (T), $[U_j,U_j]=U_j$, and $Out(U_j)\cong C$.
\end{enumerate}
\end{thm}

\begin{proof}
Let $C$ be a countable group. By Proposition \ref{Prop:Embed} applied to the countable group $C\oplus \ZZ$, there exist a finitely generated, torsion-free, residually finite group $S$ and a finitely generated subgroup $M\lhd S$ such that $S/M$ contains a subgroup isomorphic to $C\oplus \ZZ$. It follows that $S/M$ contains an infinite index subgroup isomorphic to $C$. (The only reason we use $C\oplus \ZZ$ instead of $C$ when applying Proposition \ref{Prop:Embed} is to ensure the infinite index condition.)

There exists a non-elementary hyperbolic group $H$ with Kazhdan's property (T). For example, we can let $H$ to be a cocompact lattice of $Sp(n,1)$ for $n\geqslant 2$. Using Selberg's Lemma and passing to a finite index subgroup if necessary, we may assume that $H$ is torsion-free. By Corollary \ref{Cor:TrAb}, we may further assume that $H$ has trivial abelianization. Let $G$ and $N\lhd G$ be the group and the subgroup provided by Proposition \ref{Rips} for groups $M\lhd S$ and $N$. Note that $G$ is torsion-free and $N$ has property (T) and trivial abelianization being a quotient of $H$. Since $G/N\cong S/M$, $C$ embeds in $G/N$. Henceforth, we think of $C$ as a subgroup of $G/N$ and denote by $C_0$ its full preimage in $G$.

By Lemma \ref{Lem:suit}, $N$ is a suitable subgroup of $G$. By Proposition \ref{freeCL}, we can choose an element $x\in N$ of infinite order such that $\langle x\rangle $ is a Cohen-Lyndon subgroup of $G$, $\ll x\rr ^G\ne N$, and the quotient group $\overline G= G/\ll x\rr^G$ is torsion-free, ICC, and non-elementary hyperbolic relative to a subgroup isomorphic to $S$. In particular, the peripheral subgroup of $\overline G$ is residually finite.

Let $V=W(G, \langle x\rangle)$ be the group defined by (\ref{Eq:WGHdef1}) and (\ref{Eq:WGHdef2}). By Corollary \ref{Cor:CLWR}, we have $V\in \WR(\ZZ, \overline G)$. In what follows, we will always think of $V$ as a wreath-like product with the structure described by Corollary \ref{Cor:CLWR} and Remark \ref{Rem:CLWR}.  

Further, let $B$ (respectively, $K$) denote the image of $N$ (respectively, $C_0$) in $\overline G$ and let $U$ (respectively, $W$) denote the image of $N$ (respectively, $C_0$) in $V$. We have $B\ne \{ 1\}$ since $\ll x\rr ^G\ne N$. By Lemma \ref{Lem:OutAH}, $\overline G$ embeds as a finite index subgroup in $L=Aut(B)$ via the map $\iota \colon \overline G\to L$ defined by $\iota (g) (b)= gbg^{-1}$ for all $g\in \overline G$ and $b\in B$. Note that $\iota (B)=Inn (B)$. The following commutative diagram visualizes our notation.

$$
\begin{tikzcd}
N \arrow[d, phantom, sloped, "\le"]\ar[r]&U\ar[r]\arrow[d, phantom, sloped, "\le"] & B=N/\ll x\rr^G\arrow[d, phantom, sloped, "\le"]\\
C_0\ar[r]\arrow[d, phantom, sloped, "\le"] &W\ar[r]\arrow[d, phantom, sloped, "\le"] & K = C_0/\ll x\rr^G\arrow[d, phantom, sloped, "\le"]\\
G\ar[r] & V \ar[r] & \overline G = G/\ll x\rr^G \arrow[d, hook, "\iota"]\\
&& L=Aut (B)\\
\end{tikzcd}
$$

Since $C$ is of infinite index in $S/M$, we have $|\overline G:K|=|G:C_0|=\infty $. Since $x\in N$, $W$ contains the image of $x$ in $V$. Therefore, $W$ contains the base of the wreath-like product $V$ by Remark \ref{Rem:CLWR}. Applying By Lemma \ref{WRsubgr} to $W\le V$, we obtain $W\in \WR (A, K)$, where $A$ is the free abelian group of countably infinite rank. 

We identify $A$ with the subgroup $A_1$ of the base of $W$. Let $J$ denote the set of all non-empty subsets of the set of prime numbers. For every $j=\{ p_1, p_2, \ldots \}\in J$, we let $N_j$ be a subgroup of $A$ such that
\begin{equation}\label{Eq:ANj}
A_j=A/N_j\cong \ZZ_{p_1}\oplus \ZZ_{p_2}\oplus \ldots .
\end{equation}
Let also $$W_j=W/\ll N_j\rr ^W.$$ By Lemma \ref{A/N}, we have $W_j\in \WR (A_j, K)$. 

Since non-elementary relative hyperbolicity is preserved under quasi-isometries \cite{Dru} and $|L:\iota(\overline G)|<\infty$, the group $L$ is non-elementary relatively hyperbolic. We want to show that $L$ is ICC. As usual, we denote by $K(L)$ the finite radical of $L$. For every $k\in K(L)$, and any $b\in \iota(B)$, we have $[k,b]=k^{-1}b^{-1}kb \in \iota(B)\cap K(L)$ since both $\iota(B)$ and $K(L)$ are normal in $L$. Since $\iota(B)\cong B$ is torsion-free and $K(L)$ is finite, we obtain $[k,b]=1$. This means that every element of $K(L)$ acts as the trivial automorphism of $B$. Thus, $K(L)=\{1\}$ and $L$ is ICC by Theorem \ref{Thm:HypICC}.

The subgroup $\iota(B)=Inn(B)$ is suitable in $L$ by Lemma \ref{Lem:suit}. Therefore, by part (b) of Proposition \ref{freeCL}, we can find an infinite cyclic subgroup $R\le B$ such that $\iota(R)$ is a Cohen-Lyndon subgroup of $L$. We consider the corresponding quotients $W_{j,R}=W_j/N_{R}$ of wreath-like products $W_j$ defined as in Lemma \ref{Lem:N(R)}. Let $I$ denote the set of left cosets $K/R$. By Lemma \ref{Lem:N(R)} (b), we have $W_{j,R}\in \WR_0(A_j, K \curvearrowright I)$, where the action of $K\curvearrowright I$ is by left multiplication. We denote by $\e_j\colon W_{j,R}\to K$ the canonical homomorphism. Letting $U_j$ be the natural image of $U$ in $W_{j,R}$, we obtain another commutative diagram:

$$
\begin{tikzcd}
1\ar[r] & A_j^{(I)}\ar[r]\arrow[d,equal] & U_j\ar[r]\arrow[d, phantom, sloped, "\le"] & B\ar[r]\arrow[d, phantom, sloped, "\le"] &1\\
1\ar[r] & A_j^{(I)}\ar[r] & W_{j,R}\ar[r, "\e_j"] & K\ar[r] &1\\
\end{tikzcd}
$$

We claim that $U_j$, $B$, and $I$ satisfy properties (a)--(c). From now on, we fix some $j\in J$. We have $U_{j}\in \WR(A_j,B\curvearrowright I)$ by  Lemma \ref{Lem:WRsubgr0}.  Note that $B$ is infinite since it is non-trivial and torsion-free. By \cite[Corollary 1.5]{Osi16}, any infinite normal subgroup of an acylindrically hyperbolic group is acylindrically hyperbolic. Therefore, $B$ is acylindrically hyperbolic. Since $R\cong \ZZ$, we have $|B:R|=\infty $ (see Theorem  \ref{Thm:class}). This means that $B \curvearrowright I$ has infinite orbits. Furthermore, the action is faithful by Lemma \ref{Lem:FA} since $R$ is not acylindrically hyperbolic.  This proves claim (a).

Since $B$ is torsion-free, it is ICC by Theorem \ref{Thm:HypICC}. As explained above, $\overline G$ is hyperbolic relative to a residually finite subgroup. Note also that $\overline G$ is finitely generated being a quotient of a finitely generated group $G$. Thus, we obtain (b).

The groups $U_j$ have property (T) and trivial abelianization being quotients of $N$. It remains to show that $Out(U_j)\cong C$. To this end, we define a homomorphism
$$
\psi_j\colon Out(U_j) \to Out(B)=L/\iota(B)
$$
as follows. Let $\phi_j\colon Aut(U_j)\to L=Aut (B)$ be the homomorphism defined in Section \ref{Sec:WRAut}. That is, for every $\alpha \in Aut(U_j)$ and $u\in U_j$, we have
\begin{equation}\label{Eq:Defphi}
\phi_j(\alpha) \left(uA_j^{(I)}\right) = \alpha (u) A_j^{(I)}.
\end{equation}
The map $\phi_j $ is well-defined since $A_j^{(I)}$ is characteristic in $U_j$ by Lemma \ref{Lem:char}. Note that $\phi_j (Inn (U_j))=Inn (B)$. This allows us to define $\psi_j$ by the rule
\begin{equation}\label{Eq:Defpsi}
\psi_j (\alpha Inn (U_j))= \phi_j(\alpha) Inn (B)
\end{equation}
for every $\alpha \in Aut(U_j)$. Clearly, $\psi_j $ is a homomorphism.

We first show that $\psi_j $ is injective. Suppose that $\alpha Inn (U_j)\in \Ker \psi_j$ for some $\alpha \in Aut(U_j)$. Then $\phi_j (\alpha)\in Inn(B)$ by (\ref{Eq:Defpsi}). By Proposition \ref{Prop:WRAut}, we have $\alpha \in Inn (U_j)$. Therefore, $\Ker \psi_j$ is trivial.

Further, we show that $\psi_j (Out(U_j))\le \iota (K)/\iota(B)$. Consider any $\alpha Inn (U_j)\in Out(U_j)$, where $\alpha \in Aut(U_j)$. Let also
$$
S_j=\left\langle \bigcup_{a\in A_j^{(I)}\setminus\{ 1\}} C_{W_{j,R}} (a) \right\rangle.
$$
For every $a\in A_j^{(I)}\setminus\{ 1\}$, we have
\begin{equation}\label{Eq:CWK}
\e_j(C_{W_{j,R}}(a))\le Stab_{K}(i)
\end{equation}
for some $i\in I$ by Lemma \ref{Lem:centr} and Lemma \ref{Lem:CLmaln} (a). Conversely, for every $i\in I$, we have $Stab_{K}(i)\le \e_j(C_{W_{j,R}}(a))$ for some $a\in A_j^{(I)}\setminus\{ 1\}$ since the wreath-like product $W_{j,R}$ is untwisted (see Definition \ref{Def:UWR} and part (b) of Lemma \ref{Lem:N(R)}). This implies that
\begin{equation}\label{Eq:ejSj}
\e_j(S_j) = \left\langle \bigcup_{i\in I}  Stab_K(i)\right\rangle = \left\langle \bigcup_{k\in K}  kRk^{-1}\right\rangle =\ll R\rr^K.
\end{equation}
Since $Stab_{K}(i)$ is a conjugate of $R\le B$ in $K$ and $B\lhd K$, (\ref{Eq:CWK}) implies $\e_j(C_{W_{j,R}}(a)) \le B$ for all $a\in A_j^{(I)}\setminus\{ 1\}$. Hence, $C_{W_{j,R}}(a)\le U_j$ and $C_{W_{j,R}}(a)=C_{U_j}(a)$ for all $a\in A_j^{(I)}\setminus\{ 1\}$. This means that
$$
S_j=\left\langle \bigcup_{a\in A_j^{(I)}\setminus\{ 1\}} C_{U_j} (a) \right\rangle.
$$
In particular, $S_j$ is a characteristic subgroup of $U_j$ since so is $A_j^{(I)}$. Combining this with (\ref{Eq:ejSj}) and (\ref{Eq:Defphi}), we obtain
$$
\phi_j(\alpha)(\ll R\rr^K)= \phi_j(\alpha) (\e_j(S_j)) = \e_j\circ \alpha (S_j) =\e_j(S_j) = \ll R\rr^K.
$$
Thus, $\phi_j(\alpha) \in N_{\rm L}(\iota(\ll R\rr ^{K}))= N_{\rm L}(\ll \iota(R)\rr ^{\iota (K)})$. Note that $\ll \iota(R)\rr ^L\le \iota(B) \le \iota(K)$ since $\iota(B)\lhd L$. By Lemma \ref{Lem:CLmaln} (b), we obtain $N_{\rm L}(\ll \iota(R)\rr ^{\iota (K)})=\iota(K)$. Thus, $\phi_j(\alpha) \in \iota(K)$. By (\ref{Eq:Defpsi}), we have $\psi_j (\alpha Inn(U_j))\in \iota (K)/\iota(B)$.

Next, we show that $\iota (K)/\iota(B) \le \psi_j (Out(U_j))$. Let $k$ be an arbitrary element of $K$ and let $w\in W_{j,R}$ be a preimage of $k$ under $\e_j$. By construction, we have $U_j\lhd W_{j,R}$. Therefore, conjugation by $w$ defines an automorphism $\alpha\in Aut(U_j)$. Clearly, $\phi_j (\alpha)$ is the automorphism of $B$ corresponding to conjugation by $k$, i.e., $\phi_j(\alpha)=\iota (k)$. Using (\ref{Eq:Defpsi}), we obtain $\iota (K)/\iota(B) \le \psi_j (Out(U_j))$.

Summarising the previous three paragraphs, we obtain that $\psi_j$ defines an isomorphism $Out(U_j)\to \iota (K)/\iota(B)\cong K/B\cong C_0/N\cong C$. This completes the proof of part (c).

Finally, using (\ref{Eq:ANj}) and torsion-freeness of $B$, it is easy to see that $U_j$ contains an element of a prime order $p$ if and only if $p\in J$. In particular, $U_i\not\cong U_j$ for $i\ne j$.
\end{proof}

\begin{rem}
    It would be interesting to ask whether the condition that groups $U_j$ are pairwise non-isomorphic can be upgraded to being pairwise non-quasi-isometric. The approach to constructing non-quasi-isometric groups suggested in \cite{MOW} can be relevant here.
\end{rem}

%%%%%%%%%%%%%%%%%%%%%%%%%%%%%%%%%%%%%%%%%%%%%%%%%%%%%%%%%%%%%%%%%%%%%%%%%%%%%%%%%%%%%%%%%%%%%%%%%%%%%%%%%%%%%%%
%%%%%%%%%%%%%%%%%%%%%%%%%%%%%%%%%%%%%%%%%%%%%%%%%%%%%%%%%%%%%%%%%%%%%%%%%%%%%%%%%%%%%%%%%%%%%%%%%%%%%%%%%%%%%%%

%%%%%%%%%%%%%%%%%%%%%%%%%%%%%%%%%%%%%%%%%%%%%%%%%%%%%%%%%%%%%%%%%%%%%%%%%%%%%%%%%%%%%%%%%%%%%%%%%%%%

\section{Preliminaries on von Neumann algebras}\label{prelimvN}

%%%%%%%%%%%%%%%%%%%%%%%%%%%%%%%%%%%%%%%%%%%%%%%%%%%%%%%%%%%%%%%%%%%%%%%%%%%%%%%%%%%%%%%%%%%%%%%%%%%%%%%%%%%%%%%
%%%%%%%%%%%%%%%%%%%%%%%%%%%%%%%%%%%%%%%%%%%%%%%%%%%%%%%%%%%%%%%%%%%%%%%%%%%%%%%%%%%%%%%%%%%%%%%%%%%%%%%%%%%%%%%

\subsection{Tracial von Neumann algebras}\label{tracialvN}

%%%%%%%%%%%%%%%%%%%%%%%%%%%%%%%%%%%%%%%%%%%%%%%%%%%%%%%%%%%%%%%%%%%%%%%%%%%%%%%%%%%%%%%%%%%%%%%%%%%%%%%%%%%%%%%

We start this section by recalling some terminology involving tracial von Neumann algebras. We refer the reader to \cite{AP} for more information.
A {\it tracial von Neumann algebra} is a pair $(\M,\tau)$ consisting of a von Neumann algebra $\M$  and a normal  faithful tracial state $\tau\colon\M\rightarrow\mathbb C$. We denote by L$^2(\M)$ the Hilbert space obtained as the closure of $\M$ with respect to the $2$-norm given by $\|x\|_2=\sqrt{\tau(x^*x)}$. We always assume that $\M$ is {\it separable}, i.e., that L$^2(\M)$ is a separable Hilbert space.
We denote by $\sU(\M)$ the group of {\it unitaries}  of $\M$.
%and by $(\M)_1=\{x\in \M\mid \|x\|\leq 1\}$ the {\it unit ball} of $\M$. %and by $M_{+}=\{x\in M\mid x\geq 0\}$ the cone of {\it positive} elements of $M$.
 %We denote by ${Aut}(\M)$ the group of $\tau$-preserving automorphisms of $\M$.  For $u\in\sU(\M)$,  the {\it inner} automorphism $\text{Ad}(u)$ of $\M$ is given by $\text{Ad}(u)(x)=uxu^*$.
%By von Neumann's bicommutant theorem, for any set $X\subset \M$ closed under adjoint, $X''\subset \M$ is the smallest von Neumann subalgebra which contains $X$.
%For von Neumann subalgebras $\mathcal P,\Q\subset \M$, we denote by $\P\vee \Q=(\P\cup \Q)''$ the von Neumann algebra they generate.
For a set $I$, we denote by $(\M^I,\tau)$ the tensor product of tracial von Neumann algebras $\,\overline{\otimes}\,_{i\in I}(\M,\tau)$. Given a subset $J\subset I$, we view $\M^J$ as a subalgebra of $\M^I$ by identifying it with $(\,\overline{\otimes}\,_{i\in J}\M)\,\overline{\otimes}\,(\,\overline{\otimes}\,_{i\in I\setminus J}\mathbb C1)$. For $i\in I$, we denote $\M^{\{i\}}$ by $\M^i$.

Let $\Q\subset \M$ be a von Neumann subalgebra, which we always assume to be unital. %(We write $Q\subset pMp$ to indicate that $Q$ is a unital subalgebra of $pMp$, for a projection $p\in M$).
We denote %by $\Q'\cap \M=\{\text{$x\in \M\mid xy=yx$, for all $y\in \Q$}\}$ the {\it relative commutant} of $\Q$ in $\M$, and
by $\sN_\M(\Q)=\{u\in\sU(\M)\mid u\Q u^*=\Q\}$ the {\it normalizer} of $\Q$ in $\M$.
%The {\it center} of $\M$ is given by $\sZ(\M)=\M'\cap \M$.We say that $\M$ is a {\it factor} if $\sZ(\M)=\mathbb C1$ and
We say that $\Q\subset \M$ is an {\it irreducible subfactor} if $\Q'\cap \M=\mathbb C1$. We say that $\Q$ is {\it regular} in $\M$ if $\sN_\M(\Q)''=\M$. If $\Q\subset \M$ is regular and maximal abelian, we call it a {\it Cartan subalgebra}.

{\it Jones' basic construction} $\langle \M,e_\Q\rangle$ is defined as the von Neumann subalgebra of $\mathbb B(\text{L}^2(\M))$ generated by $\M$ and the orthogonal projection $e_\Q$ from L$^2(\M)$ onto L$^2(\Q)$. The basic construction $\langle \M,e_\Q\rangle$ has a faithful semi-finite trace given by $\text{Tr}(xe_\Q y)=\tau(xy)$, for every $x,y\in \M$. For $p\in\{1,2\}$, we define the L$^p$-norm on $\langle \M,e_\Q\rangle$ by $\|T\|_p=\text{Tr}(|T|^p)^{1/p}$, where $|T|=(T^*T)^{1/2}$. 
We denote by L$^p(\langle \M,e_\Q\rangle)$ the Banach space obtained by completing $\{T\in \langle \M,e_\Q\rangle\mid \|T\|_p<\infty\}$ with respect to $\|\cdot\|_p$. Note that L$^2(\langle \M,e_\Q\rangle)$ is a Hilbert space which carries a natural $\M$-bimodule structure. 
Also, note that every $T\in$ L$^p(\langle \M,e_\Q\rangle)$ can be viewed as a closed operator on L$^2(\langle \M,e_\Q\rangle)$ which is affiliated with $\langle \M,e_\Q\rangle$.
%We denote by L$^1(\langle \M,e_\Q\rangle)$ the associated  L$^1$-space obtained as the closure of the linear span of $\M e_\Q \M$ with respect to the norm $\|T\|_{1,\text{Tr}}=\text{Tr}(|T|)$.
We also denote by $E_\Q\colon\M\rightarrow \Q$ the unique $\tau$-preserving {\it conditional expectation} onto $\Q$.

If $\Q\subset \M$ are II$_1$ factors, the {\it Jones index} $[\M:\Q]$ of the inclusion $\Q\subset \M$ is the dimension of L$^2(\M)$ as a left $\Q$-module \cite{Jo83}. By \cite[Theorem 2.2, Remark 2.4]{PP86}, this notion can be extended to inclusions of tracial von Neumann algebras $\Q\subset \M$ by letting $[\M:\Q]$ be the smallest $\lambda>0$ such that $\lambda\|E_\Q(x)\|_2^2\geq \|x\|_2^2$, for every positive $x\in \M$.

A tracial von Neumann algebra $(\M,\tau)$  is called {\it amenable} if there exists a sequence $\xi_n\in \text{L}^2(\M)\otimes \text{L}^2(\M)$ such that $\langle x\xi_n,\xi_n\rangle\rightarrow\tau(x)$ and $\|x\xi_n-\xi_nx\|_2\rightarrow 0$, for every $x\in \M$.
Let $\P\subset p\M p$ be a von Neumann subalgebra.
Following Ozawa and Popa \cite[Section 2.2]{OP07} we say that $\P$ is  {\it amenable relative to $\Q$ inside $\M$} if there exists a sequence $\xi_n\in \text{L}^2(\langle \M,e_\Q\rangle)$ such that $\langle x\xi_n,\xi_n\rangle\rightarrow\tau(x)$, for every $x\in p\M p$, and $\|y\xi_n-\xi_ny\|_2\rightarrow 0$, for every $y\in \P$. By \cite[Proposition 2.4]{PV11}, we may additionally assume that $\xi_n\in \text{L}^2(\langle \M,e_\Q\rangle)$ is positive and $\|\langle\cdot \xi_n,\xi_n\rangle-\tau(\cdot)\|=\|\langle\xi_n\cdot,\xi_n\rangle-\tau(\cdot)\|\rightarrow 0$.

%We say that $\P$ is {\it strongly nonamenable relative to $\Q$ inside $\M$} if there does not exist any nonzero projection $p'\in \P'\cap p\M p$ such that $\P p'$ is amenable relative to $\Q$ inside $\M$.

%We recall the following remark from \cite[Remark 3.1]{CIOS1}.

%\begin{rem}\label{TT} By the proof of \cite[Theorem 2.1]{OP07}, in the definition of relative amenability we may take $\xi_n=\zeta_n^{1/2}$, for some positive element $\zeta_n\in\text{L}^1(\langle \M,e_\Q\rangle)$.   Thus, we have $\langle \xi_n x,\xi_n\rangle=\text{Tr}(\zeta_nx)=\langle x\xi_n,\xi_n\rangle\rightarrow\tau(x)$, for every $x\in \M$. Using a standard convexity argument (see the proof of \cite[Lemma 13.3.11]{AP}), we may  find $\eta_n\in \text{L}^2(\langle \M,e_\Q\rangle)^{\oplus\infty}$ such that $\|\langle\cdot\eta_n,\eta_n\rangle-\tau(\cdot)\|\rightarrow 0$, $\|\langle\eta_n\cdot,\eta_n\rangle-\tau(\cdot)\|\rightarrow 0$ and $\|y\eta_n-\eta_ny\|_2\rightarrow 0$, for every $y\in \P$ .\end{rem}

Following \cite[Proposition 4.1]{Po01b}, we say that $\Q\subset \M$ has the {\it relative property (T)} if for every $\varepsilon>0$, we can find a finite set $F\subset \M$ and $\delta>0$ such that if $\mathcal H$ is an $\M$-bimodule and $\xi\in\mathcal H$ satisfies $\|\langle\cdot\xi,\xi\rangle-\tau(\cdot)\|\leq\delta,\|\langle\xi\cdot,\xi\rangle-\tau(\cdot)\|\leq\delta$ and $\|x\xi-\xi x\|\leq\delta$, for every $x\in F$, then there exists $\eta\in\mathcal H$ such that $\|\eta-\xi\|\leq\varepsilon$ and $y\eta=\eta y$, for every $y\in \Q$.

%%%%%%%%%%%%%%%%%%%%%%%%%%%%%%%%%%%%%%%%%%%%%%%%%%%%%%%%%%%%%%%%%%%%%%%%%%%%%%%%%%%%%%%%%%%%%%%%%%%%%%%%%%%%%%%

\subsection {Intertwining-by-bimodules}

%%%%%%%%%%%%%%%%%%%%%%%%%%%%%%%%%%%%%%%%%%%%%%%%%%%%%%%%%%%%%%%%%%%%%%%%%%%%%%%%%%%%%%%%%%%%%%%%%%%%%%%%%%%%%%%

We recall from  \cite [Theorem 2.1, Corollary 2.3]{Po03} Popa's {\it intertwining-by-bimodules} theory.

\begin{thm}[\cite{Po03}]\label{corner} Let $(\M,\tau)$ be a tracial von Neumann algebra and $\P\subset p\M p, \Q\subset \M$ be von Neumann subalgebras. Then the following conditions are equivalent.
\begin{enumerate}
\item[(a)] There exist projections $p_0\in \P, q_0\in \Q$, a $*$-homomorphism $\theta\colon p_0\P p_0\rightarrow q_0\Q q_0$  and a nonzero partial isometry $v\in q_0\M p_0$ such that $\theta(x)v=vx$, for all $x\in p_0\P p_0$.

\item[(b)] There is no sequence $u_n\in\sU(\P)$ satisfying $\|E_\Q(x^*u_ny)\|_2\rightarrow 0$, for all $x,y\in p\M$.

\item[(c)] There exists a projection $e\in \P'\cap p\langle \M,e_\Q\rangle p$ such that  $0<\emph{Tr}(e)<\infty$.
\end{enumerate}
%Conditions (a) and (b) are equivalent in general, and (a), (b) and (c) are equivalent if $q=1$.
\end{thm}

If (a)-(c) hold true,  we write $\P\prec_{\M}\Q$ and say that {\it a corner of $\P$ embeds into $\Q$ inside $\M$.}
Moreover, if $\P p'\prec_{\M}\Q$ for any nonzero projection $p'\in \P'\cap p\M p$, we write $\P\prec^{\text s}_{\M}\Q$.

We also recall  two key conjugacy results for Cartan subalgebras \cite[Theorem A.1]{Po01b} and regular irreducible subfactors \cite[Lemma 8.4]{IPP05} (stated here in the more general form from \cite[Lemma 4.1]{VV14}).

\begin{thm}[\cite{Po01b}]\label{Po01b} Let $\M$ be a II$_1$ factor and $\P,\Q\subset \M$ be Cartan subalgebras. If $\P\prec_\M\Q$, then there is $u\in\sU(\M)$ such that $u\P u^*=\Q$.
\end{thm}

\begin{thm}[\cite{IPP05}]\label{IPP05} Let $\M$ be a II$_1$ factor and $\P,\Q\subset \M$ be regular, irreducible subfactors such that the countable groups $\sN_\M(\P)/\sU(\P)$, $\sN_\M(\Q)/\sU(\Q)$ are ICC. If $\P\prec_\M\Q$ and $\Q\prec_\M\P$,  then there is $u\in\sU(\M)$ such that $u\P u^*=\Q$.
\end{thm}

We end this section by using the structure theorems for normalizers in crossed products arising from actions of hyperbolic \cite{PV12} and free product groups \cite{Io12, Va13} to establish the following result. Let $G$ be a countable group. Recall that a {\it trace preserving} action $G\curvearrowright^{\sigma}(\P,\tau)$ is an action of $G$ on $\P$ through $\tau$-preserving $*$-automorphisms, where $(\P,\tau)$ is a tracial von Neumann algebra.

%We end this section with two consequences of the structure theorems for normalizers in crossed products arising from actions of biexact, weakly amenable groups due to Popa-Vaes, \cite{PV12} and of free product groups due to the second author, \cite{Io12}:

\begin{thm}\label{relativeT}
Let $G,H$ be countable groups, $\delta\colon G\rightarrow H$ a homomorphism and $G\curvearrowright (\Q,\tau)$ a trace preserving action on a tracial von Neumann algebra  $(\Q,\tau)$. Let $\M=\Q\rtimes G$.

Assume that $H$ is a non-elementary subgroup of a hyperbolic group. Then the following hold:
\begin{enumerate}
\item [(a)]\emph{\cite{CIOS1}}
Let $\P\subset p\M p$ be a von Neumann subalgebra which is amenable relative to $\Q\rtimes\ker(\delta)$ inside $\M$ and let $\Nn=\sN_{p\M p}(\P)''$. If there is a von Neumann subalgebra $\R\subset \Nn$ with the relative property (T) and $\R\nprec_{\M}\Q\rtimes\ker(\delta)$, then  $\P\prec_{\M}^{\text s}\Q\rtimes\ker(\delta)$.
\item [(b)] Let  $\A,\B \subset p \M p$ be commuting von Neumann subalgebras  and let $\Nn=\sN_{p\M p}(\A\vee \B)''$. If  there is a von Neumann algebra $\R\subset  \Nn$ with the relative property (T) and such that $\R\nprec_{\M} \Q\rtimes \ker(\delta)$, then $\A\prec_{\M} \Q\rtimes {\rm ker}(\delta)$ or $\B\prec_{\M} \Q\rtimes {\rm ker}(\delta)$.
\end{enumerate}
Assume that $\delta(G)=H= H_1\ast H_2$, where $|H_1|\geq 2$ and $|H_2|\geq 3$. Then the following hold:
\begin{enumerate}
    \item [(c)] Let $\P\subset p\M p$ be a von Neumann subalgebra which is amenable relative to $\Q\rtimes\ker(\delta)$ inside $\M$ and  $\Nn=\mathscr N_{p\M p}(\P)''$.  If $\Nn\subset p\M p$ has finite index, then $\P\prec_{\M}^{\rm s} \Q \rtimes\ker( \delta)$.
    \item [(d)] Let  $\A,\B \subset p\M p$ be commuting von Neumann subalgebras  and let $\Nn=\mathscr N_{p\M p}(\A\vee \B)''$. Assume that $\Nn \subset p \M p$ has finite index. Then $\A\prec_\M \Q \rtimes {\rm ker}(\delta)$ or $\B\prec_\M \Q\rtimes {\rm ker}(\delta)$.     \end{enumerate}

\end{thm}

The proof of Theorem \ref{relativeT} uses the following easy fact which we prove for completeness.

\begin{lem}\label{amenT}
Let $(\M,\tau)$ be a tracial von Neumann algebra and $\P,\Q\subset\M$ be von Neumann subalgebras. Assume that $\P$ is amenable relative to $\Q$ inside $\M$ and that $\P\subset\M$ has the relative property (T).
Then $\P\prec_{\M}\Q$.
\end{lem}
\begin{proof}
 Since $\P$ is amenable relative to $\Q$ inside $\M$, we can find a sequence $\xi_n\in \text{L}^2(\langle \M,e_{\Q}\rangle)$ with $\|\langle \cdot\xi_n,\xi_n\rangle-\tau(\cdot)\|,\|\langle\xi_n\cdot,\xi_n\rangle-\tau(\cdot)\|\rightarrow 0$, and $\|y\xi_n-\xi_n y\|_2\rightarrow 0$, for every $y\in\P$.
 Since  $\P\subset \M$ has the relative property (T), we get that  there is a nonzero $\xi\in \text{L}^2(\langle \M,e_{\Q}\rangle)$ such that $y\xi=\xi y$, for every $y\in\P$.
Then $\zeta=\xi^*\xi\in  \text{L}^1(\langle \M,e_{\Q}\rangle )$ is a nonzero positive operator that is affiliated with $\langle \M,e_{\Q}\rangle$. Hence, $\zeta$ has a spectral decomposition $\zeta=\int_0^\infty z\;\text{d}E(z)$, with its spectral projections belonging to $\langle \M,e_{\Q}\rangle$.
 Let $t>0$ such that the spectral projection $a={\bf 1}_{[t,\infty)}(\zeta)$ corresponding to the interval $[t,\infty)$  is nonzero. 
Since $\zeta$ satisfies $y\zeta=\zeta y$, for every $y\in\P$, we have that $a\in \P'\cap \langle \M,e_{\Q}\rangle$. As $ta\leq\zeta$, we get that $\text{Tr}(a)\leq{\text{Tr}(\zeta)}/{t}<\infty$.  Theorem \ref{corner} implies that $\P\prec_{\M}\Q$.
\end{proof}

\begin{proof}[Proof of Theorem \ref{relativeT}] Since part (a) was proved in \cite[Theorem 3.10]{CIOS1}, we only prove (b)-(d).   Let $(u_g)_{g\in G}\subset\sU(\M)$ and $(v_h)_{h\in H}\subset\sU(\text{L}(H))$ be the canonical unitaries. Following \cite[Section 3]{CIK13}, define $\Delta\colon \M\rightarrow \M\,\overline{\otimes}\,\text{L}(H)$ by letting $\Delta(xu_g)=xu_g\otimes v_{\delta(g)}$, for every $x\in \Q$ and $g\in G$.
Write $\M\,\overline{\otimes}\,\text{L}(H)=\M\rtimes H$, where $H$ acts trivially on $\M$.
We will use the following fact proved in \cite[Proposition 3.4]{CIK13}.

\begin{fact}\label{CIK}
If $\Ss\subset q\M q$ is a von Neumann subalgebra such that $\Delta(\Ss)\prec_{\M\,\overline{\otimes}\,\text{L}(H)}\M\,\overline{\otimes}\,\text{L}(\Sigma)$, for some subgroup $\Sigma<H$, then $\Ss\prec_\M\Q\rtimes\delta^{-1}(\Sigma)$.
\end{fact}
(b). Since $H$ is a subgroup of a hyperbolic group it is biexact \cite{BrOz08} and weakly amenable \cite{Oz05}. Let $\C\subset \Delta(\A)$ be an arbitrary amenable von Neumann subalgebra. Since $\C$ and $\Delta(\B)$ commute, by applying \cite[Theorem 1.4]{PV12}
to the crossed product von Neumann algebra $\M\, \overline\otimes \, \text{L}(H)=(\M\, \overline\otimes \, 1)\rtimes H$ (where we consider the trivial action of $H$)
we get that either  1) $\C \prec_{\M\,\overline{\otimes}\,\text{L}(H)} \M\,\overline{\otimes}\, 1$ or 2) $ \Delta(\B)$ is amenable relative to $\M\,\overline{\otimes}\, 1$ inside $\M\, \overline\otimes \, \text{L}(H)$. If 1) holds for all such $\C$, then \cite[Corollary F.14]{BrOz08}  implies that 3) $\Delta(\A)\prec_{\M\,\overline{\otimes}\,\text{L}(H)} \M\,\overline{\otimes}\, 1$. If 2) holds, by applying \cite[Theorem 1.4]{PV12} twice,  we get that either 4) $\Delta(\B) \prec_{\M\, \overline\otimes \, \text{L}(H)}\M\,\overline{\otimes}\, 1$  or 5) $\Delta (\Nn)$ is amenable relative to $\M \,\overline{\otimes}\, 1$ inside $\M\,\overline{\otimes}\, \text{L}(H)$.  Using Fact \ref{CIK}, we see that 3) and 4) imply that $\A\prec_{\M} \Q\rtimes {\rm ker}(\delta)$ and $\B\prec_{\M} \Q\rtimes {\rm ker}(\delta)$, respectively, yielding the desired conclusion.

 Finally, assume 5) holds.  Since  $\R\subset \Nn$ has the relative property (T),   %by \cite[Proposition 4.7]{Po01b}, 
 so does $\Delta(\R )\subset\Delta(\Nn )$. By Lemma \ref{amenT}, we get that $\Delta(\R )\prec_{\M\,\overline{\otimes}\,\text{L}(H)}\M\,\overline{\otimes}\, 1$.  Applying Fact \ref{CIK} again, we get that $\R\prec_{\M}\Q\rtimes\ker (\delta)$, a contradiction.

(c). We will prove first that $\P\prec_{\M}\Q\rtimes\ker(\delta)$. Since $\Delta(\P)\subset \M\, \overline\otimes \, \text{L}(H)$ is amenable relative to $\M\,\overline{\otimes}\,1$, by \cite[Theorem 1.6]{Io12} or \cite[Theorem A]{Va13},  one of the following must hold: 1) $\Delta(\P)\prec_{\M\, \overline\otimes \, \text{L}(H)} \M\,\overline{\otimes}\, 1$, 2) $\Delta(\Nn)\prec_{\M\,\overline{\otimes}\,\text{L}(H)} \M\, \overline\otimes \, \text{L}(H_i)$ for some $1\leq i\leq 2$, or 3) $\Delta(\Nn)$ is amenable relative to $\M\,\overline{\otimes}\, 1$ inside $\M\, \overline\otimes \, \text{L}(H)$. By Fact \ref{CIK}, 1) gives that $\P\prec_\M \Q\rtimes\ker(\delta)$.
If 2) holds, since  $\Delta(\Nn)\subset \Delta(p\M p)$ has finite index, we get that $\Delta(p\M p)\prec_{\M\, \overline\otimes \, \text{L}(H)} \M\, \overline\otimes \, \text{L}(H_i)$ and Fact \ref{CIK} implies that $\M\prec_{\M} \Q\rtimes\delta^{-1}(H_i)$. This yields that $[H:H_i]=[G:\delta^{-1}(H_i)]<\infty$, which is a contradiction. Finally, assume 3) holds. As $\Delta(\Nn)\subset \Delta(p\M p)$ has finite index,  $\Delta(p\M p)$ is amenable relative to $\Delta(\Nn)$ inside $\M\,\overline{\otimes}\,\text{L}(H)$ and \cite[Proposition 2.4]{OP07} gives that $\Delta (p\M p)$ is amenable relative to $\M \,\overline{\otimes}\, 1$. Then \cite[Proposition 3.5]{CIK13} implies that $\delta(G)=H$ is amenable, a contradiction.

To see that $\P\prec_{\M}^{\text s}\Q\rtimes\ker(\delta)$, let $z\in\mathscr N_{p\M p}(\P)'\cap p\M p$ be a nonzero projection.  Since $\mathscr N_{p\M p}(\P)z\subset\sN_{z\M z}(\P z)$, the inclusion $\mathscr N_{z\M z}(\P z)''\subset z\M z$ has finite index. By the above, we get that $\P z\prec_{\M}\Q\rtimes\ker(\delta)$, and \cite[Lemma 2.4(2)]{DHI16} implies that $\P\prec_{\M}^{\text s}\Q\rtimes\ker(\delta)$.

(d). Let $\C\subset \Delta(\A)$ be an arbitrary amenable von Neumann subalgebra. By  \cite[Theorem A]{Va13} (see also \cite[Theorem 1.6]{Io12})  one of the following conditions must hold: 1) $\C\prec_{\M\, \overline\otimes \, \text{L}(H)} \M\,\overline{\otimes}\, 1$, 2) $\Delta(\B)\prec_{\M\,\overline{\otimes}\,\text{L}(H)} \M\, \overline\otimes \, \text{L}(H_i)$, for some $1\leq i\leq 2$, or 3) $\Delta(\B)$ is amenable relative to $\M\,\overline{\otimes}\, 1$ inside $\M\, \overline\otimes \, \text{L}(H)$. If 1) holds for all such $\C$, then \cite[Corollary F.14]{BrOz08}  implies that 4) $\Delta(\A)\prec_{\M\,\overline{\otimes}\,\text{L}(H)} \M\,\overline{\otimes}\, 1$.
If 2) holds, then using \cite[Theorem 1.1]{IPP05} twice we get that either 5) $\Delta(\B)\prec_{\M\,\overline{\otimes}\,\text{L}(H)}\M\,\overline{\otimes}\,1$ or 6) $\Delta(\Nn)\prec_{\M\,\overline{\otimes}\,\text{L}(H)}\M\,\overline{\otimes}\,\text{L}(H_i)$.
 If 3) holds then combining \cite[Theorem A]{Va13} with \cite[Theorem 1.1]{IPP05} implies that either 5), 6) or 7) $\Delta(\Nn)$ is amenable relative to  $\M\,\overline{\otimes}\, 1$ inside $\M\, \overline\otimes \, \text{L}(H)$, hold. By Fact \ref{CIK}, 4) and 5) imply the conclusion, while
as in the proof of part (c), both 6) and 7) lead to a contradiction.\end{proof}

%%%%%%%%%%%%%%%%%%%%%%%%%%%%%%%%%%%%%%%%%%%%%%%%%%%%%%%%%%%%%%%%%%%%%%%%%%%%%%%%%%%%%%%%%%%%%%%%%%%%%%%%%%%%%%%

%%%%%%%%%%%%%%%%%%%%%%%%%%%%%%%%%%%%%%%%%%%%%%%%%%%%%%%%%%%%%%%%%%%%%%%%%%%%%%%%%%%%%%%%%%%%%%%%%%%%%%%%%%%%%%%

\subsection{Cocycle superrigidity}\label{Sec:CSR}

%%%%%%%%%%%%%%%%%%%%%%%%%%%%%%%%%%%%%%%%%%%%%%%%%%%%%%%%%%%%%%%%%%%%%%%%%%%%%%%%%%%%%%%%%%%%%%%%%%%%%%%%%%%%%%%

The main goal of this subsection is to establish a cocycle superrigidity result (Corollary \ref{CS}) that will be used in the proof of Theorem \ref{symmetries'}. To this end, we first first recall some terminology concerning actions and cocycles, and then  several results stated in \cite{CIOS1}.

Let $G\curvearrowright^{\sigma}(\P,\tau)$ be a trace preserving action of a countable group $G$ on a tracial von Neumann algebra $(\P,\tau)$.
A {\it $1$-cocycle for $\sigma$} is a map $w\colon G\rightarrow\sU(\P)$ such that $w_{gh}=w_g\sigma_g(w_h)$, for every $g,h\in G$. Any character $\eta\colon G\rightarrow\mathbb T$ gives a (trivial) $1$-cocycle for $\sigma$.
Two cocycles $w,w'\colon G\rightarrow\sU(\P)$ are called {\it cohomologous} if there is $u\in\sU(\P)$ such that $w'_g=u^*w_g\sigma_g(u)$, for every $g\in G$.
Let $p\in \P$ be a projection. A {\it generalized $1$-cocycle for $\sigma$ with support projection $p$} is a map $w\colon G\rightarrow \P$ such that $w_gw_g^*=p,w_g^*w_g=\sigma_g(p)$ and $w_{gh}=w_g\sigma_g(w_h)$, for every $g,h\in G$.

Following Krogager and Vaes \cite[Definition 2.5]{KV15},  a trace preserving action $G\curvearrowright^{\sigma}(\P^I,\tau)$ is said to be {\it built over} an action $G\curvearrowright I$  if  $\sigma_g(\P^i)=\P^{g\cdot i}$, for all $g\in G$ and $i\in I$.
By \cite[Lemma 3.4]{CIOS1}, such actions appear naturally in the context of wreath-like products:

 \begin{lem} [\cite{CIOS1}]\label{coind}
Let $A,B$ be countable groups and $G\in\mathcal W\mathcal R(A,B\curvearrowright I)$, where $B\curvearrowright I$ is an action on a countable set $I$. Let $\varepsilon\colon G\rightarrow B$ be the quotient homomorphism  and  $(u_g)_{g\in G}$ the canonical generating unitaries of $\emph{L}(G)$.  Let  $G\curvearrowright I$ and $G\curvearrowright^{\sigma} \emph{L}(A^{(I)})=\emph{L}(A)^I$ be  the action  and  the trace preserving action given by $g\cdot i=\varepsilon(g) i$ and $\sigma_g=\emph{Ad}(u_g)$, for every $g\in G$ and $i\in I$.
Then $\sigma$ is built over $G\curvearrowright I$. %Moreover, let $J\subset I$ be a set which meets each $G$-orbit exactly once. For $i\in J$, consider the trace preserving action  $\emph{Stab}_G(i)\curvearrowright^{\rho_i} \emph{L}(A_i)$ given by $(\rho_i)_g(u_h)=u_{ghg^{-1}}$, for every $g\in\emph{Stab}_G(i)$ and $h\in A_i$. Let $\sigma_i$ be the action of $G$ obtained by co-inducing $\rho_i$. Then $\sigma$ is conjugate to $\otimes_{i\in J}\sigma_i$.
 \end{lem}

\begin{proof}
Since $\sigma_g(\text{L}(A_i))=\text{L}(A_{g\cdot i})$, for every $g\in G$ and $i\in I$, $\sigma$  is built over $G\curvearrowright I$.
%If $i\in J$ and $g\in\text{Stab}_G(i)$, the restriction of $\sigma_g$ to $\text{L}(A_i)$ is $(\rho_i)_g$, and the conclusion follows.
\end{proof}

On the other hand, as recorded in \cite[Theorem 3.6]{CIOS1}, results from \cite{Dr15,VV14} imply that actions built over satisfy Popa's cocycle superrigidity theorems \cite{Po01a,Po05}:

\begin{thm}[\cite{CIOS1}]\label{builtover} Let $G$ be a countable group with property (T), $G\curvearrowright I$ be an action on a countable set $I$ with infinite orbits and $(\P,\tau)$ be a tracial von Neumann algebra. Let $G\curvearrowright^{\sigma}(\P^I,\tau)$ be a trace preserving action built over $G\curvearrowright I$. Then the following hold.
\begin{enumerate}
\item[(a)] Any $1$-cocycle for $\sigma$ is cohomologous to a character of $G$. More generally, given a trace preserving action $G\curvearrowright^{\lambda}(\Q,\tau)$, any $1$-cocycle $w\colon G\rightarrow\sU(\P^I\,\overline{\otimes}\,\Q)$ for the product action $\sigma\otimes\lambda$ is cohomologous to a $1$-cocycle taking values into $\sU(\Q)\subset \sU(\P^I\,\overline{\otimes}\,\Q)$.

\item[(b)] Any generalized $1$-cocycle for $\sigma$ has support projection $1$.
\end{enumerate}
\end{thm}

Lemma \ref{coind} and Theorem \ref{builtover} imply the following corollary which will be used in the proof of Theorem \ref{symmetries'}.

\begin{cor}\label{CS}%[\cite{Po05,Dr15}]
In the setting of Lemma \ref{coind}, assume additionally that $G$ has property (T) and the action $B\curvearrowright I$ has infinite orbits.    Then any $1$-cocycle for $\sigma$ is cohomologous to a character of $G$ and
any generalized $1$-cocycle for $\sigma$ has support projection $1$.

\end{cor}

We end this section with two results on wreath-like product groups with abelian base. These use \cite[Remark 3.5]{CIOS1}
which we recall here for the reader's convenience.

\begin{rem}[\cite{CIOS1}]\label{bernoulli} Let $A,B$ be countable groups and  $G\in\mathcal W\mathcal R(A,B\curvearrowright I)$, for some action $B\curvearrowright I$. Assume that $A$ is abelian.
 Then the conjugation action of $G$ on $A^{(I)}$ descends to an action of $B=G/A^{(I)}$ on $A^{(I)}$.
 We denote by $B\curvearrowright^{\alpha}\text{L}(A^{(I)})$ the associated trace preserving action. Explicitly, for $g\in B$, we have $\alpha_g=\text{Ad}(u_{\widehat{g}})$, where $\widehat{g}\in G$ is any element such that $\varepsilon(\widehat{g})=g$.
Lemma \ref{coind} implies that $\alpha$ is built over $B\curvearrowright I$.
Moreover, assume that $I=B$ endowed with the left multiplication action of $B$, so that the action $B\curvearrowright I$ is free and transitive. Then the discussion after \cite[Definition 2.5]{KV15} implies that   $\alpha$ is conjugate to the  Bernoulli action $B\curvearrowright\text{L}(A)^B$.
%Moreover, $\alpha$ is conjugate to $\otimes_{i\in J}\alpha_i$, where $\alpha_i$ is obtained by co-inducing the action $\text{Stab}_B(i)\curvearrowright^{\tau_i}\text{L}(A_i)$ given by $(\tau_i)_g=(\rho_i)_{\widehat{g}}$, for every $g\in\text{Stab}_B(i)$. In particular, if $I=B$ endowed with the left multiplication action of $B$, then  $\sigma$ and $\alpha$ are conjugate to the generalized Bernoulli actions $G\curvearrowright \text{L}(A)^B$ and $B\curvearrowright\text{L}(A)^B$, respectively.

For the rest of this subsection, we also denote by $\alpha$ the probability measure preserving (pmp) action $B\curvearrowright^\alpha(\widehat{A}^I,\nu^I)$, obtained by identifying $\text{L}(A^{(I)})$ with  $\text{L}^{\infty}(\widehat{A}^I,\nu^I)$,
where $\widehat{A}$ is the dual of $A$ endowed with its Haar measure $\nu$.

\end{rem}

 Theorem \ref{builtover} and Remark \ref{bernoulli}  yield an extension of Popa's OE superrigidity theorem for Bernoulli actions of property (T) groups \cite{Po05} that is needed to prove Theorem \ref{symmetries}.
Given a pmp action $G\curvearrowright (X,\mu)$, we
denote by $\mathcal R(G\curvearrowright X)=\{(x,y)\in X\times X\mid G\cdot x=G\cdot y\}$ its {\it orbit equivalence (OE) relation}. The full group $[\mathcal R(G\curvearrowright X)]$ consists of  automorphisms $\theta$ of $(X,\mu)$ such that $\theta(x)\in G\cdot x$, for almost every $x\in X$.
The action $G\curvearrowright (X,\mu)$ is called {\it OE superrigid} if for every free p.m.p. action $H\curvearrowright (X,\mu)$ such that $G\cdot x= H\cdot x$, for almost every $x\in X$, we can find $\theta\in [\sR( G\curvearrowright X)]$ such that $\theta G \theta^{-1}=H$. Here, we view $G$ and $H$ as subgroups of $\text{Aut}(X,\mu)$, the group of automorphisms of $(X,\mu)$.

\begin{lem}\label{OES}
Let $A$ be an abelian group, $B$ an ICC property (T) group and $B\curvearrowright I$ an action with infinite orbits. % and $G\in\mathcal W\mathcal R(A,B\curvearrowright I)$. 
Then  $B\curvearrowright^{\alpha} (\widehat{A}^I,\nu^I)$ is OE superrigid.
\end{lem}

\begin{proof} By Remark \ref{bernoulli}, the associated trace preserving action $B\curvearrowright^{\alpha}\text{L}^{\infty}(\widehat{A}^I,\nu^I)\equiv\text{L}(A^{(I)})$ is built over the action $B\curvearrowright I$. Since $B$ has property (T) and $B\curvearrowright I$ has infinite orbits, Theorem \ref{builtover} (a) implies that any cocycle for $\alpha$ with values into $\sU(\Q)$, where $(\Q,\tau)$ is a tracial von Neumann algebra, is cohomologous to a homomorphism $B\rightarrow\sU(\Q)$. Since $\alpha$ is weakly mixing, \cite[Proposition 3.5]{Po05} implies that
any cocycle for $\alpha$ with values into a countable group $K$ is cohomologous to a homomorphism $B\rightarrow K$.
Since $B$ has no nontrivial normal subgroups, applying \cite[Proposition 5.11]{Po05} gives that $\alpha$ is OE superrigid.
\end{proof}

Using Remark \ref{bernoulli}, we can also provide a simpler solution to a question of Popa.
Popa asked in \cite[Section 6.6]{Po05} whether $\text{H}^2(\alpha)=\text{H}^2(B)$, for Bernoulli actions $\alpha$ of property (T) groups $B$. Recovering a result from \cite{Ji15}, we show that the equality $\text{H}^2(\alpha)=\text{H}^2(B)$ fails for certain property (T) groups $B$:

\begin{prop}\label{H^2}
Let $A$ be a nontrivial abelian group and $B$ an infinite group such that $\WR(A,B)$ contains a property (T) group. Consider the Bernoulli action $B\curvearrowright^{\alpha} (\widehat{A}^B,\nu^B)$.
Then $\emph{H}^2(\alpha)\not=\emph{H}^2(B)$. Moreover, there exists a $2$-cocycle $w$ for $\alpha$ with values in $\mathbb T$ such that $\emph{L}^\infty(\widehat{A}^B)\rtimes_{\alpha,w}B$ is a \emph{II}$_1$ factor with property (T).
\end{prop}

\begin{proof}
Let $G\in\WR(A,B)$ be a group with property (T) and $\varepsilon:G\rightarrow B$ be the quotient homomorphism.
 For $g\in B$, let $\widehat{g}\in G$ with $\varepsilon(\widehat{g})=g$. By Remark \ref{bernoulli}, $\alpha_g=\text{Ad}(u_{\widehat{g}})$, for every $g\in B$. Then $w:B\times B\rightarrow \mathscr U(\text{L}(A^{(B)}))$ given by  $w_{g_1,g_2}=u_{\widehat{g_1}}u_{\widehat{g_2}}u_{\widehat{g_1g_2}}^*$ is a $2$-cocycle for $\alpha$, i.e., $w_{g_1,g_2}w_{g_1g_2,g_3}=\alpha_{g_1}(w_{g_2,g_3})w_{g_1,g_2g_3}$, for all $g_1,g_2,g_3\in B$. Moreover, we have that \begin{equation}\label{LG}\text{L}(G)\cong \text{L}(A^{(B)})\rtimes_{\alpha,w}B.\end{equation}
%Identify $\text{L}(A^{(B)})$ with  $\text{L}^{\infty}(\widehat{A}^B,\nu^B)$, via the Fourier transform.
Since $G$ has property (T) and is ICC, \eqref{LG} implies that
$\text{L}^\infty(\widehat{A}^B)\rtimes_{\alpha,w}B\equiv\text{L}(A^{(B)})\rtimes_{\alpha,w}B$ is a \text{II}$_1$ factor with property (T).
 Assume by contradiction that $\text{H}^2(\alpha)=\text{H}^2(B)$. Then $w$ would be cohomologous to a $2$-cocycle $v:B\times B\rightarrow\mathbb T$. Hence,  $\text{L}(A^{(B)})\rtimes_{\alpha,w}B\cong  \text{L}(A^{(B)})\rtimes_{\alpha,v}B$.  Let  $c=v\circ (\varepsilon\times\varepsilon):G\times G\rightarrow \mathbb T$. 
 Then $c$ is a $2$-cocycle and we have  $\text{L}(A^{(B)})\rtimes_{\alpha,v}B \cong \text{L}_c(A \;\rm{wr}\; B)$, where $\text{L}_c(A \;\rm{wr}\; B)$ denotes the twisted group von Neumann algebra associated to $A \;\rm{wr}\; B$ and $c$ (see, e.g., \cite[Section 10(IV)]{Io10}).
 By combining the last two isomorphisms with \eqref{LG} we get that $\text{L}(G)\cong\text{L}_c(A \;\rm{wr}\; B)$ has property (T). By \cite[Proposition 5.1]{Po01b} we derive that $A\;\rm{wr}\; B$ has property (T), which is false by \cite[Proposition 2.8.2]{BdHV08}.
  \end{proof}

%%%%%%%%%%%%%%%%%%%%%%%%%%%%%%%%%%%%%%%%%%%%%%%%%%%%%%%%%%%%%%%%%%%%%
%%%%%%%%%%%%%%%%%%%%%%%%%%%%%%%%%%%%%%%%%%%%%%%%%%%%%%%%%%%%%%%%%%%%%
%%%%%%%%%%%%%%%%%%                  CALCULATION OF SYMMETRIES                %%%%%%%%%%%%%%%%%%%%%%%
%%%%%%%%%%%%%%%%%%%%%%%%%%%%%%%%%%%%%%%%%%%%%%%%%%%%%%%%%%%%%%%%%%%%%
%%%%%%%%%%%%%%%%%%%%%%%%%%%%%%%%%%%%%%%%%%%%%%%%%%%%%%%%%%%%%%%%%%%%%

\section{Outer automorphism groups of wreath-like group factors}\label{OUT}
The goal of this section is to prove Theorem \ref{symmetries'}.
 We begin with an informal outline of its proof in the case of regular wreath-like products  $G\in\WR(A,B),H\in\WR(C,D)$, under the additional assumptions that $B, D$ are ICC subgroups of hyperbolic groups and $t=1$. %$p=1$.
Let $\theta:\text{L}(G)\rightarrow\text{L}(H)$ be a $*$-isomorphism. Denote by $(u_g)_{g\in G}\subset \text{L}(G)$ and $(v_h)_{h\in H}\subset \text{L}(H)$  the canonical generating unitaries. Denote $\mathcal M=\text{L}(G)$ and identify $\mathcal M=\text{L}(H)$, via $\theta$.

The proof of Theorem \ref{symmetries'} follows the same approach as that of \cite[Theorem 1.3]{CIOS1}.
We first show that $\mathcal P=\text{L}(A^{(B)})$ and $\mathcal Q=\text{L}(C^{(D)})$ are unitarily conjugate.
Note that $\mathcal P,\mathcal Q\subset\mathcal M$ are regular subalgebras and $G/A^{(B)}\cong B, H/C^{(D)}\cong D$ are subgroups of hyperbolic groups.
The conclusion then follows by using Popa and Vaes' structure theorem \cite{PV12} (applied via Theorem \ref{relativeT}) together with conjugacy results for Cartan subalgebras or regular irreducible subfactors  from \cite{Po01b,IPP05} (see Theorems \ref{Po01b} and \ref{IPP05}), depending on whether $A,C$ are abelian or ICC.
After unitary conjugacy, we may assume that $\mathcal P=\mathcal Q$.

The second step is a ``discretization" argument  to identify $G=(u_g)_{g\in G}$ and $H=(v_h)_{h\in H}$  modulo $\mathscr U(\mathcal P)$. This is immediate if $A, C$ are ICC, since then $\mathcal P\subset\mathcal M$ is a regular irreducible subfactor and computing its normalizer in $\mathcal M$  in two ways gives that $G\mathscr U(\mathcal P)=H\mathscr U(\mathcal P)$. If $A, C$ are abelian, then identifying the equivalence relation of the inclusion $\mathcal P\subset\mathcal M$ in two ways  gives an orbit equivalence between the Bernoulli actions $B\curvearrowright\widehat{A}^B$ and $D\curvearrowright\widehat{C}^D$. As Bernoulli actions of ICC property (T) groups are OE superrigid by a result of Popa \cite{Po04},  we again deduce that $G\mathscr U(\mathcal P)=H\mathscr U(\mathcal P)$, after a unitary conjugation. In either case, we find maps $\zeta:G\rightarrow\mathscr U(\mathcal P)$ and $\delta:G\rightarrow H$ such that $\zeta_gu_g=v_{\delta(g)}$, for every $g\in G$.

Up to this point, we have used that $B$, but not $G$, has property (T). To explain how we use property (T), let $G\curvearrowright^{\sigma}\mathcal P$ be the trace preserving action given by $\sigma_g=\text{Ad}(u_g)$.
Then $\sigma$ is an action built over the action $G\curvearrowright B$ by Lemma \ref{coind}. Moreover, since the action $G\curvearrowright B$
 is transitive, $\sigma$ is a coinduced action.
  %One of main novelties of this paper is the way we use property (T) for $G$. % to complete the proof of Theorem \ref{out}.
Since $G$ has property (T), any $1$-cocycle for $\sigma$ or $\sigma\otimes\sigma$ is cohomologous to a character of $G$ (see Corollary \ref{CS}). This follows from Popa's cocycle superrigidity theorem \cite{Po05}, if $A$ is abelian, and its extension to coinduced actions \cite{Dr15}, in general.
If $\delta:G\rightarrow H$ is a homomorphism, then $(\zeta_g)_{g\in G}$ is a $1$-cocycle for $\sigma$ and we can easily conclude. % and we can  find a unitary $u\in\mathcal P$ and a character $\eta:G\rightarrow \mathbb T$ such that $uu_gu^*=\eta(g)v_{\delta(g)}$, for every $g\in G$.
However, a priori, we only know that $\delta$ ``descends" to a homomorphism $B\rightarrow D$.
To bypass this difficulty we use the comultiplication $\Delta:\mathcal M\rightarrow\mathcal M\,\overline{\otimes}\,\mathcal M$ given by $\Delta(v_h)=v_h\otimes v_h$, $h\in H$ \cite{PV09}. Denoting $\omega_g=\Delta(\zeta_g)^*(\zeta_g\otimes\zeta_g)\in \mathcal P\,\overline{\otimes}\,\mathcal P$ we have $\Delta(u_g)=\omega_g(u_g\otimes u_g)$, for every $g\in G$. Thus, $(\omega_g)_{g\in G}$ is a $1$-cocycle for $\sigma\otimes\sigma$ and so we can  find a unitary $w\in \mathcal P\,\overline{\otimes}\,\mathcal P$ and a character $\rho:G\rightarrow \mathbb T$ such that $w\Delta(u_g)w^*=\rho(g)(u_g\otimes u_g)$, for every $g\in G$. A general result from \cite{IPV10} then implies the conclusion of Theorem \ref{symmetries'}.

\subsection{Criteria for $W^*$-superrigidity for pairs of  groups}
The proof of Theorem \ref{symmetries'} relies on two general W$^*$-superrigidity criteria for pairs of groups (Theorems \ref{criterion1} and \ref{criterion}), which  are of independent interest. The first criterion is also used in an essential way in our forthcoming paper \cite{CIOS3} to construct many new examples W$^*$-superrigid wreath-like product groups with property (T). %We say that a pair of groups $A<G$ is {\it W$^*$-superrigid} if any pair $B<H$ such that $(\text{L}(B)\subset \text{L}(H))$ is isomorphic to $(\text{L}(A)\subset \text{L}(G))$ must be isomorphic to $A<G$.  We also denote by $G\curvearrowright^{\sigma}\text{L}(A)$ the trace preserving action given by $\sigma_g=\text{Ad}(u_g)$, for any $g\in G$.

\begin{thm}\label{criterion1}
Let $G$ be an ICC countable group and $A\lhd G$ a normal ICC subgroup such that $\emph{L}(A)'\cap \emph{L}(G)=\mathbb C1$. Denote by $G\curvearrowright^{\sigma}\text{L}(A)$ the trace preserving action given by $\sigma_g=\emph{Ad}(u_g)$, for every $g\in G$.
Let $p\in \emph{L}(A)$ be a projection. Let $B<H$ be an inclusion of countable groups and $\theta\colon p\emph{L}(G)p\rightarrow \emph{L}(H)$ a $*$-isomorphism such that $\theta(p\emph{L}(A)p)=\emph{L}(B)$. %Then we have
\begin{enumerate}
\item Assume that any generalized $1$-cocycle for $\sigma\otimes\sigma$ has support $1$. Then $p=1$.
\item Assume that $p=1$ and any $1$-cocycle for $\sigma\otimes\sigma$ is cohomologous to a character of $G$.
Then there are a group isomorphism $\delta\colon G\rightarrow H$, a character $\eta\colon G\rightarrow\mathbb T$ and $u\in\sU(\emph{L}(H))$ such that $\delta(A)=B$ and $\theta(u_g)=\eta(g)uv_{\delta(g)}u^*$, for all $g\in G$.

\end{enumerate}

\end{thm}

\begin{proof}
Identify $\text{L}(H)=p\text{L}(G)p$ and $p\text{L}(A)p=\text{L}(B)$, via $\theta$.
Let $\Delta\colon \text{L}(H)\rightarrow \text{L}(H)\,\overline{\otimes}\,\text{L}(H)$ be the $*$-homomorphism given by $\Delta(v_h)=v_h\otimes v_h$, for all $h\in H$.
Since $\Delta(\text{L}(B))\subset\text{L}(B)\,\overline{\otimes}\,\text{L}(B)$ and $\text{L}(A)$ is a factor, we can extend $\Delta$ to a $*$-homomorphism  $\Delta\colon \text{L}(G)\rightarrow\text{L}(G)\,\overline{\otimes}\,\text{L}(G)$ such that $\Delta(1)=1\otimes p$ and $\Delta(\text{L}(A))\subset\text{L}(A)\,\overline{\otimes}\,\text{L}(A)$. The proof relies on the following claim.

\begin{claim}\label{1-cocycle} There are $(w_g)_{g\in G}\subset \emph{L}(A)\,\overline{\otimes}\,\emph{L}(A)$ such that $\Delta(u_g)=w_g(u_g\otimes u_g)$, for all $g\in G$.\end{claim}

\begin{proof}[Proof of Claim \ref{1-cocycle}]
Assume first that $p=1$. Let $g\in G$. Since $u_g\in \sN_{\text{L}(H)}(\text{L}(B))$ and $\text{L}(B)'\cap \text{L}(H)=\mathbb C1$, \cite[Corollary 5.3(ii)]{SWW07} gives  $\omega_g\in\sU(\text{L}(B))$ and $h\in H$ such that $u_g=\omega_gv_h$.
Since $\Delta(v_h)=v_h\otimes v_h$,  $w_g=\Delta(\omega_g)(\omega_g^*\otimes \omega_g^*)\in\text{L}(A)\,\overline{\otimes}\,\text{L}(A)$ satisfies the claim.

To prove the claim in general, let $g\in G$. Let $\zeta_g\in \text{L}(A)$ be a partial isometry such that $\zeta_g\zeta_g^*=p, \zeta_g^*\zeta_g=\sigma_g(p)$. Then we have $\zeta_gu_g\in\sN_{p\text{L}(G)p}(p\text{L}(A)p)=\sN_{\text{L}(H)}(\text{L}(B))$.
Since $\text{L}(B)'\cap \text{L}(H)=\mathbb C1$, \cite[Corollary 5.3(ii)]{SWW07} gives $\xi_g\in\sU(\text{L}(B))$ and $h\in H$ such that $\zeta_gu_g=\xi_gv_h$.
Replacing $\zeta_g$ with $\xi_g^*\zeta_g$, we may assume that $\zeta_gu_g=v_h$.
Since $\Delta(v_h)=v_h\otimes v_h$, we get that
$\Delta(\zeta_g)\Delta(u_g)=(\zeta_g\otimes\zeta_g)(u_g\otimes u_g)$ and thus $\Delta(\sigma_g(p)u_g)=\rho_g(u_g\otimes u_g)$, where $\rho_g=\Delta(\zeta_g)^*(\zeta_g\otimes\zeta_g)\in\text{L}(A)\,\overline{\otimes}\,\text{L}(A)$.
Since $\text{L}(A)$ is a factor, we can find partial isometries $v_1,\cdots, v_n\in\text{L}(A)$ such that $\sum_{i=1}^nv_iv_i^*=1$ and $v_i^*v_i\leq \sigma_g(p)$, for every $1\leq i\leq n$.
Then \begin{align*}\Delta(u_g)=\sum_{i=1}^n\Delta(v_i\sigma_g(p)v_i^*u_g)&=\sum_{i=1}^n\Delta(v_i)\Delta(\sigma_g(p)u_g)\Delta(\sigma_{g^{-1}}(v_i^*))\\&=\sum_{i=1}^n\Delta(v_i)\zeta_g(u_g\otimes u_g)\Delta(\sigma_{g^{-1}}(v_i^*))\end{align*}
and thus
$w_g=\sum_{i=1}^n\Delta(v_i)\zeta_g(\sigma_g\otimes\sigma_g)(\Delta(\sigma_{g^{-1}}(v_i^*)))\in\text{L}(A)\,\overline{\otimes}\,\text{L}(A)$ satisfies the claim.
\end{proof}

First, Claim \ref{1-cocycle} implies that $(w_g)_{g\in G}$ is a generalized $1$-cocycle for $\sigma\otimes\sigma$ with support projection $\Delta(1)=1\otimes p$.
This observation implies part 1.
Second, under the assumptions of part 2,
$(w_g)_{g\in G}$ is a $1$-cocycle for $\sigma\otimes\sigma$, so there are $z\in\sU(\text{L}(A)\,\overline{\otimes}\,\text{L}(A))$ and a character $\psi\colon G\rightarrow\mathbb T$ such that $w_g=\psi(g)z(\sigma_g\otimes \sigma_g)(z)^*$, for all $g\in G$. Hence, $\Delta(u_g)=\psi(g) z(u_g\otimes u_g)z^*$, for all $g\in G$. 

Finally, applying \cite[Lemma 3.4]{IPV10} implies that there are an injective group homomorphism $\delta\colon G\rightarrow H$, a character $\eta\colon G\rightarrow\mathbb T$ and $u\in\sU(\text{L}(H))$ such that  $u_g=\eta(g)uv_{\delta(g)}u^*$, for all $g\in G$.
 Thus, $\text{L}(G)=u\text{L}(\delta(G))u^*$. Since $\text{L}(G)=\text{L}(H)$, this forces that $\delta(G)=H$, and hence $\delta$ is an isomorphism. Also, denoting $C=\delta(A)<H$, we get that $u\text{L}(C)u^*=\text{L}(A)=\text{L}(B)$. Thus, $B,C\lhd H$ are normal subgroups such that $u\text{L}(C)u^*=\text{L}(B)$ and $\text{L}(B)'\cap \text{L}(H)=\mathbb C1$. Lemma \ref{equality} below gives that $B=C$, hence $\delta(A)=B$. 
 \end{proof}

\begin{lem}\label{equality}
Let $H$ be a countable group and $B,C\lhd H$ be normal subgroups such that  $u\emph{L}(B)u^*=\emph{L}(C)$, for some $u\in\mathcal U(\emph{L}(H))$.
Assume that $\emph{L}(B)'\cap \emph{L}(H)=\mathbb C1$ or $\emph{L}(B)\subset\emph{L}(H)$ is a Cartan subalgebra.
 Then $B=C$.

\end{lem}

\begin{proof} %Denote by $(u_g)_{g\in G}\subset\mathcal U(\text{L}(G))$ the canonical unitaries.
We claim that $D:=B\cap C$ has finite index in $B$.
Indeed, since $\text{L}(B)\prec_{\text{L}(H)}\text{L}(C)$, \cite[Lemma 2.2]{CI17} implies that there is $h\in H$ such that $[B:B\cap hCh^{-1}]<\infty$. Since $C$ is normal in $H$, the claim follows.
%Otherwise, if $[B:D]=\infty$,  then we can we find a sequence $(g_n)_{n\geq 1}\subset B$ with $g_nD\not=g_mD$, for every $n\not=m$. Then $g_n C\not=g_mC$, for every $n\not=m$. Let us show that
%\begin{equation}\label{mixing_sequence} \text{$\|\text{E}_{\text{L}(C)}(xu_{g_n}y)\|_2\rightarrow 0$, for every $x,y\in\text{L}(H)$.} \end{equation} In order to prove \eqref{mixing_sequence}, we may assume that $x=u_h,y=u_k$, for $h,k\in H$. Then $\|\text{E}_{\text{L}(C)}(xu_{g_n}y)\|_2={\bf 1}_{h^{-1}Ck^{-1}}(g_n)={\bf 1}_{h^{-1}k^{-1}C}(g_n)$. Since ${\bf 1}_{h^{-1}k^{-1}C}(g_n)\not=0$ for at most one $n$, we derive that $\|\text{E}_{\text{L}(C)}(xu_{g_n}y)\|_2\rightarrow 0$, which proves \eqref{mixing_sequence}. Since $uu_{g_n}u^*\in\text{L}(C)$, we get that $\|\text{E}_{\text{L}(C)}(uu_{g_n}u^*)\|_2=1$, for every $n$, contradicting \eqref{mixing_sequence}. This proves our claim.

Second, we claim that $u\in\mathcal N_{\text{L}(H)}(\text{L}(B))$. Note that $u^*\text{L}(D)u\subset u^*\text{L}(C)u=\text{L}(B)$ and let $\theta:\text{L}(D)\rightarrow \text{L}(B)$ be the $*$-homomorphism given by $\theta(x)=u^*xu$. Then $u\theta(x)=xu$, for every $x\in\text{L}(D)$. Let $S\subset H$ be a complete set of representatives of $B$-cosets, i.e., such that $H=\sqcup_{g\in S}Bg$. For $g\in S$, let $c_g=\text{E}_{\text{L}(B)}(uu_g^*)$. Then $u=\sum_{g\in S}c_gu_g$ and  by applying $\text{E}_{\text{L}(B)}$ to the equality $u\theta(x)u_g^*=xuu_g^*$ we get that \begin{equation}\label{components}\text{$c_gu_g\theta(x)=xc_gu_g$, for every $g\in S$ and $x\in\text{L}(D)$.} \end{equation}

Next, we consider two cases. Assume first that $\text{L}(B)'\cap \text{L}(H)=\mathbb C1$. Then the set $\{bgb^{-1}\mid b\in B\}$ is infinite, for every $g\in G\setminus\{e\}$. Since $D<B$ has finite index, we get that $\text{L}(D)'\cap \text{L}(H)=\mathbb C1$.
Since \eqref{components} gives that $(c_gu_g)(c_hu_h)^*\in\text{L}(D)'\cap \text{L}(H)$ we conclude that $(c_gu_g)(c_hu_h)^*\in\mathbb C1$, for every $g,h\in S$. This implies that $c_gc_g^*\in\mathbb C1$, for every $g\in S$. Let $g\in S$ with $c_g\not=0$. Let us show that $c_h=0$, for every $h\in S\setminus\{g\}$. 
Indeed, if $c_h\not=0$, then $c_g$ and $c_h$ are invertible which implies that $(c_gu_g)(c_hu_h)^*\not=0$ and contradicts that $(c_gu_g)(c_hu_h)^*\in\mathbb C1$.
Therefore, $u=c_gu_g$ and $c_g\in\mathcal U(\text{L}(B))$, and so $u\in\mathcal N_{\text{L}(H)}(\text{L}(B))$.

Now, assume that $\text{L}(B)\subset\text{L}(H)$ is a Cartan subalgebra. Then $\text{L}(B)'\cap\text{L}(H)=\text{L}(B)$, hence $\{bgb^{-1}\mid b\in B\}$ is an infinite set, for every $g\in H\setminus B$. Since $D<B$ has finite index, we get that $\text{L}(D)'\cap\text{L}(H)\subset\text{L}(B)$. Since $(c_gu_g)(c_hu_h)^*\in\text{L}(D)'\cap \text{L}(H)$  by \eqref{components}, we conclude that $(c_gu_g)(c_hu_h)^*\in\text{L}(B)$, for every $g,h\in S$. Thus, $(c_gu_g)(c_hu_h)^*=0$,  for every $g,h\in S$ with $g\not=h$.  Using this and that $B$ is abelian, for every $b\in\text{L}(B)$ we have that 
\begin{align*}ubu^*=\sum_{g,h\in S}(c_gu_g)b(c_hu_h)^* &=\sum_{g,h\in S}(c_gu_g)(c_hu_h)^*\text{Ad}(u_{h^{-1}})(b)\\&=\sum_{g\in S}c_gc_g^*\text{Ad}(u_g)(b)\in\text{L}(B).\end{align*}

%As in the first claim, we derive that $C$ also has finite index in $B$. Since $\text{L}(B)\subset\text{L}(G)$ is also a Cartan subalgebra, we get that $\text{L}(C)'\cap\text{L}(G)\subset\text{L}(B)$. Hence $\theta(\text{L}(C))'\cap \text{L}(G)=u^*(\text{L}(C)'\cap \text{L}(G))u\subset u^*\text{L}(B)u=\text{L}(A)$.

In either case, we have shown that $u\in\mathcal N_{\text{L}(H)}(\text{L}(B))$. Hence, $\text{L}(C)=u\text{L}(B)u^*=\text{L}(B)$ and thus $B=C$.
\end{proof}

\begin{thm}\label{criterion}
Let $G$ be an ICC countable group and $A\lhd G$ a normal abelian subgroup such that $\emph{L}(A)\subset\emph{L}(G)$ is a Cartan subalgebra.
Consider the pmp action $G/A\curvearrowright^{\alpha}(\widehat{A},\mu)$ given by the conjugation action of $G$ on $A$, where $\mu$ is the Haar measure of $\widehat{A}$.
Let $p\in \emph{L}(A)$ be a projection. Let $B<H$ be an inclusion of countable groups and $\theta\colon p\emph{L}(G)p\rightarrow \emph{L}(H)$ a $*$-isomorphism such that $\theta(\emph{L}(A)p)=\emph{L}(B)$.
\begin{enumerate}
\item Assume that $\sR(G/A\curvearrowright \widehat{A})\cap (Y\times Y)$ is not the OE relation of any free pmp action of a countable group for any measurable set $Y\subset \widehat{A}$ with $0<\mu(Y)<1$. Then $p=1$.

\item
Assume that $p=1$, any $1$-cocycle for $\sigma\otimes\sigma$ is cohomologous to a character  of $G$ and $\alpha$ is OE superrigid. Then there are a group isomorphism $\delta\colon G\rightarrow H$, a character $\eta\colon G\rightarrow\mathbb T$ and $u\in\sU(\emph{L}(H))$ with $\delta(A)=B$ and $\theta(u_g)=\eta(g)uv_{\delta(g)}u^*$ for all $g\in G$.
\end{enumerate}
\end{thm}

\begin{proof} Identify $\text{L}(H)=p\text{L}(G)p$ and $\text{L}(A)p=\text{L}(B)$, via $\theta$.
Let $\nu$ be the Haar measure of $\widehat{B}$ and  $H/B\curvearrowright^{\beta} (\widehat{B},\nu)$ be the pmp action induced by the conjugation action of $H$ on $B$. Identify $\text{L}(A)=\text{L}^{\infty}(\widehat{A})$ and $\text{L}(B)=\text{L}^{\infty}(\widehat{B})$.
Since $\text{L}(A)\subset\text{L}(G)$ and $\text{L}(B)\subset\text{L}(H)$ are Cartan subalgebras, the actions $\alpha$ and $\beta$ are free.
Moreover, we can find $\mathbb T$-valued $2$-cocycles $v,w$  for $\alpha,\beta$, respectively, such that $\text{L}(G)\cong\text{L}^\infty(\widehat{A})\rtimes_{\alpha,v}G/A$ and $\text{L}(H)\cong\text{L}^\infty(\widehat{B})\rtimes_{\beta,w} H/B$.

Let $Y\subset\widehat{A}$ be a measurable set with $p={\bf 1}_Y$.
Using the identification $\text{L}(A)p=\text{L}(B)$, we identify the probability spaces  $(Y,\mu(Y)^{-1}\mu|Y)$ and $(\widehat{B},\nu)$. Then  $\alpha(G/A)\cdot x\cap Y=\beta(H/B)\cdot x$, for almost every $x\in Y$.
 Hence, $\sR(G/A\curvearrowright \widehat{A})\cap (Y\times Y)$ is equal to the OE relation of a free action, $\beta$. This forces $\mu(Y)=1$ and thus $p=1$, which  proves part 1.

To prove part 2, assume that $p=1$, i.e., $Y=\widehat{A}$.
Since the action $G/A\curvearrowright \widehat{A}$ is OE superrigid, there are a group isomorphism $\varepsilon\colon G/A\rightarrow H/B$ and $\rho\in [\sR(H/B\curvearrowright \widehat{B})]$ such that $\rho\circ\alpha_g=\beta_{\varepsilon(g)}\circ\rho$, for all $g\in G/A$. Let $u\in\sN_{\text{L}(H)}(\text{L}(B))$ such that $ufu^*=f\circ\rho^{-1}$, for all $f\in L^{\infty}(\widehat{B})=\text{L}(B)$.
For $g\in G/A$ and $h\in H/B$ fix $\widehat{g}\in gA$ and $\widehat{h}\in hB$.
Then $\alpha_g=\text{Ad}(u_{\widehat{g}})$ and $\beta_h=\text{Ad}(v_{\widehat{h}})$.
Hence, we can find $(\zeta_g)_{g\in  G/A}\subset\sU(\text{L}(B))$ such that $uu_{\widehat{g}}u^*=\zeta_gv_{\widehat{\varepsilon(g)}}$, for every $g\in G/A$. After replacing $\theta$ by $\text{Ad}(u)\circ\theta$, we  still have $\text{L}(A)=\text{L}(B)$ and $u_{\widehat{g}}=\zeta_gv_{\widehat{\varepsilon(g)}}$, for every $g\in G/A$.
Finally, let $g\in G$. Then $a=g\widehat{gA}^{-1}\in A$. Let $\omega_g=u_a\zeta_{\widehat{gA}}\in \sU(\text{L}(B))$ and $h=\widehat{\varepsilon(gA)}\in H$. Thus, for every $g\in G$, we found $h\in H$ and $\omega_g\in\sU(\text{L}(B))$ such that $u_g=\omega_gv_h$. Repeating the proof of part 2 of Theorem \ref{criterion1} and using Lemma \ref{equality} (in the abelian case) gives the conclusion. \end{proof}

%%%%%%%%%%%%%%%%%%%%%%%%%%%%%%%%%%%%%%%%%%%%%%%%%%%%%%%%%%%%%%%%%%%%%%%%%%%%%%%%%%%%%%%%%%%%%%%%%%%%%%%%%%%%%%%

\subsection{Computations of outer automorphism groups}\label{compute}

%%%%%%%%%%%%%%%%%%%%%%%%%%%%%%%%%%%%%%%%%%%%%%%%%%%%%%%%%%%%%%%%%%%%%%%%%%%%%%%%%%%%%%%%%%%%%%%%%%%%%%%%%%%%%%%
In this subsection we prove Theorems \ref{symmetries'} and \ref{Out}.

\begin{thm}\label{symmetries}
Let $A,C$ be non-trivial groups that are either abelian or ICC.
Let $B,D$ be ICC subgroups of hyperbolic groups, or, more generally, nonparabolic ICC subgroups of groups which are hyperbolic relative to a finite family of finitely generated, residually finite groups.

 Let $G\in\mathcal W\mathcal R(A,B \ca I)$ and $H\in\mathcal W\mathcal R(C,D\ca J)$ be property (T) groups, where $B \ca I$ and $D\ca J$ are faithful actions with infinite orbits.

Let $p\in {\rm L}(G)$ be a projection and let  $\theta:p{\rm L}(G)p\rightarrow {\rm L}(H)$ be a $*$-isomorphism.  Then $p=1$ and there are a group isomorphism $\delta:G\rightarrow H$, a character $\rho:G\rightarrow\mathbb T$ and $u\in\sU({\rm L}(H))$ such that $\delta(A^{(I)})=C^{(J)}$ and $\theta(u_g)=\rho(g)uv_{\delta(g)}u^*$, for every $g\in G$.
\end{thm}

\begin{proof} %To simplify the writing, we introduce the following notation.
We denote $\M={\rm L}(G), \P_0={\rm L}(A), \P={\rm L}(A^{(I)})=\P_0^{I}, \Nn={\rm L}(H), \Q_0={\rm L}(C)$ and $\Q={\rm L}(C^{(J)})=\Q_0^{J}$.  Since $\M$ is a factor we can assume $p\in \P$.
Since $H$ has property (T), so does $D$.
By applying Lemma \ref{exact} to $D$, we find a short exact sequence $$1\ra  S \ra D \overset{\varepsilon}{\ra} K\ra 1,$$ where $ S$ is either trivial or a nontrivial free product, $S= S_1\ast  S_2$ with $|S_1|\geq 2$ and $|S_2|\geq 3$, and $K$ is a non-elementary subgroup of a hyperbolic group. Let $\pi=\varepsilon\circ\eta:H\rightarrow K$, where $\eta:H\rightarrow D$ is the quotient homomorphism.
Denote $T=\ker(\pi)$. Then $\eta(T)=\ker{\varepsilon}= S$, $D^{(J)}<T<H$, and $T<H$ has infinite index.

Next we establish the following:
\begin{claim}\label{kerpi}
$\theta(p\P p)\prec_{\Nn}{\rm L}(T)$.
\end{claim}

\noindent\emph{Proof of Claim \ref{kerpi}}. We identify $\mathcal N$ with $p\mathcal Mp$, via $\theta$. Since $T<H$ has infinite index, $\Nn\nprec_{\Nn}{\rm L}(T)$.
 Assume first that $A$ is abelian. As
$p\P p\subset \Nn$ is abelian and regular, $\Nn$ has property (T) and $\Nn\nprec_{\Nn}{\rm L}(T)$,  Theorem \ref{relativeT} (a) applied to the trivial crossed product decomposition $\Nn={\rm L}(H)=\mathbb C1\rtimes H$ and the homomorphism $\pi:H\rightarrow K$
gives the conclusion.

Now assume  $A$ is ICC.  In this case, we can suppose that $p\in \P^j_0$ for some $j\in I$.  For $F\subset I$, let $\R_F=p\P_0^Fp $. As  $p\P p=\R_F\,\overline{\otimes}\,\R_{F^c}$, Theorem \ref{relativeT} (b) implies that either $\R_F\prec_{\Nn}{\rm L}(T)$ or $\R_{F^c}\prec_\Nn {\rm L}(T)$. Since $p\P p\subset\sN_{\Nn}(\R_F)''\cap \sN_{\Nn}(\R_{F^c})''$ and $(p\P p)'\cap \Nn=(p\P p)'\cap p\M p=\mathbb C1$, \cite[Lemma 2.4(3)]{DHI16} further implies that either $\R_F\prec_{\Nn}^{\text{s}}{\rm L}(T)\text{ or }\R_{F^c}\prec^{\text{s}}_\Nn {\rm L}(T)$.

Let $\sF$ be the family of all $F\subset I$ such that  $\R_F\prec_{\Nn}^{\text{s}}{\rm L}(T)$. Note that $\sF\neq \emptyset$. Since ${\rm L}(T)$ is regular in $\Nn$, $\sF$ is closed under unions.
First assume that $\{i\}\in\sF$, for every $i\in I$. Let $(i_n)_{n\in\mathbb N}$ be an enumeration of $I$. Then  $I_n=\{i_1,\ldots,i_n\}=\{i_1\}\cup\ldots\cup\{i_n\}\in\sF$, for all $n\in\mathbb N$.
 By combining \cite[Lemmas 2.6(3) and 2.7]{DHI16},
we get that $p\P p= \big({\bigcup_n \R_{I_n}}\big)''$ is amenable relative to $\text{L}(T)$ inside $\Nn$. Theorem \ref{relativeT} (a) further gives that $p\P p\prec_{\Nn}^{\text{s}}{\rm L}(T)$, as desired.
Secondly, suppose that $\{i\}\notin\sF$, for some $i\in I$. The previous paragraph implies that $I\setminus\{i\}\in\sF$, i.e., $\R_{I\setminus\{i\}}\prec_\Nn^{\text s}{\rm L}(T)$.
This implies that $\P_0^{I\setminus\{i\}}\prec_{\M}^{\text s}{\rm L}(T)$. Let $\gamma:G\rightarrow B$ the quotient homomorphism. Since the orbit $B\cdot i$ is infinite, there is $g\in G$ with $\gamma(g)\cdot i\not=i$.
Since $\P_0^{I\setminus\{\gamma(g)\cdot i\}}=u_g\P_0^{I\setminus\{i\}}u_g^*$, we get that
 $\P_0^{I\setminus\{\gamma(g)\cdot i\}}\prec_{\M}^{\text s}{\rm L}(T)$. Since $\mathcal F$ is closed under unions and $I=(I\setminus\{i\})\cup (I\setminus\{\gamma(g)\cdot i\})$, it follows that $\P\prec_{\M}^{\text s}{\rm L}(T)$ and hence $p\P p\prec_{\Nn}^{\text{s}}{\rm L}(T)$.$\hfill\blacksquare$

We continue with the following:

\begin{claim}\label{core}
$\theta(p\P p)\prec_\Nn\Q$.
\end{claim}

Before proving this claim, note that if $D$ is a subgroup of a hyperbolic group, then we can take $ S=\{e\}$, $K=D$ and $\varepsilon=\text{Id}$. Thus, in this case, $T=C^{(J)}$ and Claim \ref{core}  is redundant. Therefore,  from now on, we may assume that $ S$ is a nontrivial free product.

\noindent \emph{Proof of Claim \ref{core}}. Assume first that $A$ is abelian. Then by Lemma \ref{Lem:Centr} we have  $\theta(p\P p)\subset \Nn$ is a Cartan subalgebra. Since $\theta(p\P p)\prec_\Nn {\rm L}(T)$, \cite[Proposition 3.6]{CIK13} gives nonzero projections $e\in \theta(p\P p),q\in {\rm L}(T)$, a masa $\R\subset q{\rm L}(T)q$, a projection $q'\in \R'\cap q\Nn q$ and $u\in\sU(\Nn)$ such that the inclusion $\sN_{q {\rm L}(T)q}(\R)''\subset q {\rm L}(T)q$ has finite index and $u\theta(p\P p)eu^*=\R q'$. Since $\eta(T)= S$ is a nontrivial free product and $\ker(\eta)=C^{(J)}$,  Theorem \ref{relativeT} (c) gives that $\R\prec_{{\rm L}(T)}^{\rm s}{\rm L}(C^{(J)})=\Q$. Thus, $\R q'\prec_{\M}\Q$, which implies that $\theta(p\P p)\prec_\Nn\Q$.

Secondly, assume that $A$ is ICC. Since $\P_0$ is a factor we can assume that $p\in \P_0^j$ for some $j\in I$.  Lemma \ref{Lem:Centr} implies that $\theta(p\P p)\subset \Nn$ is an irreducible subfactor. Since $\theta(p\P p)\prec_\Nn {\rm L}(T)$, we can find nonzero projections $e\in p\P p,q\in {\rm L}(T)$,  a partial isometry $v$ and a $*$-homomorphism $\Psi:\theta(e\P e)\rightarrow q {\rm L}(T)q$ such  that $v^*v=\theta(e)$, $vv^*\leq q$ and $\Psi(x)v=vx$, for every $x\in \theta(e\P e)$. Let $\R:=\Psi(\theta(e\P e))$.
Then $v^*(\R'\cap q\Nn q)v\subset \theta(e\P e)'\cap \theta(e)\Nn \theta(e)=\mathbb C\theta(e)$, and therefore $v^*(\R'\cap q {\rm L}(T)q)v=\mathbb C\theta(e)$.  So, there is a minimal projection $q'\in \R'\cap q{\rm L}(T)q$ with $q'v\not=0$. After replacing $\Psi$ by $\Psi q'$ and $v$ by the partial isometry in the polar decomposition of $q'v$,  we may assume that $\R\subset q {\rm L}(T)q$ is an irreducible subfactor.

Next, we claim that $\sN_{q{\rm L}(T)q}(\R)''\subset q {\rm L}(T)q$ has finite index.
Since $\theta(p\P p)\subset \Nn$ is a regular irreducible subfactor, so is  $\theta(e\P e)\subset \theta(e)\Nn \theta(e)$.
Fix $u\in\sN_{\theta(e)\Nn \theta(e)}(\theta(e\P e))$ and put $\alpha=\text{Ad}(u)\in {Aut}(\theta(e\P e))$. Define $\beta=\Psi\circ\alpha\circ\Psi^{-1}\in {Aut}(\R)$. Then $\beta(y)vuv^*=vuv^*y$, for all $y\in \R$.  If $w=E_{q {\rm L}(T)q}(vuv^*)$, then
$\beta(y)w=wy$, for all $y\in \R$. Since $\R\subset q{\rm L}(T)q$ is an irreducible subfactor, we get that $w^*w,ww^*\in\mathbb Cq$, which implies that $w\in\mathbb C\sN_{q{\rm L}(T)q}(\R)$. Since $u\in\sN_{\theta(e)\Nn \theta(e)}(\theta(e\P e))$ is arbitrary and $\theta(e\P e)$ is regular in $\theta(e)\Nn \theta(e)$, we get that
\begin{equation}\label{finindex}
E_{q {\rm L}(T)q}(v\Nn v^*)''\subset\sN_{q {\rm L}( T)q}(\R)''.
\end{equation}
By \cite[Lemma 2.3]{CIK13}, we can find a nonzero projection $q_0\in  E_{q {\rm L}(T)q}(v\Nn v^*)''$ such that the inclusion $q_0E_{q {\rm L}(T)q}(v\Nn v^*)''q_0\subset q_0{\rm L}(T)q_0$ has finite index. In combination with \eqref{finindex}, it follows that $q_0\sN_{q{\rm L}(T)q}(\R)''q_0\subset
q_0{\rm L}(T)q_0$ has finite index. Since $\R\subset q\text{L}(T)q$ is an irreducible subfactor, we get that $\sN_{q{\rm L}(T)q}(\R)''$ is a factor, which implies the claim.

Since $\P_0$ is a factor, we may assume that $e\in p\P_0^jp$. For $F\subset I$,  denote $\R_F=\Psi(\theta(e\P_0^Fe))$. Since $\R=\R_F\,\overline{\otimes}\,\R_{F^c}$, $\sN_{q {\rm L}(\Omega)q}(\R)''\subset q {\rm L}(T)q$ has finite index,  $\eta(T)= S$ is a nontrivial free product and $\ker(\eta)=C^{(J)}$,  Theorem \ref{relativeT} (d)
implies that either $\R_F\prec_{{\rm L}(T)}\Q$ or $\R_{F^c}\prec_{{\rm L}(T)}\Q$.  Since $\R\subset\sN_{q {\rm L}(T)q}(\R_F)''\cap \sN_{q {\rm L}(T)q}(\R_F^c)''$ and $\R'\cap q {\rm L}(T)q=\mathbb Cq$, by \cite[Lemma 2.4(3)]{DHI16} these further give either $\R_F\prec_{{\rm L}(T)}^{\text s}\Q$ or $\R_{F^c}\prec^{\text{s}}_{{\rm L}(T)}\Q$, for any $F\subset I$. Let
$\sF$ be the family of all subsets $F\subset I$ such that $\R_F\prec^{\text{s}}_{{\rm L}(\Omega)}\Q$. If $F\in\sF$, since $\theta(e\P_0^Fe)=v^*\R_Fv$, $\theta(p\P p)\subset\sN_{\Nn}(\theta(p\P_0^Fp))''$ and $\theta(p\P p)'\cap \Nn=\mathbb C1$, we get that $\theta(p\P_0^Fp)\prec_{\Nn}^{\text s}\Q$.

We end the proof of Claim \ref{core} by treating two cases.
Assume first that $\{i\}\in\sF$, for every $i\in I$. Let $(i_n)_{n\in\mathbb N}$ be an enumeration of $I$. Since $\Q$ is regular in ${\rm L}(T)$, we get that $\{i_1,\ldots,i_n\}\in\sF$, for every $n\in\mathbb N$.
 By combining \cite[Lemmas 2.6(3) and 2.7]{DHI16},
we get that $\R=\R_{I}$ is amenable relative to $\Q$ inside ${\rm L}(T)$. Since $\sN_{q{\rm L}(T)q}(\R)''\subset q{\rm L}(T)q$ has finite index, by Theorem \ref{relativeT}  (c) we get that $\R\prec_{{\rm L}(T)}^{\text s}\Q$ and so $\theta(p\P p)\prec_\Nn \Q$.
Secondly, suppose $\{i\}\notin\sF$, for some $i\in I$. Then $I\setminus\{i\}\in\sF$, hence $\theta(p\P_0^{I\setminus\{i\}}p)\prec_\Nn^{\text s}\Q$. Using the same argument from the last part of the proof of Claim \ref{kerpi} we get
that $\theta(p\P p)\prec_\Nn^{\text s}\Q$.  This finishes the proof. $\hfill\blacksquare$

\vskip 0.05in
To finish the proof of Theorem \ref{symmetries}, let
 $G\curvearrowright^{\sigma}\P$ be the action given by $\sigma(g)=\text{Ad}(u_g)$, for every $g\in G$. By Corollary \ref{CS} any $1$-cocycle for the diagonal action $G\curvearrowright^{\sigma\otimes\sigma}\P\,\overline{\otimes}\,\P$ is cohomologous to a character of $G$.

Assume first that $A$ and $B$ are nontrivial abelian groups. Then $\theta(\P p),\Q\subset \Nn$ are Cartan subalgebras. Since by Claim \ref{kerpi} we have $\theta(\P p)\prec_\Nn\Q$, Theorem \ref{Po01b} implies that $u\theta(\P p)u^*=\Q$, for some $u\in\sU(\Nn)$. Thus, after replacing $\theta$ by $\text{Ad}\circ\theta$, we may assume that $\theta(\P p)=\Q$.
 Since $B$ has property (T) and is ICC, by Lemma \ref{OES} the action of $B$ on $\widehat{A^{(I)}}$ is OE-superrigid. Using Corollary \ref{CS} and Theorem \ref{criterion} we get the conclusion.

Secondly, assume that $A$ and $B$ are ICC groups. Claim \ref{core} entails that $\theta(p\P p)\prec_\Nn\Q$. Reversing the roles of $G$ and $H$ and arguing similarly implies that  $\theta^{-1}(\Q)\prec_\M \P$ and hence $\Q\prec_{\Nn} \theta (p\P p)$.  By Lemma \ref{Lem:Centr}, the regular subfactors $\theta(p P p) ,\Q\subset \Nn$ are irreducible. Since $\sN_\Nn(\theta(p\P p))/\sU(\theta(p\P p))\cong B,\sN_\Nn(\Q)/\sU(\Q)\cong D$ are ICC groups, Theorem \ref{IPP05} implies that $u\theta(p\P p)u^*=\Q$, for some $u\in\sU(\Nn)$. Then Corollary \ref{CS} and  Theorem \ref{criterion} give the conclusion.\end{proof}

Before proceeding to our  computations of outer automorphisms of property (T) group factors we make some general remarks and establish some notation. Let $G$ be a countable group. The group $Aut(G)$ acts naturally on $Char(G)$ by letting $(\delta \cdot \rho)(g)= \rho(\delta^{-1}(g))$, for all $\rho\in Char (G)$, $\delta\in Aut(G)$, and $g\in G$. Since inner automorphisms act trivially, this action yields an action of $Out(G)$ on $Char(G)$.  Let $Char(G)\rtimes Aut(G)$ and $Char(G)\rtimes Out(G)$ be the corresponding semidirect products. Every pair $(\rho, \delta)\in Char(G)\times Aut(G)$  induces an automorphism $\Psi_{\rho,\delta}\in Aut({\rm L}(G))$  by letting $\Psi_{\rho,\delta}(u_g)= \rho(g)u_{\delta(g)}$ for all $g\in G$. The map  $$\Psi: {Char}(G)\rtimes {Aut} (G)\ni (\rho, \delta) \ra \Psi_{\rho,\delta}\in Aut({\rm L}(G))$$ is a group monomorphism. 
 Since $1\times Inn(G)$ is a normal subgroup of $Char(G)\rtimes Aut(G)$ and $\Psi(1\times Inn(G))=\{e\}$, we can define a homomorphism
 $$\overline{\Psi}\colon {Char}(G)\rtimes {Out} (G)  \ra  Out({\rm L}(G))$$  by letting
$\overline{\Psi}((\rho,\delta Inn (G)))= \Psi_{\rho,\delta}Inn ({\rm L}(G))$
for all $\rho\in Char(G)$ and $\delta\in Aut (G)$. 

With this notation, we have the following.

\begin{prop}\label{inj} If $G$ is ICC, then $\overline{\Psi}$ is injective.

%the map $\Psi$ defined above is a monomorphism.
\end{prop}

\begin{proof} 
Let $\rho\in Char(G)$ and $\delta\in Aut(G)$ with $\overline{\Psi}(\rho,\delta Inn(G))=e$. Then $\Psi_{\rho,\delta}\in Inn(\text{L}(G))$, hence $\Psi_{\rho,\delta}=\text{Ad}(u)$, for a unitary $u\in\text{L}(G)$. Then $\rho(g)u_{\delta(g)}=uu_gu^*$ and thus $\rho(g)u_{\delta(g)}uu_g^*=u$, for every $g\in G$.
Let $u=\sum_{h\in G}c_hu_h$, where $c_h\in\mathbb C$, be the Fourier decomposition of $u$. Using the last identity and identifying Fourier coefficients, we get that
\begin{equation}
\rho(g)c_h=c_{\delta(g)hg^{-1}}, \;\;\; \text{for every $g,h\in G$.}
\end{equation}
Consequently, $|c_h|=|c_{\delta(g)hg^{-1}}|$, for every $g,h\in G$. Let $h\in G$ such that $c_h\not=0$. Since $\sum_{k\in G}|c_k|^2=\|u\|_2^2=1$, the last identity implies that the set $\{\delta(g)hg^{-1}\mid g\in G\}$ is finite. Hence, there is a finite index subgroup $G_h<G$ such that $\delta(g)hg^{-1}=h$, for every $g\in G_h$.

Let $S=\{h\in H\mid c_h\not=0\}$. If $h_1,h_2\in S$, then $\delta(g)=h_1gh_2^{-1}=h_2gh_2^{-1}$, for every $g\in G_{h_1}\cap G_{h_2}$. Hence, $h_2^{-1}h_1$ commutes with the finite index subgroup $G_{h_1}\cap G_{h_2}$ of $G$. Since $G$ is ICC, we get that $h_2^{-1}h_1=e$ and hence $h_2=h_1$. This shows that $S$ is a singleton, i.e. $S=\{h\}$, for some $h\in H$. Therefore, $u=c_hu_h$, where $|c_h|=1$.
Since $\rho(g)u_{\delta(g)}=uu_gu^*$, it follows that $\rho(g)=1$ and $\delta(g)=hgh^{-1}$, for every $g\in G$. This shows that $\rho=1$ and $\delta\in Inn(G)$, which proves that $\overline{\Psi}$ is injective.
\end{proof}

\begin{rem}\label{RemFin} We briefly explain why the ICC assumption on  $G$ is necessary in Proposition \ref{inj}. Specifically, we present examples of non-ICC groups $G$, for which even the homomorphism $\Phi: Out(G)\ra Out({\rm L}(G))$ obtained by restricting $\Psi$ to $Out(G)$ is not injective.

Let $K$ be a finite group together with a non-inner automorphism $\delta_0 \in Aut(K)$ that preserves each conjugacy class of $K$ setwise; the first such examples were found by Burnside, \cite{Bur}.  We claim that the  automorphism $\Phi_{\delta_0}\in Aut (\rm L(K))$ induced by $\delta_0$ is inner. 
Note that $\Phi_{\delta_0}$ fixes the center of $\mathbb C K$,  since this is generated by sums of elements belonging to conjugacy classes. On the other hand, since $K$ is finite, ${\rm L}(K)=\mathbb C K$. By Maschke theorem, the group algebra $\mathbb C K$ is semi-simple. Therefore, any automorphism of $\mathbb C K$ fixing its center is inner by the Skolem-Noether theorem, \cite{Fa08}. %(Alternatively, this follows since $\mathbb CK$ is isomorphic to a direct sum of matrix algebras, being a finite dimensional von Neumann algebra.) 
This shows that $\Phi_{\delta_0}$ is inner and so Proposition \ref{inj} fails for finite groups.

Moreover, let $G= H\times K$, where $H$ is an arbitrary group. Then $\delta=id_H \times \delta_0 \in Aut(G)$ is non-inner. 
However,  the  automorphism $\Phi_\delta\in Aut (\rm L(G))$ induced by $\delta$ is inner. 
Indeed, since ${\rm L}(G)\cong {\rm L}(H)\,\overline \otimes\, {\rm L}(K)$, we can identify $\Phi_\delta=Id\otimes \Phi_{\delta_0}$. 
Thus, Proposition \ref{inj} fails for $G$. In particular, letting $H$ be any nontrivial ICC group, this shows that Proposition \ref{inj} does not even hold for infinite virtually ICC groups. 
\end{rem}

\begin{proof}[Proof of Corollary \ref{Cor:JC}] Let $t\in \mathcal F({\rm L}(G))$ with $t\leq 1$. Thus one can find a projection $p\in {\rm L}(G)$ so that $p{\rm L} (G)p\cong {\rm L}(G)$. Using Theorem \ref{symmetries} for $G=H$ we get $p=1$ an hence $t=1$. This shows that $\mathcal F({\rm L}(G))=\{1\}$. 

Now, let $\theta: {\rm L}(G) \rightarrow {\rm L}(G)$ be a $\ast$-automorphism. Using Theorem \ref{symmetries} for $G=H$ one can find $\delta \in Aut(G)$, $\rho\in Char(G)$ and $v\in\sU({\rm L}(G))$ such that  $\theta= {\rm Ad}(v)\circ \Psi_{\rho,\delta}$. In particular, this shows that the homomorphism $\overline{\Psi}: Char(G)\rtimes Out(G)\ra Out({\rm L}(G))$
defined above in surjective. Since this map is injective by
 Proposition \ref{inj}, it is a group isomorphism   and thus $Char (G)\rtimes Out (G)\cong Out({\rm L}(G))$.         
\end{proof}

\begin{proof}[Proof of Corollary \ref{Out}] Let $Q$ be a countable group. By Theorem \ref{Thm:MC} there is a family $G_i\in \mathcal {WR}(A_i, B_i \ca I_i)$ consisting of $2^{\aleph_0}$ non-isomorphic, property (T) groups, where $A_i$ is abelian, $B_i$ is an ICC, non-parabolic subgroup of a relative hyperbolic group with residually finite peripheral subgroups, and $B_i\ca I_i$ has infinite orbits, for all $i$. Moreover, $G_i$ has no nontrivial characters and $Out(G_i)\cong Q$, for all $i$.

Assume that $\theta\colon \text{L}(G_i)\rightarrow\text{L}(G_j)^t$ is a $*$-isomorphism, for some $i,j\in I$ and $t>0$.  Since $G_i$ has no non-trivial characters, Theorem \ref{symmetries} implies that $t=1$ and there exist a group isomorphism $\delta\colon G_i\rightarrow G_j$,  and  a unitary $u\in\text{L}(G_j)$ such that $\theta(u_g)=uu_{\delta(g)}u^*$, for every $g\in G_i$.  In particular, $i=j$. Thus combining the above with Proposition \ref{inj} we have $Out(\text{L}(G_i))\cong Out(G_i)\cong Q$ and $\mathcal F(\text{L}(G_i))=\{1\}$, which finishes the proof.
\end{proof}

\end{document}